\documentclass[11pt]{amsart}
\usepackage{amscd,amsmath,amssymb,amsfonts,verbatim}
\usepackage[cmtip, all]{xy}
\usepackage{MnSymbol}
\usepackage{comment}

\newtheorem{thm}{Theorem}[section]
\newtheorem{prop}[thm]{Proposition}
\newtheorem{lem}[thm]{Lemma}
\newtheorem{cor}[thm]{Corollary}

\theoremstyle{definition}

\newtheorem{defn}[thm]{Definition}
\theoremstyle{remark}
\newtheorem{remk}[thm]{Remark}
\newtheorem{remks}[thm]{Remarks}

\newtheorem{exm}[thm]{Example}
\newtheorem{exms}[thm]{Examples}
\newtheorem{notat}[thm]{Notation}
\numberwithin{equation}{section}

{\hfill$\square$\end{defn}}
{\hfill$\square$\end{remk}}
{\hfill$\square$\end{remks}}
{\hfill$\square$\end{exm}}
{\hfill$\square$\end{exms}}
{\hfill$\square$\end{notat}}

\newcommand{\thmref}{Theorem~\ref}
\newcommand{\propref}{Proposition~\ref}
\newcommand{\corref}{Corollary~\ref}

\newcommand{\lemref}{Lemma~\ref}

\newcommand{\sA}{{\mathcal A}}
\newcommand{\sB}{{\mathcal B}}
\newcommand{\sC}{{\mathcal C}}

\newcommand{\sF}{{\mathcal F}}
\newcommand{\sG}{{\mathcal G}}
\newcommand{\sH}{{\mathcal H}}
\newcommand{\sI}{{\mathcal I}}

\newcommand{\sK}{{\mathcal K}}

\newcommand{\sM}{{\mathcal M}}

\newcommand{\sO}{{\mathcal O}}

\newcommand{\sV}{{\mathcal V}}
\newcommand{\sW}{{\mathcal W}}

\newcommand{\F}{{\mathbb F}}

\renewcommand{\H}{{\mathbb H}}

\newcommand{\N}{{\mathbb N}}
\renewcommand{\P}{{\mathbb P}}
\newcommand{\Q}{{\mathbb Q}}

\newcommand{\Z}{{\mathbb Z}}

\newcommand{\fm}{{\mathfrak m}}

\newcommand{\ff}{{\mathfrak f}}

\newcommand{\Irr}{{\rm Irr}}
\newcommand{\Ker}{{\rm Ker}}

\newcommand{\CH}{{\rm CH}}

\newcommand{\surj}{\twoheadrightarrow}
\newcommand{\inj}{\hookrightarrow}
\newcommand{\red}{{\rm red}}

\newcommand{\Div}{{\rm Div}}
\newcommand{\Hom}{{\rm Hom}}

\newcommand{\Spec}{{\rm Spec \,}}

\newcommand{\Tr}{{\rm Tr}}

\newcommand{\ab}{\rm ab}

\newcommand{\divf}{{\rm div}}

\newcommand{\supp}{{\rm supp}\,}
\newcommand{\sHom}{{\mathcal{H}{om}}}

\newcommand{\id}{{\operatorname{id}}}

\newcommand{\Sch}{{\operatorname{\mathbf{Sch}}}}

\newcommand{\<}{\langle}
\renewcommand{\>}{\rangle}

\newcommand{\Sm}{{\mathbf{Sm}}}

\newcommand{\Ab}{{\mathbf{Ab}}}

\newcommand{\et}{{\text{\'et}}}

\newcommand{\ds}{{/\kern-3pt/}}

\newcommand{\res}{{\operatorname{res}}}

\renewcommand{\log}{{\operatorname{log}}}

\newcommand{\Br}{{\operatorname{Br}}}

\newcommand{\tr}{{\operatorname{tr}}}

\newcommand{\un}{\underline}
\newcommand{\ov}{\overline}

\newcommand{\tuborg}{\left\{\begin{array}{ll}}
\newcommand{\sluttuborg}{\end{array}\right.}

\newcommand{\zar}{{\rm zar}}
\newcommand{\nis}{{\rm nis}}
\newcommand{\tp}{{\rm top}}

\newcommand{\tor}{{\rm tor}}

\newcommand{\dlog}{{\rm dlog}}

\newcommand{\fr}{{\rm tor}}
\newcommand{\cf}{{\rm cf}}
\newcommand{\pf}{{\rm pf}}
\newcommand{\ev}{{\rm ev}}

\newcommand{\tm}{{\rm t}}

\newcommand{\wt}{\widetilde}
\newcommand{\wh}{\widehat}
\newcommand{\cont}{{\rm cont}}
\newcommand{\Res}{{\rm Res}}

\newcommand{\coker}{{\rm Coker}}
\newcommand{\Fil}{{\rm fil}}
\newcommand{\p}{{\bf{p}}}
\newcommand{\n}{{\un{n}}}

\newcommand{\Tab}{{\mathbf {Tab}}}

\newcounter{elno}

\newcounter{elno-abc}   

\newcounter{elno-abc-prime}

\begin{document}

\title{Class field theory for curves over local fields}
\author{Amalendu Krishna, Subhadip Majumder}
\address{Department of Mathematics, Indian Institute of Science,  
Bangalore, 560012, India.}
\email{amalenduk@iisc.ac.in}
\address{Tata Institute of Fundamental Research, Homi Bhabha
  Road, Mumbai, 400005, India.}
\email{majumder@math.tifr.res.in}

%\thanks{The first author is supported the SFB 1085 Higher Invariants, Universit ̈at Regensburg.}



\maketitle

\begin{quote}\emph{Abstract.}
  We establish a ramified class field theory for smooth projective curves over local
  fields. As key steps in the proof, we obtain new results in the class field theory
  for 2-dimensional local fields of positive characteristic, and
  prove a duality theorem for the logarithmic Hodge-Witt cohomology on
  affine curves over local fields.
\end{quote}
%\end{abstract}
\setcounter{tocdepth}{1}
%\maketitle
\tableofcontents

\section{Introduction}\label{sec:Intro}
The class field theory for smooth projective curves over local fields
is now well understood following the works of Bloch \cite{Bloch-cft},
Saito \cite{Saito-JNT}, Kato-Saito \cite{Kato-Saito-Ann} and Yoshida
\cite{Yoshida03}. The prime-to-characteristic case of the class field theory for
open smooth curves was studied by Hiranouchi
(cf. \cite{Hiranouchi-1}, \cite{Hiranouchi-2}). The objective of this paper is
to extend the class field theory results for smooth projective curves over
local fields to open curves in full generality. We achieve this by proving
new results in the class field theory of 2-local fields of positive charateristics,
and extending the duality theorem of Kato-Saito \cite{Kato-Saito-Ann}
for the p-adic {\'e}tale cohomology of smooth projective curves over local fields
to open curves over such fields. We describe the main results below.

\subsection{Main results}\label{sec:MR}
We fix a local field $k$ of characteristic $p > 0$ and let $X$ be a connected,
smooth and projective curve over $k$. Let $D \subset X$ be an effective divisor on
$X$ with complement $X^o$. The class group with modulus $C(X,D)$ and
the class group $C(X^o)$ were introduced by Hiranouchi
\cite{Hiranouchi-2}{\footnote{We use a slightly different definition of $C(X^o)$
without affecting the reciprocity homomorphism.}}
both of which coincide with the class group $SK_1(X)$, used in \cite{Saito-JNT}
when $D$ is empty. Hiranouchi also defined the abelian {\'e}tale fundamental group
with modulus $\pi^{\ab}_1(X,D)$ and constructed the reciprocity homomorphisms
$\rho_{X^o} \colon C(X^o) \to \pi^{\ab}_1(X^o)$ and $\rho_{X|D} \colon
C(X,D) \to \pi^{\ab}_1(X,D)$. All of these are recalled in \S~\ref{sec:RCurve}.
One can endow $C(X^o)$ and $C(X,D)$ with the structure of
topological abelian groups such that $\rho_{X^o}$ and $\rho_{X|D}$ are continuous.
We let $C(X,D)_0$ denote the kernel of the norm map
$C(X,D) \to k^\times$ (cf. \S~\ref{sec:RCurve}) and define $C(X^o)_0$ similarly. 
We prove the following.

\begin{thm}\label{thm:Main-2}
  For the reciprocity map $\rho_{X|D}$, we have the following.
  \begin{enumerate}
  \item
    $\rho_{X|D} \colon {C(X,D)}/n \to {\pi^{\ab}_1(X,D)}/n$ is a monomorphism
    for every integer $n \ge 1$.
  \item
    $\Ker(\rho_{X|D})$ is the maximal divisible subgroup of $C(X,D)$.
  \item
    The image of the map $\rho_{X|D} \colon C(X,D)_0 \to \pi^{\ab}_1(X,D)$ is not
    necessarily finite.   
  \item
    There exists an integer $r \ge 0$ depending only on $X$ and
    an exact sequence
    \[
      0 \to ({\Q}/{\Z})^r \to {\pi^{\ab}_1(X,D)}^\star \to C(X,D)^\star_\fr \to 0.
    \]
  \end{enumerate}
\end{thm}

\begin{thm}\label{thm:Main-3}
  For the reciprocity map $\rho_{X^o}$, we have the following.
  \begin{enumerate}
  \item
    $\rho_{X^o} \colon {C(X^o)}/n \to {\pi^{\ab}_1(X^o)}/n$ is a monomorphism
    for every integer $n \ge 1$.
  \item
    
    $\Ker(\rho_{X^o})$ is the maximal divisible subgroup of $C(X^o)$.
  \item
    The image of the map $\rho_{X^o} \colon C(X^o)_0 \to \pi^{\ab}_1(X^o)$ is not
    necessarily finite.
    \item
    There is an exact sequence (with $r \ge 0$ as in \thmref{thm:Main-2})
    \[
      0 \to ({\Q}/{\Z})^r \to {\pi^{\ab}_1(X^o)}^\star \to C(X^o)^\star_\fr \to 0.
    \]
 \end{enumerate}
\end{thm}

As a consequence of item (4) of the above results, one obtains that there is a
one-to-one correspondence between the following two sets.
\begin{enumerate}
\item
  Finite abelian covers of $X^o$ (resp. with ramification
  bounded by $D$ away from $X^o$) in which not all closed points of $X^o$
  are completely split.
\item
  Open subgroups of finite index in $C(X^o)$ (resp. $C(X,D)$).
\end{enumerate}

The items (1), (2) and (4) of the above theorems
are extensions of what one knows in the class field
theory for $X$.
On the other hand, the items (3) of the above two results are in contrast
with the case of $X$. The prime-to-$p$ parts of (2)
in Theorems~\ref{thm:Main-2} and ~\ref{thm:Main-3} as well as
the $r = 0$ case of item (4) of \thmref{thm:Main-3}
were shown by Hiranouchi \cite{Hiranouchi-2}.

\vskip .2cm

We let $\wh{C}(X^o) = {\varprojlim}_D C(X,D)$, where the limit is taken over
all effective divisors on $X$ supported on the complement of $X^o$.
One knows that $\rho_{X^o}$ actually factors through
$C(X^o) \to \wh{C}(X^o) \xrightarrow{\wh{\rho}_{X^o}} \pi^{\ab}_1(X^o)$.
Some authors use $\wh{C}(X^o)$ as the definition of the class group of $X^o$.
The next result proves the analogue of \thmref{thm:Main-3} for $\wh{\rho}_{X^o}$.

\begin{thm}\label{thm:Main-5}
  For the reciprocity map $\wh{\rho}_{X^o} \colon \wh{C}(X^o) \to \pi^{\ab}_1(X^o)$,
  we have the following.
  \begin{enumerate}
  \item
    $\wh{\rho}_{X^o} \colon {\wh{C}(X^o)}/n \to {\pi^{\ab}_1(X^o)}/n$ is a monomorphism
    for every integer $n \ge 1$.
  \item
   $\Ker(\wh{\rho}_{X^o})$ is the maximal divisible subgroup of $C(X^o)$.
    \item
    There is an exact sequence (with $r \ge 0$ as in \thmref{thm:Main-2})
    \[
      0 \to ({\Q}/{\Z})^r \to {\pi^{\ab}_1(X^o)}^\star \to \wh{C}(X^o)^\star_\fr \to 0.
    \]
  \end{enumerate}
  \end{thm}

\vskip .2cm

In addition to the class field theory for $X$, the key ingredients in the proofs of
Theorems~\ref{thm:Main-2} and ~\ref{thm:Main-3} are the next two results.
The first is concerning the class field theory of
2-local fields. To state it, we fix some notations. Let $K$ be a Henselian discrete
valuation field of characteristic $p > 0$ whose residue field is a local field.
Let $G^{\ab}_K$ denote the abelianized absolute Galois group of $K$ and let
$K_*(K)$ denote the Milnor $K$-theory of $K$.
Kato \cite{Kato80-1} defined a reciprocity homomorphism
$\rho_K \colon K_2(K) \to G^{\ab}_K$. Unlike the case of local fields,
this map has a non-trivial kernel. However, Kato showed that $K_2(K)$
is endowed with a topology with respect to which $\rho_K$ is continuous
and the dual map $\rho^\vee_K \colon (G^{\ab}_K)^\star \to K_2(K)^\star$ is bijective.
We prove the following result about $\rho_K$.

\begin{thm}\label{thm:Main-1}
  For every integer $n \ge 1$, we have the following.
  \begin{enumerate}
    \item
  $\rho_K \colon {K_2(K)}/n \to {G^{\ab}_K}/n$
  is a monomorphism.
\item
  $\rho_K \colon ~_n K_2(K) \to ~_n G^{\ab}_K$ is an epimorphism.
\end{enumerate}
\end{thm}

An immediate consequence of \thmref{thm:Main-1} is that $\coker(\rho_K)$ is
torsion-free. It also implies that $\Ker(\rho_K)$ is
divisible, a difficult result of Fesenko \cite{Fesenko}. 
On the other hand, $\coker(\rho_K)$ is not divisible (cf. \propref{prop:Coker-p}),
a phenomenon which is in contrast with the class field theory of local fields. 

\vskip .2cm

We now describe the second ingredient, a duality theorem.
We let $m \ge 1$ be an integer and let $W_m\Omega^0_{X|D,\log}$ denote the complex
$[Z_1W_m\Fil_D W_m\sO_X \xrightarrow{1-C} \Fil_D W_m\sO_X]$ on $X_\et$,
where the sheaf $W_m\Fil_D W_m\sO_X$ was introduced in
\cite[\S~3]{Kerz-Saito-ANT} and $C$ is the Cartier operator (cf. \S~\ref{sec:TDR}).
For $i \ge 1$, we let $W_m\Omega^i_{X|D, \log}$ be the logarithmic Hodge-Witt sheaf on
$X_\et$ introduced in \cite[\S~1]{JSZ}. We endow the (hyper)cohomologies of
$W_m\Omega^0_{X|D,\log}$ and $W_m\Omega^i_{X|D, \log}$ with the structure of 
topological abelian groups.
For a topological abelian group $G$, let $G^\pf$ denote the profinite
completion of $G$. We prove the following.

\begin{thm}\label{thm:Main-4}
  For every $i \ge 0$, there is a perfect pairing of topological abelian groups
  \[
    H^i_\et(X, W_m\Omega^2_{X|D, \log})^\pf \times
    H^{2-i}_\et(X, W_m\Omega^0_{X|D,\log}) \to {\Z}/{p^m}.
  \]
  \end{thm}

  This theorem was proven by Kato-Saito \cite{Kato-Saito-Ann} when
  $D = \emptyset$ in which case the profinite completion is redundant.
  We show however that taking the profinite completion is essential in
  \thmref{thm:Main-4} if $D \neq \emptyset$.
  The above duality is compatible with the inclusion of divisors $D \subseteq D'$ so
  that taking the limit and colimit, we get a duality theorem for
  the analogous cohomology groups on $X^o$ (cf. \thmref{thm:Duality-M}).
  In a forthcoming paper \cite{KM-duality}, we shall extend the duality
  theorem to the cohomology groups of $W_m\Omega^1_{X|D, \log}$.

\vskip .3cm

  \subsection{Overview of proofs}\label{sec:Outline}
  The main technical results which allow us to prove Theorems~\ref{thm:Main-2}
  and ~\ref{thm:Main-3} are Theorems~\ref{thm:Main-1} and ~\ref{thm:Main-4}.
  To prove \thmref{thm:Main-1}, we first describe Kato's reciprocity
  homomorphism in terms of a duality map between Hodge-Witt forms.
  We then use the classical local
  Grothendieck duality and a topological duality result for vector spaces over a
  local field equipped with the adic topology of the field to show that this
  duality is perfect. This is a key step which makes things work and 
  allows us to conclude the injectivity of the
  local reciprocity map with finite coefficients. 
  A technical result in ramification theory that we also need in the process is
  \propref{prop:p-tor}. The next three sections of this paper are devoted to the
  proof of \thmref{thm:Main-1}.

  To prove \thmref{thm:Main-4}, the first difficulty one faces is the
  construction of correct topology on the p-adic {\'e}tale cohomology groups
  with modulus on a smooth projective curve over a local field of
  characteristic $p$. The technical steps involved in the definition of the
  topology and proofs of some important properties span through
  Sections~\ref{sec:DThm} and ~\ref{sec:CFT-mod}.
  We prove \thmref{thm:Main-4}  by combining these steps with
  the Kato-Saito duality in the unramified
  case, and a local duality which we deduce from \thmref{thm:Main-1}. 
We finally combine our results about the
  reciprocity for 2-local fields and the ramified duality theorem  to
  complete the proofs of Theorems~\ref{thm:Main-2} and ~\ref{thm:Main-3}
  in Sections~\ref{sec:CFT-mod}, ~\ref{sec:Div-Ker} and ~\ref{sec:Dual-R}.
  In the last section, we also prove a reciprocity theorem for the function field
  of curves over a local field.

\vskip .3cm

\subsection{Notations}\label{sec:Notn}
We shall use the following notations throughout this paper.
Unless a different notation is used in a section, we shall write $\sO_K$ for the
ring of integers of $K$ and $\fm_K$ for the maximal ideal of $\sO_K$ if $K$ is any
Henselian discrete valuation field. A local field will mean a complete discrete
valuation field with finite residue field.

For a Noetherian scheme $X$, we let $\Irr(X)$ be the set of all irreducible
components of $X$. We let $X_{(q)}$ denote the set of $q$-dimensional points of $X$.
For an integral scheme $X$, we let $\Div(X)$ denote the group of Weil divisors on
$X$.
For a scheme $X$ endowed with a topology $\tau \in \{\zar, \nis, \et\}$,
a closed subset $Y \subset X$, and a
$\tau$-sheaf $\sF$, we let $H^*_{\tau, Y}(X, \sF)$ denote the $\tau$-cohomology
groups of $\sF$ with support in $Y$. If $\tau$ is fixed in a subsection, we shall
drop it from the notations of the cohomology groups. We shall write
$H^*_{\tau, Y}(X, \sF)$ as $H^*_{\tau}(X, \sF)$ whenever $Y = X_\red$.
If $X = \Spec(A)$ is affine, we may write $H^*_{\tau}(X, \sF)$ as
$H^*_{\tau}(A, \sF)$. We shall write $H^q_\et(X, {\Q}/{\Z}(q-1))$ 
(resp. $H^q_\et(A, {\Q}/{\Z}(q-1))$) as $H^q(X)$ (resp. $H^q(A)$) for
any $q \ge 0$ (see \cite[\S~3.2, Defn.~1]{Kato80-2}).
We shall let $D_\tau(X)$ denote the bounded derived category of sheaves on $X_\tau$.

If $k$ is a field, we shall denote the category of
separated schemes which are of finite type over $k$ by $\Sch_k$.
We shall let $\Sm_k$ denote the category of smooth schemes over $k$.
If ${k'}/k$ is a field extension and $X \in \Sch_k$, we shall let $X_{k'} =
X \times_{\Spec(k)} \Spec(k')$.
For a Noetherian scheme $X$, we shall let $\pi^{\ab}_1(X)$
denote the abelianized {\'e}tale fundamental group of $X$ and consider it as a
topological group with its profinite topology. We shall write $\pi^{\ab}_1(\Spec(F))$
interchangeably as $G_F$, where the latter is the abelianized absolute Galois
group of $F$. We let $\Tab$ denote the category of topological abelian groups
with continuous homomorphisms.

We shall let $\Ab$ denote the category of abelian groups.
For $A, B \in \Ab$, the notation $A \otimes B$ will indicate
tensor product of $A$ and $B$ over $\Z$.
For an abelian group $A$ and $m\in \N$, we let 
${}_m A$ (resp. $A/m$) denote the kernel (resp. cokernel)
of the multiplication map $A \xrightarrow{m} A$. For a prime $p$,
we shall let $A\{p\}$ (resp. $A\{p'\}$) denote the $p$-primary (resp. prime-to-$p$
primary) component of $A$. 
%For a commutative ring $\Lambda$, we shall write $M \otimes \Lambda$ as $M_\Lambda$.
For a set $\mathbb{M}$ of prime numbers, we shall let $A_{\mathbb{M}}$ denote the 
inverse limit $\varprojlim_m A/{m}$, where the limit runs through all natural
numbers $m$ whose prime divisors lie in $\mathbb{M}$. We let $A_\divf =
{\underset{m \in \N}\bigcap} \ mA$.

\section{Reciprocity for 2-local fields I}\label{sec:RL-1}
Our goal in the next three sections is to prove \thmref{thm:Main-1}.
In this section, we shall prove some basic results concerning 
relative Milnor $K$-theory,
Kato's ramification filtration and the reciprocity homomorphism for 2-local fields.
The main technical result is \propref{prop:p-tor} which we shall use frequently in
our proofs.

\subsection{Relative Milnor $K$-theory}\label{sec:Rel-M}
For a local ring $A$, we let $A^h$ (resp. $A^{sh}$, $\wh{A}$) denote the
Henselization (resp. strict Henselization, completion) of $A$ with respect to its
maximal ideal. If $A$ is an integral domain with fraction field $K$, we let
$K^h$ (resp. $K^{sh}$, $\wh{K}$) denote the fraction field of
$A^h$ (resp. $A^{sh}$, $\wh{A}$).

For a local ring $A$, let $K_n(A)$ denote the $n$-th Milnor $K$-group
of $A$ as defined in \cite{Kato86}. For an ideal $I \subset A$,
we let $K_n(A,I) = \Ker(K_n(A) \surj K_n(A/I))$ denote the relative Milnor
$K$-group.
If $A$ is a local integral domain with fraction field $K$ and $I = (f)$ is a 
nonzero principal ideal, we let $A_f$ denote the localization of $A$ obtained by
inverting the powers of $f$.
We let $K_1(A|I) = K_1(A,I)$, and for $n \ge 2$, we let
$K_n(A|I)$ denote the image of the canonical map
of abelian groups
\begin{equation}\label{eqn:RS-0}
K_1(A,I)  \otimes (A_f)^{\times} \otimes \cdots \otimes (A_f)^{\times} \to 
K_n(K),
\end{equation}
induced by the product in Milnor $K$-theory, where the tensor product is
taken $n$ times (cf. \cite[Definition~2.4 and Lemma~2.1]{Rulling-Saito},
\cite[\S~3]{Gupta-Krishna-Duality}). If $A$ is a discrete valuation ring
with maximal ideal $\fm$,
%there are inclusions $K_n(A,I) \subseteq K_n(A|I) \subseteq K_n(A) \subset K_n(K)$
%(cf. \cite[Lem.~3.1]{Gupta-Krishna-Duality}).
%In this case,
we shall write $K_n(A,\fm^r)$ (resp. $K_r(A|\fm^r)$) as
$\Fil'_r K_n(K)$ (resp. $\Fil_r K_n(K)$) for $n, r \ge 1$.
We shall let $\Fil'_0 K_n(K) = \Fil_0 K_n(K) = K_n(A)$.
One can check using \cite[Lem.~1]{Kato86}
(e.g., see \cite[Lem.~3.1]{Gupta-Krishna-Duality} and its proof) that the
following holds.

\begin{lem}\label{lem:Fil-M-0}
  For every $n, r \ge 1$, one has
  $\Fil'_r K_n(K) \subseteq \Fil_r K_n(K) \subseteq  \Fil'_{r-1} K_n(K) \subseteq
  K_n(A)$,
  and the first inclusion is an equality when $r =1$.
\end{lem}

Let $p > 0$ be a fixed prime. Let $A$ be a discrete valuation ring which is an
$\F_p$-algebra and let $K$ be its quotient field.
Let $m,r \ge 1$ and $n \ge 0$ be two integers.
We let $\kappa^m_r(A|n) =
\frac{\Fil_n K_r(K)}{p^m K_r(K) \cap \Fil_n K_r(K)}$ and define
$\varpi^m_r(A|n)$ so that there is an exact sequence
\begin{equation}\label{eqn:Kato-0}
  0 \to \kappa^m_r(A|n) \to \kappa^m_r(A|0) \to \varpi^m_r(A|n) \to 0.
  \end{equation}

\begin{lem}\label{lem:Kato-1}
    One has
    \begin{enumerate}
    \item
    $\kappa^m_r(A|n) =
    \frac{\Fil_n K_r(K)}{p^m K_r(A) \cap \Fil_n K_r(K)}$.
  \item
    ${\Fil_1 K_r(K)}/{p^m} = \kappa^m_r(A|1) = \kappa^m_r(A|0) =
    {K_r(A)}/{p^m}$ if $r \ge 2$ and the residue field of $A$ is a local field.
\item
$\kappa^m_r(A|n) =
    \frac{\Fil_n K_r(K)}{p^m \Fil_1 K_r(A) \cap \Fil_n K_r(K)}$ under
the hypothesis of (2) if $n \ge 1$.
\end{enumerate}
\end{lem}
\begin{proof}
  To prove (1) and (2), use \lemref{lem:Fil-M-0}, 
\cite[\S~2, Lem.~1]{Kato80-1}, \cite[Thm.~8.1]{Geisser-Levine} and the universal 
exactness of the Gersten complex for $K_*(A)$. The item (3) follows from 
(1) and (2) because the latter implies that  
$p^m K_r(A) \cap \Fil_1 K_r(K) = p^m \Fil_1 K_r(A)$.
  \end{proof}

For a closed immersion of Noetherian schemes $D \subseteq X$ defined by the
sheaf of ideals $\sI_D \subset \sO_X$, we let $\sK_{n, (X,D)}$ be the sheaf
on $X_\zar$ (resp. $X_\nis, \ X_\et$) whose stalk at
$x \in X$ is $K_n(\sO_{X,x}, \sI_{D,x})$ (resp. $K_n(\sO^h_{X,x}, \sI^h_{D,x})$,
$K_n(\sO^{sh}_{X,x}, \sI^{sh}_{D,x})$), where $\sI_{D,x}$ is the stalk of $\sI_D$ at
$x$. We define the sheaves $\sK_{n, X|D}$ in a similar manner if $X$ is integral
and $D$ is locally principal. 
We shall write $\sK_{n, (X,D)}$ and $\sK_{n, X|D}$ as $\sK_{n,X}$ when
$D = \emptyset$. 

For $\tau  \in \{\zar, \nis, \et\}$ and an $\F_p$-scheme $Y$, let 
$\{W_m\Omega^\bullet_Y\}_{m \ge 1}$ be the pro-complex of the $p$-typical
    de Rham-Witt sheaves on $Y_\tau$ with the differential $d$, the Frobenius $F$,
    the Verschiebung $V$, the transition map $R$ and the
    Teichm{\"u}ller map $[-] \colon \sO^\times_Y \to (W_m\sO_Y)^\times$
    (cf. \cite{Illusie}). Let
    $W_m\Omega^r_{Y, \log} = \Ker(W_m\Omega^r_{Y} \xrightarrow{1 -\ov{F}}
    {W_m\Omega^r_Y}/{dV^{m-1}\Omega^{r-1}_Y})$
    (cf. \cite[\S~3.1]{Gupta-Krishna-Duality}). In this paper, we shall write 
$H^i_\et(Y, {\Z}/{p^m}(j)) = H^{i-j}_\et(Y, W_m\Omega^j_{Y, \log})$ and 
$H^i_\et(Y, {\Z}/{n}(j)) = H^i_\et(Y, \mu_n^{\otimes j})$ if $p \nmid n$.   
Recall that there is a canonical map of $\tau$-sheaves
${\rm NR}_Y \colon {\sK_{i, Y}}/n \to \sH^i_\tau(Y, {\Z}/{n}(i))$ which is the
Norm-residue homomorphism
when $p \nmid n$, and the map $\dlog \colon {\sK_{i,Y}}/{p^m} \to
W_m\Omega^i_{Y, \log}$, $\tau$-locally given by $\dlog(\{a_1, \cdots , a_i\}) =
\dlog[a_1]\wedge \cdots \wedge \dlog[a_i]$. The latter an isomorphism
when $Y$ is regular (cf. \cite[Thm.~1.2, Cor.~4.2]{Morrow-ENS}).

Given a regular $\F_p$-scheme $Y$ of Krull dimension one and a divisor
  $D = \sum_y n_y [y] \subset Y$, we let $\sK^m_{r, Y|D}$ be the sheaf on $Y_\tau$
    whose stalks are the groups $\kappa^m_r(\sO_{Y_\tau, y}|n_y)$.
    We let $W_m\Omega^r_{Y|D, \log} = {\rm Image}(\dlog \colon \sK_{r, Y|D} \to
    W_m\Omega^r_{Y})$. 
An easy consequence of Gersten resolutions for Milnor $K$-theory and logarithmic
de Rham-Witt sheaves is that the dlog map induces a natural isomorphism
$\dlog \colon  \sK^m_{r, Y|D} \xrightarrow{\cong} W_m\Omega^r_{Y|D, \log}$
in Nisnevich and {\'e}tale topologies
(cf. \cite[Thm.~1]{JSZ}, \cite[Thm.~2.2.2]{Kerz-Zhao}).
In particular, for $A$ as in ~\eqref{eqn:Kato-0}, we get
\begin{equation}\label{eqn:Milnor-mod}
  \dlog \colon \kappa^m_r(A|n) \xrightarrow{\cong} W_m\Omega^r_{A|n, \log}
  = {\rm Image}(\dlog \colon \Fil_n K_r(K) \to W_m\Omega^r_{A}).
  \end{equation}

We shall use the following result frequently in this paper.

\begin{lem}\label{lem:Kato-equiv}
 If $A$ is as in ~\eqref{eqn:Kato-0} and $n \ge 0, \ m, \ q \ge 1$ are integers, 
 the following hold.
  \begin{enumerate}
 \item
    The canonical maps ${K_2(K)}/{p^m} \to {K_2(\wh{K})}/{p^m}$ and
    $\kappa^m_2(A|n) \to \kappa^m_2(\wh{A}|n)$ are injective.

 \item
   The canonical maps
$\frac{K_2(K)}{\Fil_n K_2(K)} \to \frac{K_2(\wh{K})}{\Fil_n K_2(\wh{K})}$ and
$\varpi^m_r(A|n) \to \varpi^m_r(\wh{A}|n)$
  are bijective.
  \item
  $H^n_\et(R, {\Z}/{q} (n-1)) \to
  H^n_\et(\wh{R}, {\Z}/{q} (n-1))$ is bijective if $A$ is Henselian
  and $R \in \{A, K\}$.
  \item
   The canonical map ${K_2(K)} \to {K_2(\wh{K})}$ is injective if $A$ is Henselian.
  \end{enumerate}
\end{lem}
\begin{proof}
  The item (1) follows because $W_m\Omega^2_K \inj W_m\Omega^2_{\wh{K}}$.
  The items (2) and the ($R = K$) case of item (3)  follow from 
  \cite[Lem.~1]{Kato83} and \cite[Lem.~21]{Kato-Inv} while 
  (4) follows from \cite[Cor.~4]{Banaszak}.
  It remains to prove (3) when $R =A$.
 But this is clear if either $n \neq 2$ or
 $p \nmid q$ because of the following facts.
 \begin{enumerate}
 \item
   $(A^\times)\{p'\} \xrightarrow{\cong} (\wh{A}^\times)\{p'\}$.
   \item
     $\Br(A) \xrightarrow{\cong} \Br(\wh{A})$.
   \item
     $\Omega^i_{A, \log} = \Omega^i_{\wh{A}, \log} = 0$ for $i \ge 3$.
   \item
     $H^1_\et(A, {\Z}/{p^m}) \cong H^1_\et(\ff, {\Z}/{p^m})$.
     \end{enumerate}
  
The  ($n =2, q = p^m$) case follows by using the exact sequence
  \begin{equation}\label{eqn:Nis-et-ICG-2}
      0 \to H^1_\et(A, W_m\Omega^2_{A, \log}) \to
      H^1_\et(K, W_m\Omega^2_{K, \log}) \xrightarrow{\partial_1}
      H^1_\et(\ff, W_m\Omega^1_{\ff, \log}) \to 0,
    \end{equation}
    which one easily deduces from \cite[Thms.~3.2, 4.1]{Shiho}, and by
    noting that $\partial_1$ is an isomorphism by \lemref{lem:Kato-equiv} and
    \cite[\S~3, Lem.~3]{Kato80-2}
\end{proof}

%\enlargethispage{25pt}

\subsection{Kato topology}\label{sec:Kato-top}
Since we shall use Kato's topology extensively in this paper, we briefly
recall it here.
More details can be found in \cite[\S~7]{Kato80-1} (see also
\cite[\S~7.1]{KRS}).

Let $K$ be a 2-local field of positive characteristic, i.e., a
complete discrete valuation field of characteristic $p > 0$
whose residue field $\ff$ is a local field. Let $A$ denote the ring of integers of
$K$ with maximal ideal $(\pi)$. One knows that there
is a canonical isomorphism of rings $\phi \colon \ff[[\pi]] \xrightarrow{\cong}
A$. The Kato topology on $A$ is the unique topology such that
$\phi$ is an isomorphism of topological rings when $\ff[[\pi]]$ is endowed with the
product of the valuation (also called the adic) topology of $\ff$.
The Kato topology of $A^\times$
is the subspace topology induced from that of $A$.
The Kato topology of $K^\times$ is the unique topology which is compatible
with the group structure of $K^\times$ and for which
$A^\times$ is open in $K^\times$. The Kato topology of $K_2(K)$ is the finest
topology which is compatible with the group structures and for which
the product map $K^\times \times K^\times \to K_2(K)$ is continuous.
This makes $K_2(K)$ into a topological abelian group.

For any integers $m, r \ge 0$,
the Kato topology of $K_2(K)$ induces the quotient topology on ${K_2(K)}/m$
and the subspace topology on $\Fil_r K_2(K)$. We shall use the term Kato topology
for these topologies as well.
By the Kato topologies of $\kappa^m_r(A|n)$ (resp. $\varpi^m_r(A|n)$), we shall
mean the subspace topology of $\kappa^m_r(A|n)$ induced from that of
${K_2(K)}/{p^m}$ (resp. the
  quotient topology of $\kappa^m_r(A|0)$ via ~\eqref{eqn:Kato-0}).
  Let $\partial_K \colon K_2(K) \surj \ff^\times$ be the Tame symbol map.
Via this surjective homomorphism, the Kato topology of $K_2(K)$ induces
the quotient topology (called the Kato topology) on $\ff^\times$.
On the other hand, the latter is also equipped with its valuation topology.
The following should be known to experts but we could not find a reference.

\begin{lem}\label{lem:Kato-adic}
The two topologies of $\ff^\times$ coincide.
\end{lem}
\begin{proof}
  To show that the adic topology of $\ff^\times$ is finer than the Kato topology,
  it suffices to show that if $U \subset \ff^\times$ is a subgroup which is
  open in the Kato topology, then it contains an open subgroup in the adic topology.
  To that end, we look at the composite map $T \colon K^\times \times K^\times
  \to K_2(K) \xrightarrow{\partial_K} \ff^\times$. By definition, $T$ is 
  continuous with respect to the Kato topology of $\ff^\times$.
  It follows that the map $T_\pi \colon K^\times \to \ff^\times$,
  given by $T_\pi(a) = \partial_K(\{a, \pi\})$ is a continuous homomorphism.
  In particular, $T_\pi^{-1}(U)$ contains an open subgroup of $K^\times$.
  By definition of the Kato topology of $K^\times$ and
  \cite[\S~7, Rem.~1]{Kato80-1}, $T^{-1}_\pi(U)$ contains a subgroup of the
  form $V_0 \times V_1$, where $V_0 \subset \ff^\times$ is an open subgroup in the
  adic topology and
  $V_1 \subset {\underset{i \ge 1}\prod} \ff \pi^i$ is an open subgroup
  in the Kato topology of $\pi A \subset A$. It is clear from the
  definition of $T$ that $T_\pi(V_0 \times V_1) = V_0$ which implies that
  $V_0 \subset U$.

To show that the Kato topology of $\ff^\times$ is finer than the adic topology,
  it is equivalent to show that $\partial_K$ is continuous with respect  to the
  Kato topology on $K_2(K)$  and adic topology on $\ff^\times$.
  Now,  we look at the commutative diagram
  (cf. \cite[\S~6, Rem.~1]{Kato80-1})
  \begin{equation}\label{eqn:Kato-adic-0}
    \xymatrix@C.8pc{
K_2(K) \ar[r]^-{\partial_K} \ar[d]_-{\rho_K} & \ff^\times
      \ar@{^{(}-}[d]^-{\rho_\ff} \\
      G_K \ar[r]^-{\partial'_K} & G_\ff,}
  \end{equation}
  where the vertical arrows are the reciprocity homomorphisms.
  It is well known and easy to show that the adic topology of $\ff^\times$
  coincides with the subspace topology induced from the profinite topology of
  $G_\ff$ via $\rho_\ff$. Since $\partial'_K$ is continuous and so is
  $\rho_K$ (cf. \cite[Thm.~2, p.~304-305]{Kato80-1}), it follows that
  $\partial_K$ is also continuous. 
  \end{proof}

In the sequel, the default topology of ${K_2(K)}/n$ ($n \ge 0$) and its subspaces
      (e.g., ${K_2(A)}/n$) will be the Kato topology, and the same for
      ${\ff^\times}/n$ will be the adic topology.
If $K'$ is a Henselian discrete valuation field with completion $K$, then the
Kato topology of $K'^\times$ is the subspace topology induced from that
of $K^\times$. The Kato topology on $K_2(K')$ is the finest topology
such that the product map $K'^\times \times K'^\times \to K_2(K')$ is
continuous. It is clear that the canonical map $K_2(K') \to K_2({K})$ is
continuous.

Before we prove the next lemma, we recall some topological notions.
For $G \in \Ab$, we let $G^\vee =  \Hom_\Ab(G, {\Q}/{\Z})$.
For $G \in \Tab$, we let $G^\star =  \Hom_\Tab(G, {\Q}/{\Z})$.
If $G$ is either a profinite or a torsion locally compact Hausdorff topological
abelian group, then $G^\star$ is a topological abelian group with its compact open
topology and satisfies the Pontryagin duality, i.e., the evaluation map
$G \to (G^*)^*$ is an isomorphism in $\Tab$ (cf. \cite[\S~2.9]{Pro-fin}).
For $A, B \in \Tab$, we let $\Hom_\cf(A,B) = (\Hom_{\Tab}(A,B))_\tor$.
We shall write $\Hom_\cf(A,{\Q}/{\Z})$ as $A^\star_{\fr}$.
In this paper, we shall endow $\Q$ as well as the direct sums of its subgroups and
subquotients (e.g., finite abelian groups) with the discrete topology.

\begin{lem}\label{lem:K-dual}
  The map $\partial_K \colon K_2(K) \surj \ff^\times$ induces an exact sequence
  \begin{equation}\label{eqn:K-dual-0}
    0 \to (\ff^\times)^\star \to {K_2(K)}^\star \to {K_2(A)}^\star \to 0.
  \end{equation}
 \end{lem}
 \begin{proof}
    Since $K_2(A) = \Ker(\partial_K)$, the exactness of ~\eqref{eqn:K-dual-0}
    except at $K_2(A)^\star$ follows immediately from \lemref{lem:Kato-adic}.
    To prove the exactness at ${K_2(A)}^\star$, we let
    $\phi_K \colon K_2(K) \to K_2(A)$ be
    given by $\phi_K(\{a, b\}) = \{a,b\} - \{\partial_K(\{a,b\}), \pi\}$.
    It is easily checked that $\phi_K$ is a group homomorphism whose restriction to
    $K_2(A)$ is identity. It remains to show that $\phi_K$ is continuous with
    respect to the Kato topology. Equivalently, we need to show that 
    $\phi_K \colon K_2(K) \to K_2(K)$ is continuous. As $\partial_K$ is
    continuous, the latter claim follows if we show that the map
    $\psi_K \colon \ff^\times \to K_2(K)$ is continuous, where $\psi_K(a) =
    \{a, \pi\}$. But this map is the composition of
    $\ff^\times \inj K^\times \to K_2(K)$, where the first map is the canonical
inclusion and the second map takes $a$ to $\{a, \pi\}$.
The second map is clearly continuous and the first map is continuous by
\cite[\S~7, Rem.~1]{Kato80-1}. 
\end{proof}

\begin{remk}\label{remk:K-dual-1}
The same proof as that of \lemref{lem:K-dual} shows that the sequence
\begin{equation}\label{eqn:K-dual-2}
    0 \to ({\ff^\times}/n)^\star \to ({K_2(K)}/n)^\star \to ({K_2(A)}/n)^\star \to 0
  \end{equation}
is exact for every $n \in \Z$.
\end{remk}

\subsection{The ramification filtration}\label{secKRF}
We recall Kato's ramification filtration \cite{Kato89}.
Let $K$ be a Henselian discrete valuation field of 
characteristic $p > 0$ with ring of integers
$A$, maximal ideal $\fm_K \neq 0$ and residue field $\ff$.
For an integer $m \ge 1$, the Artin-Schreier-Witt sequence gives rise to an exact
sequence
\begin{equation}\label{eqn:DRC-3}
0 \to {\Z}/{p^m} \to W_m(K) \xrightarrow{1 - \ov{F}} W_m(K)
\xrightarrow{\partial_m} H^1_{\et}(K, {\Z}/{p^m}) \to 0,
\end{equation}
where $\ov{F}((a_{m-1}, \ldots , a_{0})) = (a^p_{m-1}, \ldots , a^p_{0})$.
We write $(a_{m-1}, \ldots , a_{0})$ as $\un{a}$ in short.
Let $v_K \colon K^{\times} \to \Z$ be the normalized valuation.
We fix integers $n \geq 0$ and $m \ge 1$ and let
$\Fil_{n} W_m(K) = \{\un{a}| p^{i}v_K(a_i) \ge -n \ \forall \ i\}$.
We let $\Fil_{n}H^1_\et(K, {\Z}/{p^m}) = \partial_m(\Fil_{n} W_m(K))$ and
$\Fil_n H^1_\et(K, {\Q_p}/{\Z_p}) = {\underset{m \ge 1}\bigcup} \
\Fil_{n}H^1_\et(K, {\Z}/{p^m})$.
We let $ \Fil_n H^1(K) = 0$ for $n < -1$ and
\begin{equation}\label{eqn:Fil-K-p-1}
  \Fil_n H^1(K) = \left\{ \begin{array}{ll}
                            H^1(K)\{p'\} \bigoplus \Fil_n H^1_\et(K, {\Q_p}/{\Z_p})
                            & \mbox{if $n \ge 0$} \\
                            H^1(A) & \mbox{if $n = -1$}.
                          \end{array}
                        \right.
                      \end{equation}
Let $\lfloor \cdot \rfloor$ (resp. $\lceil \cdot \rceil$) denote the
greatest integer (resp. smallest integer) function on $\N$.

\begin{lem}\label{lem:Matsuda-0}
  For $n \ge 0$ and $\un{a} \in W_m(K)$, one has $(1-\ov{F})(\un{a}) \in
  \Fil_n W_m(K)$
  if and only if $\un{a} \in \Fil_{\lfloor \frac{n}{p}\rfloor} W_m(K)$.
\end{lem}
\begin{proof}
  The lemma is easily deduced using induction on $m$, the exact sequence
  \begin{equation}\label{eqn:Matsuda-0-0}
0 \to \Fil_{n} W_1(K) \xrightarrow{V^{m-1}} \Fil_{n} W_m(K) 
\xrightarrow{R}  \Fil_{\lfloor {\frac{n}{p}}\rfloor} W_{m-1}(K) \to 0,
\end{equation}
and the fact that $(1-\ov{F})(\un{a}) \in \Fil_n W_m(K)$ implies
$(1-\ov{F})(\un{a'}) \in \Fil_{\lfloor {\frac{n}{p}}\rfloor} W_{m-1}(K)$ and
$(1-\ov{F})(a_0) \in
\Fil_{n} W_1(K)$, where $\un{a'} = (a_{m-1}, \ldots , a_1)$. Note that
$m =1$ case is trivial.
\end{proof}

A result we shall need about the ramification filtration is the
following.

\begin{prop}\label{prop:p-tor}
For every $m \ge 1$ and $n \ge 0$, we have 
  \[
    \Fil_{n}H^1_\et(K, {\Z}/{p^m}) = H^1_\et(K, {\Z}/{p^m}) \bigcap 
\Fil_{n}H^1_\et(K, {\Q_p}/{\Z_p}) = ~_{p^m} \Fil_{n} H^1_\et(K, {\Q_p}/{\Z_p})
\]
as subgroups of $H^1_\et(K, {\Q_p}/{\Z_p})$.
\end{prop}
\begin{proof}
 The second equality is obvious. To prove the first, 
we write $\Fil^*_{n}H^1_\et(K, {\Z}/{p^m}) =  H^1_\et(K, {\Z}/{p^m}) \bigcap 
\Fil_{n}H^1_\et(K, {\Q_p}/{\Z_p})$. It is clear that $\Fil_{n}H^1_\et(K, {\Z}/{p^m}) 
\subseteq \Fil^*_{n}H^1_\et(K, {\Z}/{p^m})$.
For the other inclusion, we can assume $n \ge 1$ as the assertion is clear
otherwise. For $m=1$, this is already shown in the proof of
\cite[Lem.~3.6, p.~2002]{Kerz-Saito-ANT-E}.

We now assume $m \ge 2$, write $n' = {\lfloor \frac{n}{p}\rfloor}$ and
look at the commutative diagram
\begin{equation}\label{eqn:p-tor-2}
  \xymatrix@C.8pc{
0 \ar[r] & \Fil_n H^1_\et(K, {\Z}/p) \ar[r]^-{(p^{m-1})^*} \ar[d] &
\Fil_n H^1_\et(K, {\Z}/{p^m}) \ar[r]^-{R^*} \ar[d] &
\Fil_{n'} H^1_\et(K, {\Z}/{p^{m-1}}) \ar[d] \ar[r] & 0. \\ 
 0 \ar[r] & \Fil^*_n H^1_\et(K, {\Z}/p) \ar[r]^-{(p^{m-1})^*} & 
\Fil^*_n H^1_\et(K, {\Z}/{p^m}) \ar[r]^-{R^*} & 
\Fil^*_{n'} H^1_\et(K, {\Z}/{p^{m-1}}) &}
\end{equation}
The arrows in the bottom row are the restrictions
of $(p^{m-1})^*$ and $R^*$ induced on $H^1_\et(K, {\Z}/{p^i})$ by the
maps of sheaves ${\Z}/{p} \xrightarrow{p^{m-1}} {\Z}/{p^m} \xrightarrow{R}
{\Z}/{p^{m-1}}$. The exactness of the bottom row is clear and the same for
the top row is easily deduced from \lemref{lem:Matsuda-0} and
~\eqref{eqn:Matsuda-0-0}.
The left and right vertical arrows are bijective by induction on $m$.
It follows that $R^*$ on the bottom row is surjective and the middle vertical
arrow is bijective.
\end{proof}

\begin{cor}\label{cor:p-tor-5}
    For any pair of integers $n \ge -1$ and $m \ge 1$, the canonical maps
    $\Fil_n H^1(K) \to \Fil_n H^1(\wh{K})$ and $\Fil_n H^1_\et(K, {\Z}/{p^m}) \to
    \Fil_n H^1_\et(\wh{K}, {\Z}/{p^m})$ are isomorphisms.
  \end{cor}
  \begin{proof}
    The $n \le 0$ case of the first isomorphism is clear from
    \lemref{lem:Kato-equiv}. The general case follows
    by induction using \cite[Cor.~3.3]{Kato89}. The second isomorphism follows
    by the first isomorphism, \lemref{lem:Kato-equiv} and \propref{prop:p-tor}.
\end{proof}

\subsection{The reciprocity homomorphism}\label{sec:RHom}
Let $K$ be a Henselian discrete valuation field of characteristic $p > 0$
whose residue field $\ff$ is a local field.
Let $A$ denote the ring of integers of $K$ and $\fm = (\pi)$ the maximal ideal of
$A$. We let $S = \Spec(A)$ and let $x$ denote the closed
point of $S$. We shall write $S_\eta = \Spec(K)$ and $S_n = \Spec(A/{\fm^{n+1}})$ for
  $n \ge 0$.
We let $j \colon S_\eta \inj S$ and $\iota \colon S_0 \inj S$ be the inclusions.
Recall from \cite{Kato80-1} that the
reciprocity homomorphism $\rho_K \colon K_2(K) \to G_K$ is induced by the cup
product pairing
\begin{equation}\label{eqn:KR**}
  K_2(K) \times H^1(K) \xrightarrow{{\rm NR}_K \times \id}
  H^2_\et(K, {\Q}/{\Z}(2)) \times H^1(K) \xrightarrow{\cup} H^3(K) \cong {\Q}/{\Z}.
\end{equation}
One knows that $\rho_K$ is a continuous homomorphism with respect to the Kato
topology of $K_2(K)$ and the profinite topology of $G_K$.
 Letting $G^{(n)}_K = \left(\frac{H^1(K)}{\Fil_{n-1} H^1(K)}\right)^\star$
  and $G^{(n)}_{K,m} = \left(\frac{H^1(K, {\Z}/{p^m})}{\Fil_{n-1} H^1(K, {\Z}/{p^m})}
  \right)^\star$, we have the following.

\begin{lem}\label{lem:LR-mod}
    For $n \ge 0$ and $m \ge 1$, ~\eqref{eqn:KR**} induces continuous homomorphisms
    \[
      (1) \ \rho^n_{K} \colon \Fil_{n} K_2(K) \to G^{(n)}_K; \
      (2) \ \rho^n_{K,m} \colon \kappa^m_2(A|n) \to  G^{(n)}_{K,m}.
    \]
  \end{lem}
\begin{proof}
    The existence of $\rho^n_{K,m}$ follows from that of $\rho^n_K$ using
    \propref{prop:p-tor}. For $n \ge 1$, item (1) follows from
    ~\eqref{eqn:KR**} and \cite[Rem.~6.6]{Kato89}.
    The map $\rho^0_K$ is the one induced between the kernels of the horizontal
    arrows in ~\eqref{eqn:Kato-adic-0}. In particular, it is continuous.
\end{proof}

\begin{lem}\label{lem:REC-Loc-0}
 For every integer $q \ge 1$, we have exact sequences
  \[
    0 \to {K_2(A)}/q \to {K_2(K)}/q \xrightarrow{\partial_K} {\ff^\times}/q \to 0; \ \ \
      0 \to {G^{(0)}_K}/q \to {G_K}/q \xrightarrow{\partial'_K} {G_\ff}/q \to 0.
      \]
\end{lem}
\begin{proof}
  The exactness of the sequence on the left is well known.
  To show the exactness of the other sequence, note that
  $G_\ff\{p\} = G_K\{p\} = 0$ by \lemref{lem:No-p-tor}.
  If $p \nmid q$, then one knows that $\rho_\ff \colon ~_{q} \ff^\times \to
  ~_{q} G_\ff$ is an isomorphism. On the other
  hand, $\partial_K$ is a split surjection whose section is given by
$\iota_*(a) = \{a, \pi\}$. It follows that $\partial_K$ and
$\partial'_K$ are surjective at the level of torsion subgroups. This proves our
claim.
\end{proof}

\begin{lem}\label{lem:No-p-tor}
If $U$ is an affine $\F_p$-scheme, then $\pi^{\ab}_1(U)\{p\} = 0$.
\end{lem}
\begin{proof}
  We need to show that $H^1(U)$ is $p$-divisible. Since
  ${H^1(U)}/p \inj H^2_\et(U, {\Z}/{p})$, we are done by the Artin-Schreier exact
  sequence and \cite[Thm.~III.3.7]{Hartshorne-AG}.
  \end{proof}

\section{A duality theorem for $\kappa^1_2(A|n)$}
\label{sec:D-Thm}
We maintain the notations and hypotheses of \S~\ref{sec:RHom}. 
In this section, we additionally assume that $A$ is complete and prove a perfect
duality theorem for $\kappa^1_2(A|n)$,
a key step in the proof of \thmref{thm:Main-1}.
Since we shall only deal with $m =1$ case, we shall
denote the map $\rho^n_{K,1}$ (cf. \lemref{lem:LR-mod}) by $\rho^n_K$, for brevity.

We fix some further notations.
For $Y = \Spec(R)$ and a sheaf $\sF$
on $Y_\tau$, we shall write $H^*_\tau(Y, \sF)$ as $H^*_\tau(R, \sF(Y))$.
In particular, $H^*_\tau(S, W_m\Omega^r_{S|S_n, \log})$ will be written as
$H^*_\tau(A, W_m\Omega^r_{A|n+1, \log})$ and so on.
We write $\ff = \F_q((t))$ and $A = \ff[[\pi]]$ so that
$K = \ff((\pi)) = \F_p((t))((\pi))$.
Thanks to the $F$-finiteness of $\ff$ and $A$, one knows that
$\Omega^1_A \cong Adt \bigoplus Ad\pi$ and $\Omega^1_{A/\ff} \cong Ad\pi$.
In particular, $\Omega^2_A \cong A\omega$, where we let $\omega = dt\wedge d\pi$.
We let $\omega' = \dlog(t) \wedge \dlog(\pi) \in \Omega^2_{K, \log}$.
We shall write $\sO_S(n) = \sO_S(nE)$, where $E = V((\pi))$.
For $n \in \Z$ and  $M \in A$-mod, we let $A(n) = H^0_\et(S, \sO_S(n)) =
\{f \in K| f = \sum_{i \ge -n} a_i \pi^i \ \mbox{with} \ a_i \in \ff\}$ and
$M(n) = M \otimes_A A(n)$. We let
\begin{equation}\label{eqn:Diff-symbol}
\Omega^1_A(\log \pi) = A dt \bigoplus A\dlog(\pi) \subset \Omega^1_K, \
\Omega^2_A(\log \pi) = A (dt \wedge \dlog(\pi)) \subset \Omega^2_K,
\end{equation}
\[
\Omega^i_{A|n} = \Omega^i(\log \pi)(-n), \ B\Omega^i_{A|n} = d\Omega^{i-1}_{A|n}
  \ \mbox{and} \
  B'\Omega^i_{A|n} = \frac{\Omega^i_{A|n}}{d\Omega^{i-1}_{A|n}}.
\]

We let $\psi \colon S \to S$ be the absolute Frobenius. This is finite
and flat as $S$ is regular. 
For any $a \in K$ (resp. in $\ff$), we shall denote by $\ov{a}$ the residue
class of $a$ in $K/A$ (resp. in ${\ff}/{\sO_\ff}$).
For any $A$-module $M$, we let $H^*_\fm(M)$ denote the local cohomology modules of
$M$. We let $Z\Omega^i_A (n) = \Omega^i_A (n) \bigcap \Ker(d)$, where
$d \colon \Omega^i_K \to \Omega^{i+1}_K$ is the classical differential.

\subsection{Some complexes and their pairings}\label{sec:Pair}
We now fix an integer $n \ge 0$ and consider the complexes of sheaves
$\sF_n = (\psi_* (ZA(n)) \xrightarrow{1-C} A(n)); \
\sG_n = (\Omega^2_{A|n+1} \xrightarrow{1-\ov{F}} \psi_* B'\Omega^2_{A|n+1})$ and
$\sH = (\Omega^2_A \xrightarrow{1-C} \Omega^2_A)$ in $D_\et(S)$.
Recall here that $C \colon H^i(\Omega^\bullet_R) \xrightarrow{\cong}
\Omega^i_R$ is the Cartier operator for a regular $\F_p$-algebra $R$ which 
is characterized by the properties
\begin{equation}\label{eqn:Cartier}
  C(a^p) = a, \ C(a^{p-1}da) = da, C(\alpha \wedge \alpha') = C(\alpha) \wedge
  C(\alpha').
\end{equation}
It is easy to check
that the wedge product of differential forms gives rise to a pairing of complexes
(cf. \cite[\S~1, p.~175]{Milne-ENS}, \cite[\S~3.1]{JSZ})
\begin{equation}\label{eqn:Pairing-0}
  \<,\> \colon \sF_n \times \sG_n \to \sH;
\end{equation}
\[
  \<a, b\>^0_{0,0} = a \wedge b, \ \<a,b\>^1_{1,0} = a \wedge b, \
\<a,b\>^1_{0,1} = -C(a \wedge b).
\]
On restriction to $S_\eta$ and taking the associated
pairing of cohomology, we get a pairing
\begin{equation}\label{eqn:Pairing-2}
  H^i_\et(K, {\Z}/p) \times H^{1-i}_\et(K, \Omega^2_{K,\log}) \to
  H^1_\et(K, \Omega^2_{K, \log}) \xrightarrow{{\Res}'_K} {\Z}/p,
\end{equation}
where ${\Res}'_K$ is Kato's residue (cf. \cite[\S~3, Prop.~1]{Kato80-2}).

\begin{lem}\label{lem:Pairing-1}
  The pairings of finite type $A$-modules
  \begin{equation}\label{eqn:Pairing-1-0}
    A(n) \times  \Omega^2_{A|n+1} \xrightarrow{\<,\>^1_{1,0}} \Omega^2_A \cong A; \
    \psi_* (Z A(n)) \times \psi_* B'\Omega^2_{A|n+1} 
    \xrightarrow{\<,\>^1_{0,1}}  \Omega^2_A \cong A
  \end{equation}
  are perfect.
\end{lem}
\begin{proof}
  The perfectness claim for the first pairing is straightforward and the same
  for the second pairing follows at once from the first
  (cf. \cite[Lem.~1.7]{Milne-ENS}).
\end{proof}

Recall that there is a commutative diagram 
\begin{equation}\label{eqn:Res-1}
  \xymatrix@C1pc{
    H^1_\fm(\Omega^1_{A/\ff}) \ar[r]^-{\cong} \ar[d]_-{\gamma_A}
    & \frac{\Omega^1_{K/\ff}}{\Omega^1_{A/\ff}}
    \ar[d]^-{\gamma_A} \ar[r]^-{\res'_K} & \ff \ar[d]^-{\gamma_\ff}
    \ar[rr]^-{\chi_\ff} & & {\Z}/{p} \\
    H^1_\fm(\Omega^2_A) \ar[r]^-{\cong} & \frac{\Omega^2_K}{\Omega^2_A}
    \ar[r]^-{\res_K} & \Omega^1_\ff \ar@{->>}[r] &
    \frac{\Omega^1_\ff}{\Omega^1_{\sO_\ff}} \ar[r]^-{\res_\ff} & \F_q \ar[u]_-{\tr},}
  \end{equation}
  where $\gamma_A(\ov{a} d\pi) = \ov{a}\omega$ (for ${a} \in K$) is an
  isomorphism of $A$-modules while $\gamma_\ff({b}) = bdt$ (for $b \in \ff$) is
  an isomorphism of $\ff$-modules.
The middle horizontal arrow on the top is the Poincar{\'e} residue map
  $\ov{a} d\pi \mapsto \res(a)$ and the one below it on the bottom
  is $\ov{a} \omega \mapsto \res(a) dt$, where $a \in K$.
  The arrow $\res_\ff$ is the residue map $\ov{b}dt \mapsto \res(b)$ for $b \in
 {\ff}$. Here, $\res(a)$ (res. $\res(b)$) denotes the classical residue of the
Laurent power series $a \in K = \ff((\pi))$ (resp. $b \in \ff = \F_q((t))$).
We let $\Res_K \colon  H^1_\fm(\Omega^2_A) \to {\Z}/p$ denote the composition
of $\tr$ and the bottom horizontal arrows.

We write $H^1_\fm(A(n)) \cong {\underset{i \ge n+1}\bigoplus} \ff \pi^{-i}$.
  We endow each $\ff \pi^{-i}$ with the adic topology of $\ff$
   under the canonical isomorphism $\ff \xrightarrow{\cong} \ff \pi^{-i}$ and let
   $H^1_\fm(A(n)) $ be endowed with the direct sum topology
   (cf. \cite[\S~2.6]{Gupta-Krishna-REC}). This makes $H^1_\fm(A(n))$ into
   a topological $\ff$-vector space.
   We shall call this the adic topology of $H^1_\fm(A(n))$.
   Using the isomorphism of topological fields $\ff \xrightarrow{p} \ff^p$
   and the isomorphism $\psi_* (ZA(n)) \cong
   (\psi_* A^p)(\lfloor n/p \rfloor)$ in $A$-mod, the above construction shows
   that just like $H^1_\fm(A(n))$, the $A$-module $H^1_\fm(\psi_* (ZA(n)))$ is a
   topological $\ff$-vector space endowed with the adic topology.

We now look at the maps
  \begin{equation}\label{eqn:Local-Duality}
     \Hom_A(M, \Omega^2_A) \times H^1_\fm(M) \to H^1_\fm(\Omega^2_A)
     \xrightarrow{\Res_K} {\Z}/p
   \end{equation}
   for any finite type $A$-module $M$,
   where the first map is induced by the composition of Ext groups of $A$-modules.
   This yields the duality map $\Res^*_M \colon \Hom_A(M, \Omega^2_A)  \to
   \Hom_{{\Z}/{p}}(H^1_\fm(M), {\Z}/p)$. Letting $M \in \{A(n), \psi_* (ZA(n))\}$
   and using \lemref{lem:Pairing-1}, we equivalently get the maps
   $\Res^*_n \colon \Omega^2_{A|n+1}  \to \Hom_{{\Z}/{p}}(H^1_\fm(A(n)), {\Z}/p)$
   and $\Res'^*_n \colon \psi_* B'\Omega^2_{A|n+1} \to
   \Hom_{{\Z}/{p}}(H^1_\fm(\psi_*(ZA(n))), {\Z}/p)$. It is clear from
   the definition of $\Res_K$ and the adic topology of $\ff$ that for every
   $\alpha \in \Omega^2_{A|n+1}$, the map
   $\Res^*_n(\alpha) \colon \ff \pi^{-i} \to {\Z}/p$
   is continuous. It follows that $\Res^*_n(\alpha) \in
   \Hom_\cont(H^1_\fm(A(n)), {\Z}/p) \cong H^1_\fm(A(n))^\star$ with respect to the
   adic topology of $H^1_\fm(A(n))$. The same reason shows that
   $\Res'^*_n(\alpha') \in
   \Hom_\cont(H^1_\fm(\psi_* (ZA(n))), {\Z}/p) \cong H^1_\fm(\psi_* (ZA(n)))^\star$
   for $\alpha' \in  \psi_* B'\Omega^2_{A|n+1}$.

   \vskip .3cm

Our first claim is the following.

   \begin{lem}\label{lem:LD-0}
    We have isomorphisms of $A$-modules:
     \[
       \Res^*_n \colon \Omega^2_{A|n+1} \to  H^1_\fm(A(n))^\star \ \mbox{and} \ \
       \Res'^*_n \colon \psi_* B'\Omega^2_{A|n+1} \to H^1_\fm(\psi_* (ZA(n)))^\star.
     \]
\end{lem}
\begin{proof}
     We only prove the first statement since the same argument applies to the second
     statement too.
It is easy to verify that $\Res^*_n$ is $A$-linear, where the $A$-module
structure of $ H^1_\fm(A(n))^\star$ is defined by $af(x) = f(ax)$.
Now, the local Grothendieck duality for finite type $A$-modules
(cf. \cite[Thm.~6.3]{Hartshorne-LC}) and \lemref{lem:Pairing-1} together
imply that  ~\eqref{eqn:Local-Duality} induces an isomorphism
$\Omega^2_{A|n+1} \xrightarrow{\cong} \Hom_A(H^1_\fm(A(n)), H^1_\fm(\Omega^2_A))$ of
$A$-modules. On the other hand, one easily checks using the definition of $\res_K$
above (cf. \cite[Exm.~3, p.~67-68]{Hartshorne-LC}) that
the map $\res^*_K \colon H^1_\fm(\Omega^2_A) = \Hom_A(A, H^1_\fm(\Omega^2_A)) 
\xrightarrow{\circ \ \res_K} \Hom_\ff(A, \Omega^1_\ff)$ induces an
isomorphism of $\ff$-vector spaces
$\res^*_K \colon H^1_\fm(\Omega^2_A) \xrightarrow{\cong}
\Hom_{\ff, \cont}(A, \Omega^1_\ff)$. Here, $A$ is given its $\fm$-adic topology and
$\Omega^1_\ff$ the discrete topology.

We thus get an isomorphism
\[
  \Omega^2_{A|n+1} \xrightarrow{\cong}
  \Hom_A(H^1_\fm(A(n)), \Hom_{\ff, \cont}(A, \Omega^1_\ff))
\cong \Hom_A(H^1_\fm(A(n)), \Hom_{\ff, \cont}(A, \ff)).
\]
Combining this with the easy-to-check isomorphisms
\[
\Hom_A(H^1_\fm(A(n)), \Hom_{\ff, \cont}(A, \ff)) =
\Hom_A(H^1_\fm(A(n)), \Hom_{\ff}(A, \ff)) \xrightarrow{\cong}
\Hom_\ff(H^1_\fm(A(n)), \ff),
\]
we get that ~\eqref{eqn:Local-Duality} with $M = A(n)$ induces an isomorphism
$ \Omega^2_{A|n+1} \xrightarrow{\cong} \Hom_\ff(H^1_\fm(A(n)), \ff)$.

On the other hand, one checks using the definition of $\res_\ff$ that the map
\[
  \res^*_\ff \colon \Hom_\ff(H^1_\fm(A(n)), \ff) \cong
\Hom_\ff(K/{A(n)}, \ff) \xrightarrow{\circ \ \chi_\ff}
\Hom_\cont(K/{A(n)}, {\Z}/p)
\]
is an isomorphism with respect to the adic topology of
${K}/{A(n)} \cong H^1_\fm(A(n))$ (cf. \cite[Chap.~II, Thm.~3]{Weil}).
It follows that the map $\Res^*_n = \res^*_\ff \circ \res^*_K \colon
\Omega^2_{A|n+1} \to H^1_\fm(A(n))^*$ is an isomorphism. 
\end{proof}

Our second claim is the following.
\begin{lem}\label{lem:LD-1}
We have commutative diagrams
  \begin{equation}\label{eqn:LD-1-0}
    \xymatrix@C.8pc{
      {\Omega^2_K}/{\Omega^2_A} \ar[r]^-{C} \ar[d]_-{\res_K} &
      {\Omega^2_K}/{\Omega^2_A} \ar[d]^-{\res_K} & &
      \Omega^1_\ff \ar[r]^-{C} \ar[d]_-{\res_\ff} & \Omega^1_\ff \ar[d]^-{\res_\ff}
      & & \F_q \ar[r]^-{C} \ar[d]_-{\tr} & \F_q \ar[d]^-{\tr} \\
\Omega^1_\ff \ar[r]^-{C} & \Omega^1_\ff & &
\F_q \ar[r]^-{C} & \F_q  & & {\Z}/p \ar[r]^-{C} & {\Z}/p.}     
\end{equation}
\end{lem}
\begin{proof}
 This is an elementary exercise using ~\eqref{eqn:Cartier} (which implies in
  particular that $C|_{\F_p} = \id$), the equalities
$\res_K(\ov{f}\omega) = \res(f) dt, \ \ \res_\ff(\ov{g}dt) = \res(g)$ 
  for $f \in K$ and $g \in \ff$ (which implies in particular that
  $\res_\ff(\Omega^1_{\sO_\ff}) = 0$), and that the trace map in the
  third diagram commutes with the Frobenius.
 We leave out the details.  
\end{proof}

We now consider the exact sequence of {\'e}tale cohomology with support
\begin{equation}\label{eqn:F-seq-0}
  H^1_\fm(\psi_* (ZA(n))) \xrightarrow{1-C} H^1_\fm(A(n)) \xrightarrow{\tau^n_K}
  \H^2_x(S, \sF_n) \to 0.
\end{equation}
Using this exact sequence and the identities
\begin{equation}\label{eqn:F-seq-1}
  \begin{array}{lllll}
    \frac{H^1_\fm(A(n))}{(1-C)(H^1_\fm(\psi_* (ZA(n))))} & = &
     \frac{W_1(K)}{\Fil_nW_1(K) + (1-C)(K^p)} \\
    & = &
 \frac{W_1(K)}{\Fil_nW_1(K) + (1-\ov{F})(W_1(K))} 
 & = & \frac{H^1_\et(K, {\Z}/p)}{\Fil_n H^1_\et(K, {\Z}/p)}, \\
\end{array}
\end{equation}
we see that the boundary map $\partial_1$ in ~\eqref{eqn:DRC-3} induces a
canonical isomorphism
\begin{equation}\label{eqn:F-seq-2}
  \partial_1 \colon \H^2_x(S, \sF_n) \xrightarrow{\cong}
  \frac{H^1_\et(K, {\Z}/p)}{\Fil_n H^1_\et(K, {\Z}/p)}.
  \end{equation}

As a consequence of the continuity of $C \colon H^1_\fm(\psi_* (ZA(n))) \to
H^1_\fm(A(n))$, the isomorphism $\dlog \colon \kappa^1_2(A|n+1) \xrightarrow{\cong}
\Omega^2_{A|n+1, \log}$ and \cite[Thm.~1.2.1]{JSZ}, we obtain a diagram
\begin{equation}\label{eqn:C-F-sequence}
  \xymatrix@C.8pc{
    0 \ar[r] & \kappa^1_2(A|n+1) \ar[r] \ar@{.>}[d]_-{{\rho}^{n+1}_K} &
    \Omega^2_{A|n+1} \ar[r]^-{1-\ov{F}}
    \ar[d]_-{\Res^*_{n+1}} & \psi_* B'\Omega^2_{A|n+1} \ar[d]^-{- \Res'^*_{n+1}} \\
    0 \ar[r] & \H^2_x(S, \sF_n)^\star \ar[r] & H^1_\fm(A(n))^\star
    \ar[r]^-{1 - C^\vee} & H^1_\fm(\psi_* (ZA(n)))^\star,}
\end{equation}
where $\H^2_x(S, \sF_n)$ is endowed with the quotient topology via
$\tau^n_K$ in ~\eqref{eqn:F-seq-0}.

Our third claim is the following.

\begin{lem}\label{lem:LD-2}
    The right square in ~\eqref{eqn:C-F-sequence} is commutative.
  \end{lem}
  \begin{proof}
We let $A_n = {K}/{A(n)}$ and $B_n = \psi_*({K^p}/{Z(A(n))})$.
    By definition, the right square in ~\eqref{eqn:C-F-sequence} is the composition
    of squares  
\begin{equation}\label{eqn:LD-2-0}
  \xymatrix@C.8pc{
    \Omega^2_{A|n+1} \ar[d]_-{1-F} \ar[r]^-{\vartheta_{n}} &
    \Hom_A(A_n, {\Omega^2_K}/{\Omega^2_A}) \ar[d]^-{\alpha_n} \ar[r]^-{\res^*_K} 
    & \Hom_\ff(A_n, \ff) \ar[r]^-{\chi^*_\ff} & \Hom_{{\Z}/p}(A_n, {\Z}/p)
    \ar[d]^-{1- C^\vee} \\
    \psi_* B'\Omega^2_{A|n+1} \ar[r]^-{\vartheta'_{n}} &
    \Hom_A(B_n, {\Omega^2_K}/{\Omega^2_A})
    \ar[r]^-{-\res^*_K} & \Hom_\ff(B_n, \ff) \ar[r]^-{\chi^*_\ff} &
    \Hom_{{\Z}/p}(B_n, {\Z}/p),}
\end{equation}
where $\alpha_n(\phi) = (1-C) \circ \phi - \phi \circ (1-C)$ and 
$\vartheta_n$ (resp. $\vartheta'_n$) is the duality map induced by the left
(resp. right) pairing of ~\eqref{eqn:Pairing-1-0}.
The commutativity of the left square in ~\eqref{eqn:LD-2-0} is straightforward
(cf. proof of \cite[Lem.~7.1]{Gupta-Krishna-Duality}).
By the definition of $\alpha_n$ and ~\eqref{eqn:Res-1}, the commutativity of the
right square follows if we show that $\res_K, \ \res_\ff$ and $\tr$ commute
with $C$. But this the content of \lemref{lem:LD-1}.
\end{proof}

Our final claim is the following.

\begin{lem}\label{lem:LD-3}
  The map $\kappa^1_2(A|n+1) \xrightarrow{\rho^{n+1}_K} \H^2_x(S, \sF_n)^\vee$
  factors through $\H^2_x(S, \sF_n)^\star$ and makes the left square 
  in ~\eqref{eqn:C-F-sequence} commute.
\end{lem}
\begin{proof}
By \lemref{lem:LD-2}, it is enough to show that the left square of
 ~\eqref{eqn:C-F-sequence} commutes if we replace $\H^2_x(S, \sF_n)^\star$
 (resp. $H^1_\fm(A(n))^\star$) by $\H^2_x(S, \sF_n)^\vee$ (resp. $H^1_\fm(A(n))^\vee$).
We now look at the diagram
  \begin{equation}\label{eqn:LD-3-0}
    \xymatrix@C.8pc{
      \kappa^1_2(A|n+1) \ar[r]^-{\dlog} \ar[ddr]_-{\rho^{n+1}_K} &
      \Omega^2_{A|n+1, \log} \ar[rr] \ar[dr] \ar[dd]^-{\alpha'} & &
      \Omega^2_{A|n+1} \ar[dr] \ar[dd] & \\
      & & \Omega^2_{K, \log} \ar[dd]^<<<<<<<{\alpha} \ar[rr] & & \Omega^2_K \ar[dd] \\
      & \H^2_x(S, \sF_n)^\vee \ar[rr] \ar[dr] & &
      H^1_\fm(A(n))^\vee \ar@{^{(}->}[dr] & \\
      & & H^1_\et(K, {\Z}/p)^\vee \ar[rr] & & K^\vee}
   \end{equation}
   in which the slanted arrows of the cube
   are the canonical maps induced by the inclusion
   $A \inj K$. The map $\alpha$ is induced by ~\eqref{eqn:Pairing-2} and
$\alpha'$ is the restriction of $\alpha$.
   It is clear that except for the back and the front,
   all faces of the cube in the above diagram are commutative. 
   Since the triangle on the left end commutes by the definition of
   $\rho_K$ (cf. ~\eqref{eqn:KR**}),
   it suffices therefore to show that the front face of the cube commutes.

The commutativity of the front face of ~\eqref{eqn:LD-3-0} is equivalent to that
   of the diagram 
 \begin{equation}\label{eqn:LD-3-1}
    \xymatrix@C.8pc{
      H^1_\et(K, \Omega^0_{K, \log}) \times H^0_\et(K, \Omega^2_{K, \log})
      \ar@<7ex>[d]    \ar[r]^-{\wedge} &
      H^1_\et(K, \Omega^2_{K, \log}) \ar[r]^-{{\Res}'_K} &
      {\Z}/p \\
      H^0_\et(K, \Omega^0_{K}) \times H^0_\et(K, \Omega^2_{K})
        \ar[r]^-{\wedge}
        \ar@<7ex>[u]^-{\partial_F}  & H^0_\et(K, \Omega^2_{K})
        \ar[u]^-{\partial_C} \ar[ur]_-{\Res_K}, &}
    \end{equation}
    where $\partial_F$
    (resp. $\partial_C$) is the connecting homomorphism of the cohomology sequence
    associated to $j^*(\sG_n)$ (resp. $j^*(\sF_n)$) and $\Res_K$ is defined in
    ~\eqref{eqn:Res-1} 

It is standard that the diagram of wedge product pairings on the left in
~\eqref{eqn:LD-3-1} is commutative. To show that the triangle on the right end
of ~\eqref{eqn:LD-3-1} is commutative, we look at the diagram
 \begin{equation}\label{eqn:LD-3-2}
      \xymatrix@C.8pc{
        \Omega^2_K \ar@{->>}[r] \ar[d]_-{\partial_C} & H^1_\fm(\Omega^2_A)
          \ar[r]^-{\res_K} \ar[d]^-{\partial_C} &
        \Omega^1_\ff \ar[r]^-{\res_\ff} \ar[d]^-{\partial_C} &
        \F_q \ar[d]^-{\partial_C} \ar[r]^-{\tr} & {\Z}/p \ar@{=}[dl] \\
        H^1_\et(K, \Omega^2_{K, \log}) \ar[r]^-{\partial_K}_-{\cong}
        \ar@/_2pc/[rr]^-{\partial_2} &
        H^2_x(S, \Omega^2_{S,\log}) \ar[r]_-{\cong} &
      H^1_\et(\ff, \Omega^1_{\ff, \log}) \ar[r]^-{\partial_1}_-{\cong} & 
      H^1_\et(\F_q, {\Z}/p), & }
    \end{equation}
    where $\partial_1$ and $\partial_2$ are Kato's residue isomorphisms 
    (cf. \cite[\S~3, Prop.~1]{Kato80-2}). Since $\Res_K = \tr \circ \res_\ff \circ
    \res_K$ and $\Res'_K = \partial_1 \circ \partial_2$, the commutativity of 
    the triangle on the right end of ~\eqref{eqn:LD-3-1} is equivalent to that
    of the composite outer trapezium in ~\eqref{eqn:LD-3-2}.
    However, the composition of the first two squares in ~\eqref{eqn:LD-3-2}
    is commutative by \cite[\S~1.6, Lem.~13]{Kato80-2}, and the same proof works
    verbatim for the commutativity of the third square too. It remains to observe that
    the triangle on the right end of ~\eqref{eqn:LD-3-2} commutes. But this
    is clear because
    $\frac{\F_q}{(1-C)(Z\F_q)} = \frac{\F_q}{(1 -\ov{F})(\F_q)}$ and the sequence
    $\F_q \xrightarrow{1 -\ov{F}} \F_q \xrightarrow{\tr} {\Z}/p \to 0$ is
   exact. 
    %Use the bar resolution of $\Z$ as $\Z[G]$-module and apply $\Hom_G(\Z[G], \F_q)$ to this resolution, where $G$ = \Gal({F_q}/{\F_p})$.
\end{proof}

The upshot of what we have shown above is the following.

 \begin{prop}\label{prop:LD-4}
   For every $n \ge 0$, the map $\rho^{n+1}_{K,1} \colon \kappa^1_2(A|n+1) \to
   \left(\frac{H^1_\et(K, {\Z}/p)}{\Fil_n H^1_\et(K, {\Z}/p)}\right)^\vee$ admits a
   factorization
   \[
     \kappa^1_2(A|n+1) \xrightarrow{\cong} 
     \left(\frac{H^1_\et(K, {\Z}/p)}{\Fil_n H^1_\et(K, {\Z}/p)}\right)^\star
     \inj \left(\frac{H^1_\et(K, {\Z}/p)}{\Fil_n H^1_\et(K, {\Z}/p)}\right)^\vee,
   \]
   where the first dual is taken with respect to the adic topology.
   In particular, $\rho^{n+1}_{K,1}$ is injective.
 \end{prop}
 \begin{proof}
   Combine Lemma~\ref{lem:LD-0},  ~\eqref{eqn:C-F-sequence} and \lemref{lem:LD-3}.
   \end{proof}

 \section{Reciprocity for 2-local fields II}\label{sec:RL-2}
 In this section, we shall complete the proof of \thmref{thm:Main-1}.
 We let $K$ be a Henselian discrete valuation field of
   characteristic $p > 0$ whose residue field $\ff$ is a local field.
Let $A$ denote the ring of integers of $K$ and $\fm = (\pi)$ the maximal ideal of
$A$. We shall maintain the notations of \S~\ref{sec:RHom}.

\subsection{Kernel of $\rho_K$}\label{sec:Inj-FC}
The following two results will prove part (1) of \thmref{thm:Main-1}, and also the
divisibility of $\Ker(\rho_K)$.

\begin{thm}\label{thm:Local-inj}
  For every pair of integers $m \ge 1$ and $n \ge 0$, the map
  $\rho^n_{K,m} \colon \kappa^m_2(A|n) \to G^{(n)}_{K,m}$
  is injective. 
\end{thm}
\begin{proof}
  In view of \lemref{lem:Kato-equiv} and \corref{cor:p-tor-5}, we can assume
  that $K$ is complete.
Since $\kappa^m_2(A|0) = \kappa^m_2(A|1)$ by \lemref{lem:Kato-1} and
  $G^{(0)}_{K,m} = G^{(1)}_{K,m}$ by ~\eqref{eqn:Fil-K-p-1}, we can assume
  that $n \ge 1$. We shall now prove the theorem by induction on $m$.
  The base case $m =1$ follows from  \propref{prop:LD-4}.
 For $m \ge 2$, we let $n' = \lceil n/p \rceil$ and
 look at the commutative diagram
\begin{equation}\label{eqn:Local-inj-0}
  \xymatrix@C1pc{
    0 \ar[r] & \kappa^{m-1}_2(A|n') \ar[r]^-{p} \ar[d]_-{\rho^{n'}_{K,m-1}} &
    \kappa^m_2(A|n) \ar[r] \ar[d]^-{\rho^{n}_{K,m}} &
    \kappa^1_2(A|n) \ar[r] \ar[d]^-{\rho^n_{K,1}} & 0 \\
    0 \ar[r] &  G^{(n')}_{K,m-1} \ar[r] &
    G^{(n)}_{K,m} \ar[r]^-{p} &
    G^{(n)}_{K,1} \ar[r] & 0,}
\end{equation}
where the unmarked arrows are the canonical inclusions and surjections.

The top row in ~\eqref{eqn:Local-inj-0}
is exact by \cite[Thm.~1.1.7]{JSZ}. The exactness of the bottom
row follows (via the Pontryagin duality)
by comparing the exact sequence
\begin{equation}\label{eqn:Local-inj-1}
0 \to H^1_\et(K, {\Z}/{p}) \xrightarrow{(p^{m-1})^*} H^1_\et(K, {\Z}/{p^m}) 
\xrightarrow{R^*} H^1_\et(K, {\Z}/{p^{m-1}}) \to 0
\end{equation}
with the bottom row of ~\eqref{eqn:p-tor-2} and using \propref{prop:p-tor}.
Note here that $\lceil n/p \rceil = \lfloor {(n-1)}/p \rfloor + 1$.
Since the left and the right vertical arrows in ~\eqref{eqn:Local-inj-0} are
injective by induction on $m$, we are done.
\end{proof}

\begin{cor}\label{cor:Local-inj-2}
For every integer $q \ge 1$, the map
${\rho_{K}} \colon {K_2(K)}/q \to
  {G_{K}}/q$ is injective. Equivalently, the map
  ${\rho_{A}} \colon {K_2(A)}/q \to {G^{(0)}_{K}}/q$ is injective.
\end{cor}
\begin{proof}
  The two assertions are equivalent by \lemref{lem:REC-Loc-0}.
  To prove the result for $K$, we can assume $K$ to be complete in view of
  \lemref{lem:Kato-equiv} and \corref{cor:p-tor-5}.
If $q = p^m$, the claim is the $n = 0$ case
  of \thmref{thm:Local-inj}. If $p \nmid q$, then the Merkurjev-Suslin isomorphism
${\rm NR}_K \colon {K_2(K)}/q \xrightarrow{\cong} H^2_\et(K, {\Z}/q(2))$ and 
the Poincar{\'e} duality for the {\'e}tale cohomology of $K$
(cf. \cite[Thm.~I.2.17]{Milne-AD} or \cite{Koya})
imply that ${\rho_K}/q$ is in fact an isomorphism.
\end{proof}

The following result recovers a theorem of Fesenko \cite{Fesenko}.

\begin{cor}\label{cor:Local-inj-8}
  $\Ker(\rho_K)$ is the maximal divisible subgroup of $K_2(K)$.
\end{cor}
\begin{proof}
  Combine \lemref{lem:No-p-tor}, \corref{cor:Local-inj-2} and
  \propref{prop:Local-ker-coker-0}.
\end{proof}

\subsection{Reciprocity map for $\varpi^m_2(A|n)$}
\label{sec:Rec-quot}
Recall from ~\eqref{eqn:Kato-0} that $\varpi^m_2(A|n)$ is the cokernel of the
canonical inclusion $\kappa^m_2(A|n) \inj \kappa^m_2(A|0)$ for $n \ge 0$.
Since $\left\{\rho^n_{K,m}\right\}_{n \ge 0} \colon
\left\{\kappa^m_2(R|n)\right\}_{n \ge 0} \to \left\{\left(\frac{H^1_\et(K, {\Z}/{p^m})}
  {\Fil_{n-1} H^1_\et(K, {\Z}/{p^m})}\right)^\vee\right\}_{n \ge 0}$ constitute a
morphism between two pro-abelian groups as we vary $n$, taking the
quotients of the transition maps of these pro-abelian groups yields
a reciprocity homomorphism
\begin{equation}\label{eqn:Rec-quot-0}
  \wt{\rho}^n_{K,m} \colon \varpi^m_2(R|n) \to
  \left(\frac{\Fil_{n-1} H^1_\et(K, {\Z}/{p^m})}
    {\Fil_{-1} H^1_\et(K, {\Z}/{p^m})}\right)^\vee.
\end{equation}
Our aim in this subsection is to study this map as an application of the results of
\S~\ref{sec:D-Thm}. We shall use the results of this subsection in the
later parts of this paper.

We first consider a special case. We fix an integer $n \ge 0$ and assume that $K$ is
complete. We look at the commutative diagram
\begin{equation}\label{eqn:Local-inj-9}
  \xymatrix@C1pc{
    0 \ar[r] & \frac{A(n)}{A} \ar[r]^-{\alpha_n} \ar@{->>}[d]_-{\partial_1} &
    \frac{K}{A} \ar[r]^-{\beta_n} \ar@{->>}[d]^-{\partial_1} &
    \frac{K}{A(n)} \ar[r] \ar@{->>}[d]^-{\partial_1} & 0 \\
    0 \ar[r] & \frac{\Fil_n H^1_\et(K, {\Z}/{p})}{\Fil_0 H^1_\et(K, {\Z}/{p})}
    \ar[r]^-{\alpha'_n} & \frac{H^1_\et(K, {\Z}/{p})}{\Fil_0 H^1_\et(K, {\Z}/{p})}
    \ar[r]^-{\beta'_n} &
  \frac{H^1_\et(K, {\Z}/{p})}{\Fil_n H^1_\et(K, {\Z}/{p})} \ar[r] & 0,}
\end{equation}
in which $\alpha_n$ (resp. $\alpha'_n$) and $\beta_n$ (resp. $\beta'_n$)
are canonical inclusion and surjection.

Since the adic topology of $H^1_\fm(A(n))$ coincides with the quotient topology
under the canonical surjection $H^1_\fm(A(0)) \surj H^1_\fm(A(n))$, it follows
that the adic topology of
$\H^2_x(S, \sF_n) \cong \frac{H^1_\et(K, {\Z}/{p})}{\Fil_n H^1_\et(K, {\Z}/{p})}$
(cf. ~\eqref{eqn:F-seq-2}) coincides with the quotient (via $\beta'_n$)
of the adic topology of $\frac{H^1_\et(K, {\Z}/{p})}{\Fil_0 H^1_\et(K, {\Z}/{p})}$.
We endow $\frac{\Fil_n H^1_\et(K, {\Z}/{p})}{\Fil_0 H^1_\et(K, {\Z}/{p})}$ with the
quotient of the adic topology of $\frac{A(n)}{A}$. It is clear that all arrows in
the above diagram are continuous.

\begin{lem}\label{lem:Local-inj-10}
  There exists a continuous homomorphism $\gamma'_n \colon 
\frac{H^1_\et(K, {\Z}/{p})}{\Fil_0 H^1_\et(K, {\Z}/{p})} \to 
\frac{\Fil_n H^1_\et(K, {\Z}/{p})}{\Fil_0 H^1_\et(K, {\Z}/{p})}$ such that
$\gamma'_n \circ \alpha'_n = \id$.
\end{lem}
\begin{proof}
  We let $\gamma_n \colon \frac{K}{A} \to \frac{A(n)}{A}$ be the projection
  map $\stackrel{\infty}{\underset{i = 1}\bigoplus} \ff\pi^{-i} \to
  \stackrel{n}{\underset{i =1}\bigoplus} \ff\pi^{-i}$ so that 
$\gamma_n \circ \alpha_n = \id$. It is clear that
  $\gamma_n$ is continuous and $\ov{F} \circ \gamma_{\lfloor n/p\rfloor} = 
\gamma_n \circ \ov{F}$. Taking quotients by the images of $(1-\ov{F})$
and using \lemref{lem:Matsuda-0},
  we get that $\gamma_n$ descends to the map $\gamma'_n$, as desired.
 \end{proof}

Using \lemref{lem:Local-inj-10}, we get that the bottom row in the commutative 
diagram
\begin{equation}\label{eqn:Local-inj-5}
    \xymatrix@C1pc{
      0 \ar[r] &  \kappa^1_2(A|n+1) \ar[d]_-{\rho^{n+1}_{K,1}} \ar[r] &
      \kappa^1_2(A|1) \ar[r] \ar[d]^-{\rho^{1}_{K,1}}  &
      \varpi^1_2(A|n+1) \ar[r] \ar[d]^-{\wt{\rho}^{n+1}_{K,1}} & 0 \\
0 \ar[r] & \left(\frac{H^1_\et(K, {\Z}/{p})}{\Fil_n H^1_\et(K, {\Z}/{p})}\right)^\star
\ar[r] & \left(\frac{H^1_\et(K, {\Z}/{p})}{\Fil_0 H^1_\et(K, {\Z}/{p})}\right)^\star
\ar[r] &
\left(\frac{\Fil_n H^1_\et(K, {\Z}/{p})}{\Fil_0 H^1_\et(K, {\Z}/{p})}\right)^\star
\ar[r] & 0}
  \end{equation}
  is exact, where the duals are taken with respect to the adic topologies.
  The top row is exact by ~\eqref{eqn:Kato-0}.
It follows from Lemma~\ref{lem:LD-0},  ~\eqref{eqn:C-F-sequence}
and \lemref{lem:LD-3} that the left and the middle vertical arrows are
bijective. It follows that $\wt{\rho}^{n+1}_{K,1}$ is bijective.

If $K$ is not necessarily complete, we can use
\lemref{lem:Kato-equiv} and \corref{cor:p-tor-5} to make sense of the adic
topology on $\frac{\Fil_n H^1_\et(K, {\Z}/{p})}{\Fil_0 H^1_\et(K, {\Z}/{p})}$ and
conclude the following.

\begin{cor}\label{cor:Local-inj-11}
  The reciprocity homomorphism
  \[
    \wt{\rho}^{n+1}_{K,1} \colon  \varpi^1_2(A|n+1) \to
    \left(\frac{\Fil_n H^1_\et(K, {\Z}/{p})}{\Fil_0 H^1_\et(K, {\Z}/{p})}\right)^\star
  \]
  is bijective.
\end{cor}

We generalize the above result as follows.

\begin{cor}\label{cor:Local-inj-6}
   For $m \ge 1$ and $n \ge 0$, the map $\rho^n_{K,m}$
   induces a monomorphism
   \[
     \wt{\rho}^{n}_{K,m} \colon  \varpi^m_2(A|n) \inj     
\left(\frac{\Fil_{n-1} H^1_\et(K, {\Z}/{p^m})}{\Fil_{-1} H^1_\et(K, {\Z}/{p^m})}\right)^\star,
 \]
 where the dual on the right is taken with respect to the discrete topology.
\end{cor}
\begin{proof}
We can assume that $K$ is complete in view of
\lemref{lem:Kato-equiv} and \corref{cor:p-tor-5}.
We can also assume that $n \ge 1$ else both sides of $\wt{\rho}^{n}_{K,m}$ are zero.
  We now let $n' = \lfloor n/p \rfloor$ and look at the commutative diagram
 \begin{equation}\label{eqn:Local-inj-7}
    \xymatrix@C.8pc{ 
      0 \ar[r] & \varpi^{m-1}_2(A|n') \ar[r] \ar[d]_-{\wt{\rho}^{n}_{K,m-1}} &
    \varpi^{m}_2(A|n) \ar[r] \ar[d]^-{\wt{\rho}^{n}_{K,m}} &  
    \varpi^{1}_2(A|n) \ar[r] \ar[d]^-{\wt{\rho}^{n}_{K,1}} & 0 \\
    0 \ar[r] &
 \left(\frac{\Fil_{n'-1} H^1_\et(K, {\Z}/{p^{m-1}})}{\Fil_{-1}H^1_\et(K, {\Z}/{p^{m-1}})}\right)^\star \ar[r] & 
 \left(\frac{\Fil_{n-1} H^1_\et(K, {\Z}/{p^{m}})}{\Fil_{-1} H^1_\et(K, {\Z}/{p^m})}\right)^\star \ar[r] &
 \left(\frac{\Fil_{n-1} H^1_\et(K, {\Z}/{p})}{\Fil_{-1} H^1_\et(K, {\Z}/{p})}\right)^\star
 \ar[r] & 0.}
\end{equation}

The exactness of the top row follows by comparing the top row of
~\eqref{eqn:Local-inj-0} with the similar exact sequence for $n =0$.
The exactness of the bottom row follows by comparing the bottom row of
~\eqref{eqn:p-tor-2} with ~\eqref{eqn:Local-inj-1} (with $K$ replaced by $A$).
The left and the right vertical arrows in ~\eqref{eqn:Local-inj-7} are injective
by induction on $m$. It follows that $\wt{\rho}^{n}_{K,m}$ is injective.
\end{proof}

A stronger claim is true about the dual of $\wt{\rho}^{n}_{K,m}$ which we can 
easily deduce from what Kato already showed about $\rho^\star_K$.

%The following result is about the dual of $\wt{\rho}^{n}_{K,m}$.

\begin{prop}\label{prop:Dual-iso}
  For $m \ge 1$ and $n \ge 0$, the map
  \[
    (\wt{\rho}^{n}_{K,m})^\star \colon
\frac{\Fil_{n-1} H^1_\et(K, {\Z}/{p^m})}{\Fil_{-1} H^1_\et(K, {\Z}/{p^m})} 
    \to (\varpi^m_2(A|n))^\star
  \]
  is an isomorphism, where the dual on the right
  is taken with respect to the Kato topology.
\end{prop}
\begin{proof}
  We can assume that $K$ is complete.
  %as in \corref{cor:Local-inj-6}.
We consider the commutative diagram
  \begin{equation}\label{eqn:Dual-iso-0}
    \xymatrix@C.8pc{
      0 \ar[r] & \Fil_{-1} H^1_\et(K, {\Z}/{p^m}) \ar[r] \ar[d] &
      H^1_\et(K, {\Z}/{p^m})
      \ar[r] \ar[d] &
      \frac{H^1_\et(K, {\Z}/{p^m})}{\Fil_{-1} H^1_\et(K, {\Z}/{p^m})} 
      \ar[r] \ar[d]^-{(\rho^0_{K,m})^\star} & 0 \\
      0 \ar[r] & ({\ff^\times}/{p^m})^\star \ar[r] & ({K_2(K)}/{p^m})^\star \ar[r] &
      ({K_2(A)}/{p^m})^\star \ar[r] &  0,}
  \end{equation}
  where the vertical arrows are the duals of the reciprocity maps.
  The bottom row is exact by Remark~\ref{remk:K-dual-1}.
  The left vertical arrow is well known to be bijective.
  The middle vertical arrow is bijective
  by \cite[\S~8, Thm.~2]{Kato80-1} and \cite[\S~3, Rem.~4]{Kato80-2}.
  It follows that the right vertical arrow is bijective.
  In particular,  $(\wt{\rho}^{n}_{K,m})^\star$ is injective.

To prove that $(\wt{\rho}^{n}_{K,m})^\star$ is surjective,
let $\chi \in  (\varpi^m_2(A|n))^\star$. The surjectivity of
  $(\rho^0_{K,m})^\star$ implies that there exists an element $\chi' \in
  \frac{H^1_\et(K, {\Z}/{p^m})}{\Fil_{-1} H^1_\et(K, {\Z}/{p^m})}$ such that
  $(\rho^0_{K,m})^\star(\chi') = \chi$. We only have to show that
  $\chi' \in \frac{\Fil_{n-1} H^1_\et(K, {\Z}/{p^m})}
  {\Fil_{-1} H^1_\et(K, {\Z}/{p^m})}$. But this follows at from
  \cite[Rem.~6.6]{Kato89}.
  %because $\chi \in  (\varpi^m_2(A|n))^\star \subset  ({K_2(A)}/{p^m})^\star$.
\end{proof}

\subsection{Cokernel of $\rho_K$}\label{sec:KCKer}
We shall now prove the remaining parts of \thmref{thm:Main-1}.
We fix an integer $n$ not divisible by $p$.
We let $A = {\rm Image}(\rho_{K})$ and $B = \coker(\rho_{K})$. We also let
\[
M = \frac{K_2({K})}{\Fil_{1} K_2({K})}, \
  N = \frac{\Fil_0 K_2(K)}{\Fil_1 K_2(K)}, \
  M' = \frac{G_{K}}{G^{(1)}_{K}} \ \mbox{and} \ N' = \frac{G^{(0)}_{K}}{G^{(1)}_{K}}.
    \]

\begin{prop}\label{prop:Local-ker-coker-0}
  The map $~_n K_2(K) \to ~_n G_K$ is surjective.
\end{prop}
\begin{proof}
  We consider the commutative diagram
  \begin{equation}\label{eqn:Local-ker-coker-0-1}
    \xymatrix@C1pc{
      0 \ar[r] & ~_n \Fil_1 K_2(K) \ar[r]  & ~_n K_2(K)
      \ar[r] \ar[d]_-{\rho_{K}} & ~_n M \ar[r] \ar[d]^-{\rho'_{K}} & 0 \\
      & & ~_n G_{K} \ar[r]^-{\cong} & ~_n M'. &}
  \end{equation}
  The top row of this diagram is exact because $\Fil_1 K_2(K)$ is $n$-divisible
  (cf. \cite[Lem.~5.9]{GKR-arxiv} and its proof) and the bottom horizontal arrow
  is an isomorphism because $G^{(1)}_{K'}$ is a pro-$p$ group and hence uniquely
  $n$-divisible. We thus have to show that $\rho'_{K}$ is surjective.
  By \lemref{lem:Kato-equiv} and \corref{cor:p-tor-5}, we can assume that
  $K$ is complete.

We now look at the commutative diagram of exact sequences
   \begin{equation}\label{eqn:Local-ker-coker-0-2}
    \xymatrix@C1pc{ 
      0 \ar[r] &  N \ar[r] \ar[d]_-{\alpha_K} &
      M \ar[d]^-{\rho'_K} \ar[r] & K_1(\ff) \ar[r]
      \ar[d]^-{\rho_\ff} & 0 \\
      0 \ar[r] & N' \ar[r] & M' \ar[r] & G_\ff \ar[r] & 0,}
  \end{equation}
  where the vertical arrows are induced by the reciprocity maps.
  One knows that the top row is split exact (cf. the proof of \lemref{lem:K-dual}).

We thus get a commutative diagram with exact rows
   \begin{equation}\label{eqn:Local-ker-coker-0-3}
    \xymatrix@C1pc{ 
      0 \ar[r] &  ~_n N \ar[r] \ar[d]_-{\alpha_K} & ~_n M \ar[d]^-{\rho'_K} \ar[r] &
      ~_n K_1(\ff) \ar[r] \ar[d]^-{\rho_\ff} & 0 \\
      0 \ar[r] & ~_n N' \ar[r] & ~_n M' \ar[r] & ~_n G_\ff. & }
  \end{equation}
  Since the right vertical arrow is an isomorphism, we get that the map
  $ ~_n M' \to ~_n G_\ff$ in the bottom row is surjective.
Since $N \cong \divf(N) \bigoplus N_0$ and $\alpha_K({\divf(N)}) = 0$, where
  $\divf(N)$ is the maximal divisible subgroup of $N$, it suffices to show that
  the map $\alpha_K \colon ~_n N_0 \to ~_n N'$ is surjective.
  We shall show that $\alpha_K \colon N_0 \to N'$ is in fact bijective.
  
By \lemref{lem:Kato-adic}, the two rows of the commutative
  diagram     
 \begin{equation}\label{eqn:Local-ker-coker-0-4}
    \xymatrix@C1pc{ 
      0 \ar[r] & H^1(\ff) \ar[r] \ar[d] & H^1(K) \ar[r] \ar[d] &
      (G^{(0)}_K)^\star \ar[r] \ar[d] & 0 \\
      0 \ar[r] & K_1(\ff)^\star \ar[r] & K_2(K)^\star \ar[r] &
      (\Fil_0 K_2(K))^\star &}
  \end{equation}
  are exact with respect to the Kato and adic topologies.
  Since the left and the middle vertical arrows are bijective
  (cf. \cite[\S~3.5]{Kato80-2}), it follows that the right vertical arrow is
  injective. Combining this with the commutative diagram

\begin{equation}\label{eqn:Local-ker-coker-0-5}
    \xymatrix@C1pc{ 
      N'^\star \ar@{^{(}->}[r] \ar[d]_-{(\alpha_K)^\vee} & (G^{(0)}_K)^\star \ar[d] \\
      N^\star \ar[r] &  (\Fil_0 K_2(K))^\star,}
  \end{equation}
  we conclude that $(\alpha_K)^\vee \colon N'^\star \to  N_0^\vee
  \xrightarrow{\cong} N^\vee$ is injective.

On the other hand,
  $N' \cong \mu_\infty(\ff) \cong N_0$ are finite groups by \cite[Thm.~3]{Kato80}
  and \cite{Tate-Kyoto}, where $\mu_\infty(R) = R^\times_\fr$ for any ring $R$.
  It follows that $(\alpha_K)^\vee \colon N'^\vee \to
  N_0^\vee$ is an injective homomorphism between finite groups of the same
  cardinality. In particular, it is bijective. Taking the dual, we get that
  $\alpha_K \colon N_0 \to N'$ is bijective. This concludes the proof. 
\end{proof}

\begin{cor}\label{cor:Local-ker-coker-0-6}
The cokernel of $\rho_{K}\colon K_2(K) \to G_{K}$ is uniquely $n$-divisible.
\end{cor}
\begin{proof}
  It follows from the Merkurjev-Suslin isomorphism ${K_2(K)}/n \xrightarrow{\cong}
    H^2_\et(K, {\Z}/n(2))$, \lemref{lem:Kato-equiv} and
  \cite[Thm.~3.3]{Koya} (see also \cite[Thm.~2.17]{Milne-AD}) that the map
    ${\rho_{K}} \colon {K_2(K)}/n \to {G_{K}}/n$ is bijective.
    Combining this with \propref{prop:Local-ker-coker-0} and
    the exact sequence
\begin{equation}\label{eqn:Local-ker-coker-0-0}
  0 \to ~_n A \to ~_n G_{K} \to ~_n B \to A/n \to {G_{K}}/n \to B/n \to 0,
\end{equation}
we conclude the proof.
\end{proof}

We end this section with the following result.

\begin{prop}\label{prop:Coker-p}
  The cokernel of $\rho_{K}\colon K_2(K) \to G_{K}$ has no $p$-torsion but it is
  not $p$-divisible.
\end{prop}
\begin{proof}
  The first claim of the proposition follows directly from \lemref{lem:No-p-tor}
  and \corref{cor:Local-inj-2}.
  To prove the second claim, it suffices to show, using the bijectivity of
  $\rho_\ff$ with finite coeffcients, \lemref{lem:Kato-1} and the diagram
  ~\eqref{eqn:Local-inj-5} (with $(-)^\star$ in place of $(-)^\vee$ in the
  bottom row), that
  $\wt{\rho}^2_{K,1}\colon \varpi^1_2(A|2) \to
  \left(\frac{\Fil_1 H^1_\et(K, {\Z}/p)}{\Fil_0 H^1_\et(K, {\Z}/{p})}\right)^\vee$ is
  not surjective.
  To show the latter claim, we write
  $F = \frac{\Fil_1 H^1_\et(K, {\Z}/p)}{\Fil_0 H^1_\et(K, {\Z}/{p})}$.
  By \corref{cor:Local-inj-11} and \cite[\S~2, Prop.~1]{Kato80-1},
  we get that $\ff  \xrightarrow{\cong}  \varpi^1_2(A|2) \xrightarrow{\cong}
  F^\star$.  On the other hand, it follows from
  \cite[Cor.~3.3]{Kato89} that $F \cong \ff$. We are now done because  an
  elementary cardinality argument tells us that no homomorphism 
  $\ff \to \ff^\vee$ can be surjective.
\end{proof}

 \section{Reciprocity map for curves}\label{sec:RCurve}
 Our objective in the rest of this paper is to study the class field theory of
 smooth curves over local fields of positive characteristics.
 For all of this section, we fix a local field $k$ of characteristic $p > 0$.
We also fix a projective curve $X$ over $k$ which is normal and connected.
We let $K$ denote the function field of $X$. By a modulus pair, we shall mean
a pair $(X,D)$, where $D \subset X$ is an effective Cartier divisor.

For $x \in X_{(0)}$, we let $A_x = \sO_{X,x}$ with the maximal ideal 
$\fm_x$. We shall write $X_x = \Spec(A_x), \ X^h_x = \Spec(A^h_x)$ and
$\wh{X}_x = \Spec(\wh{A}_x)$. We let $K^h_x$ (resp. $\wh{K}_x$) denote the
quotient field of $A^h_x$ (resp. $\wh{A}_x$). Given an effective Weil divisor
$D = \stackrel{s}{\underset{i =1}\sum} n_i[x_i] \in \Div(X)$ with complement $U$ and
$x \in X_{(0)}$, we shall let $n_x = n_i$ if $x = x_i$ for some
$i \in \{1, \ldots , s\}$ and $n_x = 0$ if $x \notin D$. 
We let 
\begin{equation}\label{eqn:Curve-notn}
  \Fil_D K_2(K) = {\underset{x \in D}\bigcap} K_2(A_x|\fm^{n_x}_x), \
  F_0(U) = {\underset{x \in U}\bigoplus} \Z, \
  F_1(U) = {\underset{x \in U}\bigoplus} k(x)^\times, \
\end{equation}
\[
F_2(U) = {\underset{x \notin U}\bigoplus} {K_2(K^h_x)}, \
F^c_2(U) = {\underset{x \notin U}\bigoplus}
{K_2(\wh{K}_x)} \ \ \mbox{and}
\]
\[
F_2(X|D) = {\underset{x \in D}\bigoplus}
  \frac{K_2({K}^h_x)}{\Fil_{n_x} K_2({K}^h_x)}
  \xrightarrow{\cong} {\underset{x \in D}\bigoplus}
  \frac{K_2(\wh{K}_x)}{\Fil_{n_x} K_2(\wh{K}_x)},
\]
where the isomorphism on the right follows from \lemref{lem:Kato-equiv}.

\subsection{The idele class groups}\label{sec:ICG}
Let $D \subset X$ be an effective divisor with complement $U$.
Recall (cf. \cite{Kato86}) that there is a tame symbol map
$\partial_U \colon K_{n+1}(K) \to  {\underset{x \in U_{(0)}}\bigoplus} K_n(k(x))$.
We let  $\vartheta_{X|D}$ denote the composition of the diagonal map
$K_{n}(K) \xrightarrow{\vartheta_{U}}
{\underset{x \in D}\bigoplus} K_n({K}^h_x)$ with the quotient map
${\underset{x \in D}\bigoplus} K_n({K}^h_x) \surj 
{\underset{x \in D}\bigoplus}
\frac{K_n({K}^h_x)}{\Fil_{n_x} K_n({K}^h_x)}$.

\begin{defn}\label{defn:CG-1}$($cf. \cite{Hiranouchi-2}$)$
  The idele class group $C(X,D)$ of the modulus pair $(X,D)$ is defined to be the
  cokernel of the map $K_2(K) \xrightarrow{(\partial_{U}, \vartheta_{X|D})}
  F_1(U) \bigoplus F_2(X|D)$. We let $C(X) = C(X, \emptyset)$.
\end{defn}

For each $x \in X_{(0)}$, we endow $\frac{K_2({K}^h_x)}{\Fil_n K_2({K}^h_x)}$ with
the quotient of the Kato topology of $K_2({K}^h_x)$.
The Kato topology of $F_1(U) \bigoplus F_2(X|D)$ is the direct sum
topology with respect to the Kato topology on each factor.
We let $C(X,D)$ be endowed with the quotient of the topology of
$F_1(U) \bigoplus F_2(X|D)$. This makes $C(X,D)$ into a topological
abelian group.

The following description of $C(X,D)$ follows from \cite[Lem.~5.6]{GKR-arxiv}.

\begin{lem}\label{lem:ICG-PC}
  There is an exact sequence
  \[
  \Fil_D K_2(K) \xrightarrow{\partial_{U}} F_1(U) \to C(X,D) \to 0.
\]
\end{lem}

Let $\CH^q(X|D, p)$ (resp. $H^q_\tau(X|D, \Z(p))$)
denote the higher Chow group (resp. motivic cohomology with respect to the
topology $\tau$) of the modulus pair $(X,D)$ (cf. \cite[\S~1]{Rulling-Saito}).

\begin{cor}\label{cor:ICG-BF}
  For  $\tau \in \{\zar, \nis\}$, we have the following.
  \begin{enumerate}
  \item
    There are canonical maps
  \[
    \CH^2(X|D,1) \surj H^3_\tau(X|D, \Z(2)) \xrightarrow{\cong}
    C(X,D) \xrightarrow{\cong} H^1_{\tau}(X, \sK_{2, X|D}). 
  \]
\item
 The canonical maps
 \[
   {C(X,D)}/{p^m} \to H^1_\tau(X, {\sK_{2, X|D}}/{p^m}) \to
    H^1_\tau(X, \sK^m_{2, X|D})
\]
are isomorphisms.
\end{enumerate}
\end{cor}
\begin{proof}
The left arrow in (1) is the well known canonical map.
  The claim that
  the right arrow is an isomorphism is an easy consequence of
  \lemref{lem:ICG-PC} using the support spectral sequence for
  $H^1_{\tau}(X, \sK_{2, X|D})$.
  The claim that the composition of the middle and the right arrows is an
  isomorphism is shown in \cite[Thm.~1]{Rulling-Saito}. Finally, it is
  easy to check that the composition of the left and the middle arrows is
  surjective because it is induced by the norm map
  $z^2(X|D, 1) \to F_1(X \setminus D)$, which has a section since $k$ is infinite.
  The proof of (2) is identical to that of (1).
\end{proof}

\begin{remk}\label{remk:ICG-BF-0}
  It follows from \lemref{lem:Fil-M-0} and \cite[Prop.~6.5]{GKR-arxiv} that
  the maps in \corref{cor:ICG-BF} induce an isomorphism between the pro-abelian
  groups (compare with \cite[Thm.~2]{Rulling-Saito})
  \[
    \Phi_X \colon \{\CH^2(X|nD,1)\}_{n \ge 1} \xrightarrow{\cong}
    \{C(X,nD)\}_{n \ge 1}.
  \]
\end{remk}

\vskip .2cm

\begin{defn}\label{defn:ICG-open}$($cf. \cite[defn.~3.2]{Hiranouchi-2}$)$
  The idele class group $C(U)$ of an open curve $U \subset X$ is defined to be the
cokernel of the canonical map $K_2(K) \xrightarrow{(\partial_{U}, \vartheta_{U})}
F_1(U) \bigoplus F_2(U)$.
We let $C^c(U)$ be the cokernel of the canonical map $K_2(K)
\xrightarrow{(\partial_U, \vartheta_U)} F_1(U) \bigoplus F^c_2(U)$.
The class groups $C(U)$ and $C^c(U)$ are topological abelian groups endowed with
the quotients of the Kato topology of
$F_1(U) \bigoplus F_2(U)$ and $F_1(U) \bigoplus F^c_2(U)$,
  respectively.
We let $\wh{C}(U) = {\varprojlim}_D \ C(X,D)$ with the
inverse limit topology, where the limit is taken over all effective divisors
having support disjoint from $U$.
\end{defn}

Note that the above class groups depend only on $U$ since
  the latter admits a unique regular compactification up to isomorphism.
  There are continuous homomorphisms
  \begin{equation}\label{eqn:ICG-open-0}
    {C(U)}/n \inj {C^c(U)}/n \xrightarrow{\theta_U} {\wh{C}(U)}/n
    \xrightarrow{\phi_{X|D}} {C(X,D)}/n
\end{equation}
for every integer $n \ge 0$ and effective divisor $D$ supported away from $U$.
The first arrow is injective because
${F_2(U)}/n \inj {F^c_2(U)}/n$ by \lemref{lem:Kato-equiv}.
 All arrows in this sequence are isomorphisms if $p \nmid n$.

\subsection{{\'E}tale idele class group}\label{sec:EICG}
 Let $D = \sum_x n_x [x] \in \Div(X)$ be an effective divisor with complement $U$.
 We fix an integer $m \ge 1$. We let $C_m(X,D) = {C(X,D)}/{p^m}$
  and $C^\et_m(X,D) = H^1_\et(X, \sK^m_{2, X|D})$. We let
    $T_m(X)$ be the cokernel of the canonical map $C_m(X) \to C^\et_m(X)$. It follows
    from the Leray spectral sequence associated to the morphism of sites
    $\epsilon \colon X_\et \to X_\nis$ that $T_m(X) \cong
    H^0_\nis(X, {\bf R}^1\epsilon_*(W_m\Omega^2_{X, \log}))$.

\begin{lem}\label{lem:Surj-TM}
      For $D' \ge D$, the canonical map $C^\et_m(X,D') \to  C^\et_m(X,D)$ is
      surjective.
    \end{lem}
    \begin{proof}
      It suffices to show that the cokernel of the inclusion
      $\sK^m_{2, X|D'} \inj \sK^m_{2, X|D}$ is a coherent sheaf
      supported on $D'$. In view of \lemref{lem:Kato-1} (2), we can assume that
       $D'$ has the same support as that of $D$. The coherence claim then 
      follows from \cite[Prop.~1.1.9]{JSZ}.
\end{proof}

\begin{prop}\label{prop:Nis-et-ICG}
  The change of topology map induces a short exact sequence
  \[
    0 \to C_m(X,D) \xrightarrow{\gamma_{X|D}}  C_m^\et(X,D) \to T_m(X) \to 0.
    \]
  \end{prop}
  \begin{proof}
Using the Leray spectral sequence associated to
    $\epsilon$ and \corref{cor:ICG-BF}, we get an exact sequence
    \begin{equation}\label{eqn::Nis-et-ICG-7}
      0 \to C_m(X,D) \xrightarrow{\gamma_{X|D}}  C_m^\et(X,D) \to
      H^0_\nis(X, {\bf R}^1\epsilon_*(\sK^m_{2, X|D})) \to 0.
    \end{equation}
    It suffices therefore to show that the canonical map
    ${\bf R}^1\epsilon_*(\sK^m_{2, X|D}) \to {\bf R}^1\epsilon_*(W_m\Omega^2_{X, \log})$ is
    bijective. Equivalently, $H^1_\et(X^h_x, \sK^m_{2, X|D}) \to
    H^1_\et(X^h_x, W_m\Omega^2_{X,\log})$ is bijective for every $x \in X$.
    Clearly, it suffices to show this bijection when $x \in D$.
We shall in fact show that both these groups are zero in this case.

We first note that the exact sequence
      \begin{equation}\label{eqn:Nis-et-ICG-1}
        0 \to \sK^m_{2,X^h_x|D^h_x} \to \sK_{2, X^h_x}/{p^m} \to \sG_x \to 0
        \end{equation}
        on $(X^h_x)_\et$ yields an exact sequence
        \begin{equation}\label{eqn:Nis-et-ICG-1*}
          {K_2(A^h_x)}/{p^m} \to H^0(X^h_x, \sG_x) \to
        H^1(X^h_x, \sK^m_{2,X^h_x|D^h_x}) \to H^1(X^h_x, {\sK_{2, X^h_x}}/{p^m}).
      \end{equation}
Since ~\eqref{eqn:Nis-et-ICG-1} is exact also in the Nisnevich
      topology, we get
      $H^0(X^h_x, \sG_x) \cong H^0_\nis (X^h_x, \sG_x) =
      \frac{K_2(A^h_x)}{\kappa^m_2(A^h_x|n_x) +
        p^mK_2(A^h_x)}$. It follows that the left arrow in
      ~\eqref{eqn:Nis-et-ICG-1*} is surjective.
      We can now conclude the proof using \lemref{lem:Kato-equiv}.
\end{proof}

\subsection{Fundamental group with modulus}\label{sec:FMP}
Let $D = \sum_x n_x [x] \in \Div(X)$ be an effective divisor with complement $U$.
Let $D_\dagger = \{x_1, \ldots , x_s\}$ be the support of $D$ with the reduced closed
subscheme structure.
%We let $D' = D - D_\dagger$.
We let 
\begin{equation}\label{eqn:Fil-U-p}
  \Fil_D H^1_{\et}(K, {\Z}/{p^m}) = \Ker\left(H^1_\et(U, {\Z}/{p^m}) \to
  {\underset{x\in D_\dagger}\bigoplus} 
  \frac{H^1_\et({K}^h_x, {\Z}/{p^m})}{\Fil_{n_x -1}H^1_\et({K}^h_x, {\Z}/{p^m})}
  \right).
      \end{equation}
\begin{equation}\label{eqn:Fil-U-p-0}
  \Fil_D H^1_{\et}(K, {\Q_p}/{\Z_p}) = \Ker\left(H^1_\et(U, {\Q_p}/{\Z_p}) \to
  {\underset{x\in D_\dagger}\bigoplus} 
  \frac{H^1_\et({K}^h_x, {\Q_p}/{\Z_p})}{\Fil_{n_x -1}
    H^1_\et({K}^h_x, {\Q_p}/{\Z_p})}
  \right).
      \end{equation}

      Note that $\Fil_0 H^1_{\et}(K, A) \cong  H^1_{\et}(X, A)$ for
      $A \in \{{\Z}/{p^m}, {\Q_p}/{\Z_p}\}$.
As a consequence of \propref{prop:p-tor}, we get the following.
      \begin{cor}\label{cor:p-tor-*}
        All maps in the commutative diagram
        \begin{equation}\label{eqn:p-tor-*-0}
          \xymatrix@C.8pc{
            \Fil_D H^1_{\et}(K, {\Z}/{p^m})  \ar[r] \ar[d] &
            \Fil_D H^1_{\et}(K, {\Q_p}/{\Z_p}) \bigcap H^1_{\et}(U, {\Z}/{p^m})
            \ar[d] \\
            ~_{p^m} \Fil_D H^1_{\et}(K, {\Q_p}/{\Z_p}) \ar[r] &
            \Fil_D H^1_{\et}(K, {\Q_p}/{\Z_p}) \bigcap H^1_{\et}(K, {\Z}/{p^m})}
        \end{equation}
        are bijections.
      \end{cor}

\begin{defn}\label{defn:FGMP}
        We let $\Fil_D H^1(K) = H^1(U)\{p'\} \bigoplus
        \Fil_D H^1_{\et}(K, {\Q_p}/{\Z_p})$ and $\pi^{\ab}_1(X,D) =
        \Hom_\Ab(\Fil_{D} H^1(K), {\Q}/{\Z})$. We shall consider
        $\Fil_D H^1(K)$ as a discrete topological abelian group so that
        $\pi^{\ab}_1(X,D) \cong (\Fil_{D} H^1(K))^\star$ is a profinite topological
        abelian group.
      \end{defn}

By \cite[Lem.~2.8]{GKR-arxiv}, there are canonical quotient homomorphisms
      of profinite groups
\begin{equation}\label{eqn:Adiv}
  \pi^{\ab}_1(K) \surj \pi^{\ab}_1(U) \surj \pi^{\ab}_1(X,D) \surj \pi^{\ab, \tm}_1(U)
  \surj \pi^{\ab}_1(X),       
\end{equation}
where $\pi^{\ab, \tm}_1(U)$ is the abelianized tame {\'e}tale fundamental group of
$U$ (cf. \cite[\S~2]{GKR-arxiv}). Furthermore, an argument similar to that
of \cite[Thm.~7.17]{Gupta-Krishna-REC} shows that $\pi^{\ab}_1(X,D)$ is the
automorphism group of the fiber functor associated to the Galois category
consisting of finite abelian covers of $U$ whose logarithmic ramifications (in the
sense of Abbes-Saito \cite{Abbes-Saito}) at a point $x \in D_\dagger$ are bounded by
$n_x-1$ (cf. \cite[Defn.~7.2]{Gupta-Krishna-REC}).

By combining \corref{cor:p-tor-*} with \cite[Lem.~2.2]{Kato89}, we obtain the
      following.
  \begin{cor}\label{cor:p-tor-*-1}    
    We have topological isomorphisms
    \[
      H^1(U) = {\underset{\supp(D) = D_\dagger}\bigcup} \ \Fil_D H^1(K)
      \ \ \mbox{and} \ \
      H^1_{\et}(U, {\Z}/{p^m})  = {\underset{\supp(D) = D_\dagger}\bigcup} \
      \Fil_D H^1_{\et}(K, {\Z}/{p^m}).
    \]
  \end{cor}

\subsection{Reciprocity homomorphism}\label{sec:RMap}
We keep the notations of \S~\ref{sec:FMP}.
For $x \in U_{(0)}$, let $\rho^x_{X|D}$ denote the composite map 
$k(x)^\times \xrightarrow{\rho_{k(x)}} G_{k(x)} \xrightarrow{(\iota_x)_*}
\pi^{\ab}_1(U) \surj \pi^{\ab}_1(X,D)$, where $\rho_{k(x)}$ is the classical
reciprocity for the local field $k(x)$.
For $x \in D_\dagger$, we have the composite map
$\rho^x_{X|D}\colon K_2({K}^h_x)  \xrightarrow{\rho_{{K}^h_x}} G_{{K}^h_x}
\xrightarrow{(j_{{K}^h_x})_*} \pi^{\ab}_1(U) \surj \pi^{\ab}_1(X,D)$,
where $\rho_{{K}^h_x}$ is defined in \S~\ref{sec:RHom}.
Taking the sum over $x \in U_{(0)}$ and
$x \in D_\dagger$, we get continuous homomorphisms
$\rho_{U} \colon F_1(U) \bigoplus F_2(U) \to \pi^{\ab}_1(U)$ and
$\rho_{X|D} \colon F_1(U) \bigoplus F_2(X|D) \to \pi^{\ab}_1(X,D)$.
It is shown in \cite[\S~3]{Hiranouchi-2} that $\rho_U$ and $\rho_{X|D}$
factor through compatible continuous homomorphisms
\begin{equation}\label{eqn:Rec-curve-Main}
\rho_U \colon C(U) \to  \pi^{\ab}_1(U); \ \  \ \rho_{X|D} \colon C(X,D) \to
\pi^{\ab}_1(X,D).
\end{equation}

 \corref{cor:p-tor-*-1} implies that
$\rho_U$ factors through $C(U) \to C^c(U) \to \wh{C}(U) \to \pi^{\ab}_1(U)$.
Let $\pi^{\ab}_1(X,D)_0$ denote the kernel of the composite map
$\pi_{X|D} \colon \pi^{\ab}_1(X,D) \surj \pi^{\ab}_1(X) \xrightarrow{\pi_X}
G_k$, where $\pi_X$ is the push-forward map induced by the
structure map $X \to \Spec(k)$. We let $C(U)_0$ (resp. $\pi^{\ab}_1(U)_0$) denote
the kernel of the composite map $C(U) \to C(X) \to k^\times$ (resp.
$\pi^{\ab}_1(U) \to \pi^{\ab}_1(X) \to G_k$). One defines $C(X,D)_0$ and
$\pi^{\ab}_1(X,D)_0$ similarly.

\begin{cor}\label{cor:REC-T-Norm}
  There are commutative diagrams
  \begin{equation}\label{eqn:REC-T-Norm-0}
    \xymatrix@C1pc{
  C(X,D)_0 \ar[r] \ar[d]_-{\rho_{X|D}} & C(X,D) \ar[r]^-{N_{X|D}}
  \ar[d]^-{\rho_{X|D}} & k^\times \ar[d]^-{\rho_k} & C(U)_0 \ar[r] \ar[d]_-{\rho_U}
  & C(U) \ar[r]^-{N_U} \ar[d]^-{\rho_U} & k^\times \ar[d]^-{\rho_k} \\
  \pi^{\ab}_1(X,D)_0 \ar[r] & \pi^{\ab}_1(X,D) \ar[r]^-{\pi_{X|D}} & G_k &
  \pi^{\ab}_1(U)_0 \ar[r] & \pi^{\ab}_1(U) \ar[r]^-{\pi_U} & G_k.}
  \end{equation}
\end{cor}
\begin{proof}
  We only need to show that the right squares in both diagrams are commutative.
  But this is a simple consequence of our constructions and
  \cite[\S~3.1, Cor.~1, Cor.~2]{Kato80-2}.
\end{proof}

\section{Logarithmic Hodge-Witt cohomology with
  modulus}\label{sec:LPair}
Our goal in the next two sections is to prove \thmref{thm:Main-4}.
In this section, we shall
construct a pairing between the {\'e}tale cohomology of
logarithmic Hodge-Witt sheaves with modulus (cf. \propref{prop:Pair-0-4}).
We begin by recalling some definitions.

\subsection{Twisted de Rham-Witt sheaves}\label{sec:TDR}
Let $X$ be a connected regular $\F_p$-scheme of dimension one and let $E \subset X$
be a reduced effective divisor on $X$.  For a divisor
$D = {\underset{x \in E}\sum} n_x [x]$ on $X$ with support $E$,
we let $W_m \sO_X(D)$ be the Nisnevich
(resp.  {\'e}tale) sheaf on $X$ which is locally defined as
$W_m \sO_X(D)  = [f^{-1}] \cdot W_m \sO_X \subset j_* (W_m\sO_U)$,
where $[-] \colon \sO_X \to W_m\sO_X$ is the Teichm{\"u}ller function,
$f\in \sO_X$ is a local equation of $D$, and $j \colon U \inj X$ is the
inclusion of the complement of $E$ in $X$. One knows that $W_m \sO_X(D)$
is a sheaf of invertible $W_m\sO_X$-modules. We let $W_m\Omega^\bullet_X(D) =
W_m\Omega^\bullet_X \otimes_{W_m\sO_X} W_m\sO_X(D) \subset j_* (W_m\Omega^\bullet_U)$.

We let $W_m\Omega^r_{X|D} = W_m\Omega^r_X(\log E)(-D)$, where
$W_m\Omega^\bullet_X(\log E)$ is the logarithmic de-Rham-Witt complex
{\`a} la Hyodo-Kato \cite{Hyodo-Kato}. Let $Z_1W_m\Omega^r_X =
{\rm Image}(F \colon W_{m+1}\Omega^r_X \to W_m\Omega^r_X) =
\Ker(F^{m-1}d \colon W_m\Omega^r_X \to \Omega^{r+1}_X)$.
Let $Z_1W_m\Omega^r_X(D) = j_*(Z_1W_m\Omega^r_U) \cap W_m\Omega^r_X(D)$
and $Z_1W_m\Omega^r_{X|D} = j_*(Z_1W_m\Omega^r_U) \cap W_m\Omega^r_{X|D}$.
By \cite[Prop.~3.16]{Gupta-Krishna-Duality}, there is a Cartier map 
$C \colon Z_1W_m\Omega^r_X(pD) \to W_m\Omega^r_X(D)$
which is surjective and whose kernel is $dV^{m-1}\Omega^{r-1}_X(p^mD)$.
Furthermore, $C$ is the unique morphism such that $\p \circ C = V$, where
$\p \colon W_m\Omega^r_X(D) \inj W_{m+1}\Omega^r_X(D)$ is a monomorphism such that
$\p \circ R$ is multiplication by $p$ (cf. \cite[Prop.~1.1.4]{Zhao},
\cite[Lem.~3.14]{Gupta-Krishna-Duality}).
Let $\ov{F} \colon W_m\Omega^r_X \to \frac{W_m\Omega^r_X}{dV^{m-1}\Omega^{r-1}_X}$ be
the absolute Frobenius (cf. ~\eqref{eqn:DRC-3})
such that $\pi \circ F = \ov{F} \circ R$, where
$\pi \colon  {W_m\Omega^r_X} \to \frac{W_m\Omega^r_X}{dV^{m-1}\Omega^{r-1}_X}$ is the
quotient map.

Assume now that $D$ is effective and let $K$ denote the function field of $X$.
For $x \in X_{(0)}$, let $\iota^x \colon
\Spec({K}^h_x) \to X$ be the canonical morphism.  We let $\Fil_D W_m \sO_X$ be
the Nisnevich (resp.  {\'e}tale) sheaf on $X$ defined by
\begin{equation}\label{eqn:Log-fil-2}
  \begin{array}{lll}
  \Fil_D W_m \sO_X & = & \Ker\left(j_* W_m\sO_U \to
    \bigoplus_{x \in E}

  \frac{\iota^x_* W_m({K}^h_x)}{\iota^x_* \Fil_{n_x} W_m({K}^h_x)}\right) \\
    & = & \Ker\left(j_* W_m\sO_U \to
    \bigoplus_{x \in E}

          \iota^x_*(\frac{W_m({K}^h_x)}{\Fil_{n_x} W_m({K}^h_x)})\right).
          \end{array}
  \end{equation}

We let $Z_1 \Fil_D W_m\sO_X = j_*(Z_1W_m\sO_U) \cap \Fil_D W_m \sO_X$.
  Since $C \colon Z_1W_m \sO_U \to W_m \sO_U$ is given by
  $C(a^p_{m-1}, \ldots , a^p_0) = (a_{m-1}, \ldots , a_0)$,
  % (this follows easily by the fact that $C$ is the inverse of $\ov{F}$  as one checks from \cite[(11.2)]{Gupta-Krishna-Duality}).
  it follows that $C$ preserves $Z_1 \Fil_D W_m\sO_X$ and one gets a
  morphism $C \colon Z_1 \Fil_D W_m\sO_X \to \Fil_D W_m\sO_X$.
The following is easy to verify using the formula
  $[f]\un{a} = (fa_{m-1}, f^pa_{m-2}, \ldots , f^{p^{m-1}}a_0) \in W_m(K)$.

  \begin{lem}\label{lem:Fil-comp}
    We have the identity $\Fil_D W_1\sO_X = W_1 \sO_X(D)$. In general, we have
    \[
      \Fil_D W_m\sO_X \subseteq W_m \sO_X(D) =   \Fil_{p^{m-1}D} W_m\sO_X.
    \]
    \end{lem}

\subsection{Complexes of sheaves and their pairings}\label{sec:Exist}
For the rest of \S~\ref{sec:LPair}, we fix a local field $k$ of characteristic
$p > 0$ and a connected, smooth and projective curve $X$ over $k$.
We let $K$ be the function field of $X$. We also fix a positive integer $m$.
Let $D_\dagger = \{x_1, \ldots, x_s\}$ be a
reduced effective Cartier divisor on $X$ and let $X^o \subset X$ be
the complement of $D_\dagger$. For $\un{n} = (n_1, \ldots, n_s)
\in \Z^s$, we let $D_{\un{n}} = \stackrel{s}{\underset{i =1}\sum} n_i[x_i] \in
\Div(X)$. For $\un{n} \in \Z^s$ and $q \in \Z$, we let
$\un{n}_+ = \un{n} + \un{1}:= 
(n_1 +1, \ldots , n_s +1), \ \un{n}_- = \un{n} - \un{1}:= 
(n_1 -1, \ldots , n_s -1), \ q\n = (qn_1, \ldots , qn_s)$ and
$\<{\n}\> = p^{m-1}\n$.
We shall say that $\n \ge \n'$ if $n_i \ge n'_i$ for each $i$.
We shall write $\n \ge 0$ if $n_i \ge 0$ for each $i$. 
 For any divisor $D = \sum_y m_y [y]$ on $X$, we let
 $D/p = \sum_y \lfloor {m_y}/p \rfloor [y]$.
 We let ${\n}/p = (\lfloor {n_1}/p \rfloor, \ldots , \lfloor {n_r}/p \rfloor)$.

For $r \ge 0$ and $\un{n} \ge 0$, we consider the following complexes
in $D_\et(X)$.
\begin{equation}\label{eqn:Comp-0}
W_m\sF^r = \left[Z_1W_m\Omega^r_{X} \xrightarrow{1 -C}
W_m\Omega^r_{X}\right]; \ \ 
W_m\sM^r_{\un{-n}} = \left[Z_1W_m\Omega^r_{X|D_{\un{-n}}} \xrightarrow{1 -C}
W_m\Omega^r_{X|D_{\un{-n}}}\right].
\end{equation}
\begin{equation}\label{eqn:Comp-1}
W_m\sG^r = \left[W_m\Omega^r_{X} \xrightarrow{1 - \ov{F}} 
\frac{W_m\Omega^r_{X}}{dV^{m-1}\Omega^{r-1}_{X}}\right]; \ \ 
W_m\sG^r_{\un{n}} = \left[W_m\Omega^r_{X|D_{\un{n}}} \xrightarrow{1 - \ov{F}} 
\frac{W_m\Omega^r_{X|D_{\un{n}}}}{dV^{m-1}\Omega^{r-1}_{X|D_{\<\n\>}}}\right].
\end{equation}
\begin{equation}\label{eqn:Comp-2}
W_m\sH = \left[W_m\Omega^2_X \xrightarrow{1-C} W_m\Omega^2_X\right] 
\cong W_m\Omega^2_{X, \log}.
\end{equation}

The differential graded algebra structure of $W_m\Omega^\bullet_X(p^iD)$
induces a commutative diagram of cup product pairings of complexes
(cf. \cite[Lem.~3.1.1]{JSZ})
\begin{equation}\label{eqn:Comp-3}
  \xymatrix@C.8pc{
  W_m\sG^r \times W_m\sF^{2-r}  \ar@<7ex>[d]  \ar[r]^-{\cup} & W_m\sH \ar@{=}[d] \\
  W_m\sG^r_{\n_+} \times W_m\sM^{2-r}_{-\n}  \ar@<7ex>[u]
  \ar[r]^-{\cup} & W_m\sH.}
\end{equation}

By \cite[Prop.~4.4]{KRS} (see also \cite[\S~3, Prop.~4]{Kato-Saito-Ann}),
there is a bijective trace homomorphism
$\Tr_X \colon H^2_\et(X, W_m\Omega^2_{X, \log}) \xrightarrow{\cong} {\Z}/{p^m}$.
Using this trace map, ~\eqref{eqn:Comp-3}
gives rise to a commutative diagram of cohomology pairings

\begin{equation}\label{eqn:Comp-4}
  \xymatrix@C.8pc{
    \H^i_\et(X, W_m\sG^r) \times \H^{2-i}_\et(X, W_m\sF^{2-r})  \ar@<7ex>[d]
    \ar[r]^-{\cup} & {\Z}/{p^m} \ar@{=}[d]  \\
  \H^i_\et(X, W_m\sG^r_{\n_+}) \times \H^{2-i}_\et(X, W_m\sM^{2-r}_{-\n})  \ar@<7ex>[u]
  \ar[r]^-{\cup} & {\Z}/{p^m}.}
\end{equation}

Since the canonical map $H^2_x(X, W_m\Omega^2_{X, \log}) \to
H^2_\et(X, W_m\Omega^2_{X, \log})$
is an isomorphism for any $x \in X_{(0)}$ (cf. \cite[Prop.~4.4]{KRS}),
~\eqref{eqn:Comp-3} also yields a commutative diagram of cohomology pairings
(cf. \cite[\S~3.1]{Zhao})
\begin{equation}\label{eqn:Comp-5}
  \xymatrix@C.8pc{
    \H^i_\et(X^h_x, W_m\sG^r) \times \H^{2-i}_x(X^h_x, W_m\sF^{2-r})  \ar@<7ex>[d]
    \ar[r]^-{\cup} & {\Z}/{p^m} \ar@{=}[d] \\
    \H^i_\et(X^h_x, W_m\sG^r_{\n_+}) \times \H^{2-i}_x(X^h_x, W_m\sM^{2-r}_{-\n})
    \ar@<7ex>[u]
  \ar[r]^-{\cup} & {\Z}/{p^m},}
\end{equation}
which is compatible with ~\eqref{eqn:Comp-4}.

In order to prove our duality theorem,
we need to modify the bottom complexes in ~\eqref{eqn:Comp-3} and the induced
cohomology pairings. We first claim the following.

\begin{lem}\label{lem:G-comp}
  For any $\n \ge 0$, there is a canonical exact sequence of {\'e}tale sheaves
\begin{equation}\label{eqn:G-comp-0}
  0 \to W_m\Omega^r_{X|D_{\<\n\>}, \log} \to W_m\Omega^r_{X|D_{\un{n}}}
  \xrightarrow{1 - \ov{F}} 
\frac{W_m\Omega^r_{X|D_{\un{n}}}}{dV^{m-1}\Omega^{r-1}_{X|D_{\<\n\>}}}.
\end{equation}
\end{lem}
\begin{proof}
The lemma follows by comparing ~\eqref{eqn:G-comp-0} with the known case $\n = 0$ 
and using the proof of \cite[Lem.~2.3.1]{JSZ} if we show that 
$(1 - {F})(W_m\Omega^r_{X|D_{\<\n\>}, \log}) = 0$.
The latter statement follows from the stronger claim that
$(1- {F})(W_m \Omega^r_{X, \log}) = 0$. But this can be easily checked
(cf. \cite[Prop.~2.6]{Lorenzon}).
\end{proof}

\begin{cor}\label{cor:G-comp-1}
  We can replace $W_m\sG^r_{\n_+}$ by $W_m\Omega^r_{X|D_{\<{\n_+}\>}, \log}$
  in ~\eqref{eqn:Comp-3}.
\end{cor}

We next consider the complex
$W_m\Omega^0_{X|D,\log} = \left[Z_1 \Fil_{D} W_m \sO_X \xrightarrow{1-C}
  \Fil_{D} W_m \sO_X\right]$ for any effective divisor $D \subset X$.
We let $W_m\sF^0_\n = W_m\Omega^0_{X|D_\n,\log}$ for $\n \ge 0$.
It follows from \lemref{lem:Fil-comp} that there
is a term-wise isomorphism of complexes $W_m\sM^0_{-\un{n}} \xrightarrow{\cong}
W_m\sF^0_{\<\n\>}$. In particular, we get a commutative diagram
of cup product pairings of complexes 
\begin{equation}\label{eqn:G-comp-2}
  \xymatrix@C.8pc{
    W_m\sG^2 \ \ \times \ \ W_m\sF^0  \ar@<7ex>[d]  \ar[r]^-{\cup} &
    W_m\sH \ar@{=}[d] \\
  W_m\Omega^2_{X|D_{\<{\n_+}\>}, \log} \times W_m\sF^0_{\<\n\>}    \ar@<7ex>[u]
  \ar[r]^-{\cup} & W_m\sH.}
\end{equation}

For any $x \in X_{(0)}$, let $\tau^h_x \colon \Spec(k(x)) \inj X^h_x$ and
$\tau_x \colon \Spec(k(x)) \inj X$ be the inclusion maps, and let
$j_x \colon X^h_x \to X$ be the canonical {\'e}tale map. 
 Recall (cf. \cite[\S~3.1]{Zhao}) that any pairing of complexes of sheaves
$\sA^\cdot \otimes \sB^\cdot \to \sC^\cdot$ on $X_\et$ gives rise to a pairing
$\tau^*_x \sA^\cdot \otimes^L \tau^{!} \sB^\cdot \to
\tau^{!} \sC^\cdot$ in $D_\et(k(x))$. Applying this to ~\eqref{eqn:G-comp-2},
we get a commutative diagram of pairings 
\begin{equation}\label{eqn:G-comp-7}
  \xymatrix@C.8pc{
    \tau^*_x W_m\sG^2 \ \ \times \ \ \tau^{!}_x W_m\sF^0  \ar@<7ex>[d]
    \ar[r]^-{\cup} &  \tau^{!}_x W_m\sH \ar@{=}[d] \\
    \tau^*_x  W_m\Omega^2_{X|D_{\<{\n_+}\>}, \log} \times \tau^{!}_x
    W_m\sF^0_{\<\n\>}    \ar@<7ex>[u] \ar[r]^-{\cup} & \tau^{!}_x W_m\sH}
\end{equation}
in  $D_\et(k(x))$ which is compatible with ~\eqref{eqn:G-comp-2}.

\subsection{Relation with local duality}\label{sec:LG-dual}
For any
$\n \ge 0$, we let $W_m\sW^r_{\n}$ be defined by the short exact sequence 
\begin{equation}\label{eqn:Coker-0}
  0 \to  W_m\Omega^r_{X|D_{\n}, \log} \to W_m\Omega^r_{X,\log} \to
  W_m\sW^r_{\n} \to 0
\end{equation}
of Nisnevich (resp. {\'e}tale) sheaves on $X$. We can write
$W_m\sW^r_{\n} = \ {\underset{x \in D_\dagger}\bigoplus} (\tau_x)_* \circ
(\tau_x)^*(W_m\sW^r_{\n})$. In particular, the proper base change theorem for the
torsion {\'e}tale sheaves implies that
$H^i_\et(X, W_m\sW^r_{\n}) = {\underset{x \in D_\dagger}\bigoplus}
H^i_\et (X^h_x, W_m\sW^r_{\n})$. 
We let $W_m\sV^0_{\n} = \left[\frac{Z_1 \Fil_{D_\n}W_m\sO_X}{Z_1 W_m\sO_X}
  \xrightarrow{1-C} \frac{\Fil_{D_\n}W_m\sO_X}{W_m\sO_X}\right]$ so that we have a
term-wise short exact sequence of complexes on $X_\nis$ (resp. $X_\et$):
\begin{equation}\label{eqn:G-comp-3}
  0 \to W_m\sF^0 \to W_m\sF^0_\n \to W_m\sV^0_{\n} \to 0.
  \end{equation}

By definition of the pairing between complexes of sheaves (cf.
\cite[\S~1, p.~175]{Milne-ENS}), the pairing on the top row of ~\eqref{eqn:G-comp-2}
is equivalent to a morphism of complexes $W_m \sG^2 \to
\sHom^\cdot(W_m\sF^0, W_m \sH)$. 
Composing this with the canonical map $\sHom^\cdot(W_m\sF^0, W_m \sH) \to 
{\bf R}\sHom^\cdot(W_m\sF^0, W_m \sH)$, we get a morphism $\theta \colon
W_m\Omega^2_{X, \log} \cong W_m \sG^2 \to  
{\bf R}\sHom^\cdot(W_m\sF^0, W_m \sH)$ in $D_\et(X)$.
 By \cite[Prop.~6.6]{Spaltenstein}, 
this is equivalent to a morphism $\theta' \colon W_m\Omega^2_{X, \log}
\otimes^L W_m\sF^0 \to  W_m \sH$ in $D_\et(X)$. Similarly, the bottom row of
~\eqref{eqn:G-comp-2} gives rise to
a morphism in the derived category
$\theta_\n \colon W_m\Omega^2_{X|D_{\<{\n_+}\>}, \log} \to
{\bf R}\sHom^\cdot(W_m\sF^0_{\<\n\>}, W_m \sH)$. Moreover, we have a
commutative diagram
\begin{equation}\label{eqn:G-comp-4}
  \xymatrix@C.8pc{
    W_m\Omega^2_{X|D_{\<{\n_+}\>}, \log} \ar[r]^-{\theta_{\n}} \ar[d] &
    {\bf R}\sHom^\cdot(W_m\sF^0_{\<\n\>}, W_m \sH) \ar[d]  \\
W_m\Omega^2_{X, \log} \ar[r]^-{\theta} &  {\bf R}\sHom^\cdot(W_m\sF^0, W_m \sH).}
\end{equation}

Letting $\psi_\n$ denote the induced map between the homotopy fibers, we
get a commutative diagram of distinguished triangles in $D_\et(X, {\Z}/{p^m})$:
\begin{equation}\label{eqn:G-comp-5}
  \xymatrix@C.8pc{
    W_m \sW^2_{\<{\n_+}\>}[-1] \ar[d]_-{\psi_{\<{\n}\>}} \ar[r] &
    W_m\Omega^2_{X|D_{\<{\n_+}\>}, \log} \ar[d]_-{\theta_{\<{\n}\>}} \ar[r] &
     W_m\Omega^2_{X, \log} \ar[d]^-{\theta} \ar[r]^-{+1} & \\ 
     {\bf R}\sHom^\cdot(W_m \sV^0_{\<\n\>}, W_m \sH) \ar[r] &
     {\bf R}\sHom^\cdot(W_m\sF^0_{\<\n\>}, W_m \sH) \ar[r] &
     {\bf R}\sHom^\cdot(W_m\sF^0, W_m \sH) \ar[r]^-{+1} & .}
 \end{equation}
 
Applying ${\bf R}\Hom^\cdot({\Z}/{p^m}, (-))$ to the above diagram, and using the
 canonical isomorphism ${\bf R}\Hom^\cdot({\Z}/{p^m},
 {\bf R}\sHom^\cdot(\sB^\cdot, \sC^\cdot)) \xrightarrow{\cong}
 {\bf R}\Hom^\cdot(\sB^\cdot, \sC^\cdot)$ (cf. \cite[Prop.~6.6]{Spaltenstein})
 as well as the canonical map
 ${\bf R}^i\Hom^\cdot(\sB^\cdot, \sC^\cdot) \to
 \Hom({\bf R}^j\Hom({\Z}/{p^m}, \sB^\cdot),
{\bf R}^{i+j}\Hom({\Z}/{p^m}, \sC^\cdot))$,
 we get a commutative diagram of long exact hypercohomology sequences 
 \begin{equation}\label{eqn:G-comp-6}
  \xymatrix@C.5pc{
    0 \ar[r] & H^0(W_m\Omega^2_{X|D_{\<{\n_+}\>}, \log}) \ar[r]
    \ar[d]_-{\theta^*_{\<{\n}\>}} &
    H^0(W_m\Omega^2_{X, \log}) \ar[r] \ar[d]^-{\theta^*} &
    H^0(W_m \sW^2_{\<{\n_+}\>}) \ar[r]^-{\partial}
    \ar[d]^-{\psi^*_{\<{\n}\>}} &
    H^1(W_m\Omega^2_{X|D_{\<{\n_+}\>}, \log}) \ar[r]
    \ar[d]^-{\theta^*_{\<{\n}\>}} & \cdots \\
    0 \ar[r] & \H^2(W_m\sF^0_{\<\n\>})^\vee \ar[r] &
    \H^2(W_m\sF^0)^\vee \ar[r] & \H^1(W_m\sV^0_{\<\n\>})^\vee \ar[r] &
    \H^1(W_m\sF^0_{\<\n\>})^\vee \ar[r] & \cdots ,}
\end{equation}
where we have used the shorthand $\H^i(\sA^\cdot)$ for
$\H^i_\et(X, \sA^\cdot)$. The exactness of the bottom row at the left end
is a consequence of \lemref{lem:Coh-V1}.

If $x \in D_\dagger$, then ~\eqref{eqn:G-comp-7} gives rise to a commutative diagram
\begin{equation}\label{eqn:G-comp-8}
  \xymatrix@C.8pc{
\tau^*_x W_m\Omega^2_{X|D_{\<{\n_+}\>}, \log} \ar[r]^-{\theta_{\n,x}} \ar[d] &
{\bf R}\sHom^\cdot(\tau^{!}_x W_m\sF^0_{\<\n\>}, \tau^{!}_x
W_m \sH) \ar[d] \\
\tau^*_x W_m\Omega^2_{X, \log} \ar[r]^-{\theta_x} &
{\bf R}\sHom^\cdot(\tau^{!}_xW_m\sF^0, \tau^{!}_xW_m \sH),}
\end{equation}
which is compatible with ~\eqref{eqn:G-comp-4}.

Since $\tau^*_x$ is exact, it is easy to check that the homotopy fibers of
the left and the right vertical arrows are $\tau^*_x  W_m \sW^2_{\<{\n_+}\>}[-1]$
and ${\bf R}\sHom^\cdot(\tau^{!}_x W_m \sV^0_{\<\n\>}, \tau^{!}_x
W_m \sH)$, respectively. We also note that the maps
$j^*_x W_m \sW^2_{\<{\n_+}\>} \to (\tau^h_x)_* \tau^*_x  W_m \sW^2_{\<{\n_+}\>}$
and $(\tau^h_x)_* \tau^{!}_x W_m \sV^0_{\<\n\>} \to j^*_x W_m \sV^0_{\<\n\>}$
are isomorphisms on the {\'e}tale (and Nisnevich) site of $X^h_x$.
Since the canonical (i.e., forget support) map
$H^2_x(X, W_m\Omega^2_{X, \log}) \to H^2_\et(X, W_m\Omega^2_{X, \log})$ is bijective,
letting $\psi_{\<\n\>,x}$ denote the induced map between the mapping cones
of the vertical arrows in ~\eqref{eqn:G-comp-8}, considering the
corresponding maps between the hypercohomology and using the proper base change,
we obtain a commutative diagram 
\begin{equation}\label{eqn:G-comp-9}
  \xymatrix@C.5pc{
 0 \ar[r] & G^0(x) \ar[r]
    \ar[d]_-{\theta^*_{\<{\n}\>}} &
    H^0(W_m\Omega^2_{X^h_x, \log}) \ar[r] \ar[d]^-{\theta^*_x} &
    H^0(\tau^*_x W_m \sW^2_{\<{\n_+}\>}) \ar[r]^-{\partial}
    \ar[d]^-{\psi^*_{\<{\n}\>,x}} & G^1(x) \ar[r]
    \ar[d]^-{\theta^*_{\<{\n}\>,x}} & \cdots \\
    0 \ar[r] & \H^2_x(W_m\sF^0_{\<\n\>})^\vee \ar[r] &
    \H^2_x(W_m\sF^0)^\vee \ar[r] & \H^1_x(W_m\sV^0_{\<\n\>})^\vee \ar[r] &
    \H^1_x(W_m\sF^0_{\<\n\>})^\vee \ar[r] & \cdots }
\end{equation}
of {\'e}tale cohomology on $X^h_x$ whose rows are exact and which is compatible with
~\eqref{eqn:G-comp-6}. Here, we let
$G^i(x) = H^i_\et(X^h_x, W_m\Omega^2_{X^h_x|D_{\<{\n_+}\>}, \log})$.

\subsection{The maps $\psi_\n$ vs. $\wt{\rho}^{\n}_{K,m}$}
\label{sec:Loc-der-rec}
Our goal now is to compute the cohomology of $W_m\sW^r_{\n}$ and $W_m\sV^0_\n$,
and show that $\psi^*_{\<\n\>, x}$ coincides with the reciprocity map
$\wt{\rho}^{n_x+1}_{K,m}$ of \corref{cor:Local-inj-6}.
We write $D_\n = {\underset{x \in D_\dagger}\amalg} D_{\n, x}$.

\begin{lem}\label{lem:Coker-1}
For $x \in D_\n$, we have the following.
\begin{enumerate}
\item
$H^i_x(X, W_m\sW^r_\n) \xrightarrow{\cong} H^i_\et(X^h_x, \sW^r_\n)$. 
\item
\[
H^0_\et(X^h_x, W_m\sW^r_{\n}) = \left\{\begin{array}{ll}
{\Z}/{p^m} & \mbox{if $r = 0$} \\
{\sO^\times(D_{\n, x})}/{p^m} & \mbox{if $r = 1$} \\
\varpi^m_r(A^h_x|n_x)  & \mbox{if $r \ge 2$}.
\end{array}\right.
\]
\item
$H^i_\et(X^h_x, W_m\sW^r_{\n}) = 0$ for $i \ge 1$ and $r \ge 2$.
\end{enumerate}
\end{lem}
\begin{proof}
  The items (1)  and (2) are clear from the definition of $W_m\sW^r_{\n}$ while
 (3) follows from \cite[Prop.~1.1.9]{JSZ} which says
  that $ W_m\sW^r_{\n}$ has a finite decreasing filtration whose graded quotients
  are coherent sheaves on $D_\n$. 
\end{proof}

\begin{lem}\label{lem:Coh-V0}
  For $x \in D_\dagger$, we have $H^i_\et(X^h_x, Z_1 \Fil_{D_\n} W_m \sO_X) =
  H^i_\et(X^h_x, \Fil_{D_\n} W_m \sO_X) = 0$ for $i \ge 1$.
  In particular,
  \[
\H^i_\et(X^h_x, W_m\sF^0_\n) = \left\{\begin{array}{ll}
{\Z}/{p^m} & \mbox{if $i = 0$} \\
\Fil_{n_x} H^1_\et(K^h_x, {\Z}/{p^m}) & \mbox{if $i = 1$} \\
 0 & \mbox{if $i > 1$.}
\end{array}\right.
\]
\end{lem}
\begin{proof}
  % From the proof of \cite[Lem.~3.5]{Kerz-Saito-ANT},
  One deduces from ~\eqref{eqn:Matsuda-0-0} that there is a commutative diagram
  of exact sequences
  \begin{equation}\label{eqn:Coh-V0-0}
    \xymatrix@C.8pc{
      0 \ar[r] & \sO_X(D/p) \ar[r]^-{V^{m-1}} \ar[d]_-{\ov{F}} &
      \Fil_{D/p}W_m\sO_X \ar[r]^-{R} \ar[d]^-{\ov{F}} & 
      \Fil_{D/{p^2}}W_{m-1}\sO_X \ar[d]^-{\ov{F}} \ar[r] & 0 \\
 0 \ar[r] & \sO_X(D) \ar[r]^-{V^{m-1}} &
      \Fil_{D}W_m\sO_X \ar[r]^-{R} & 
      \Fil_{D/p}W_{m-1}\sO_X \ar[r] & 0.}
    \end{equation}
Since $\ov{F}$ is injective and $Z_1 W_m \sO_X = {\rm Image}(F \colon
W_{m+1}\sO_X \to W_m\sO_X) = {\rm Image}(\ov{F} \colon W_{m}\sO_X \to W_m\sO_X)$,
it follows that the sequence
\begin{equation}\label{eqn:Coh-V0-1}
0 \to Z_1\sO_X(D) \xrightarrow{V^{m-1}} Z_1\Fil_{D}W_m\sO_X \xrightarrow{R}  
Z_1\Fil_{D/p}W_{m-1}\sO_X \to 0
\end{equation}
is exact on $X_\et$.

Since $Z_1\sO_X(D) \cong (\psi_* Z_1\sO_X)(D/p)$, it follows
that $Z_1\sO_X(D)$ is a coherent sheaf on $X$. In particular,
$H^i_\et(X^h_x, Z_1\sO_X(D)) = 0$ for all $i \ge 1$. Using induction on $m$ and
the long exact cohomology sequence associated to ~\eqref{eqn:Coh-V0-1},
we get $H^i_\et(X^h_x, Z_1 \Fil_{D_\n} W_m \sO_X) = 0$ for $i \ge 1$.
Similarly, the bottom row of ~\eqref{eqn:Coh-V0-0} implies that
$H^i_\et(X^h_x, \Fil_{D_\n} W_m \sO_X) = 0$ for $i \ge 1$.
It follows that $\H^i_\et(X^h_x, W_m\sF^0_\n) = 0$ for $i > 1$
and $\H^1_\et(X^h_x, W_m\sF^0_\n)$ is the cokernel of the map
$H^0_\et(X^h_x, Z_1\Fil_{D_\n} W_m \sO_X) \to H^0_\et(X^h_x, \Fil_{D_\n} W_m \sO_X)$.

Now, we note that
  $H^0_\et(X^h_x, Z_1\Fil_{D_\n} W_m \sO_{X}) = Z_1\Fil_{D_\n} W_m(K^h_x)$ and
  $H^0_\et(X^h_x, \Fil_{D_\n} W_m \sO_{X})$ \
  $= \Fil_{n_x} W_m(K^h_x)$.
  Since $(1-\ov{F})(\Fil_{\lfloor {n_x}/{p} \rfloor} W_m(K^h_x))$ \
  $= (1 -C)(Z_1 \Fil_{n_x} W_m(K^h_x))$, it follows from 
~\eqref{eqn:p-tor-2} and the first statement of the lemma that 
%replace ~\eqref{eqn:p-tor-2} \lemref{lem:Matsuda-0} and defn of $\Fil_{n_x} H^1_\et(K^h_x, {\Z}/{p^m})$
$\H^1_\et(X^h_x, W_m\sF^0_\n) \cong \Fil_{n_x} H^1_\et(K^h_x, {\Z}/{p^m})$.
It remains to show that $\H^0_\et(X^h_x, W_m\sF^0_\n) \cong {\Z}/{p^m}$.
But this follows because, in the sequence of canonical maps
\begin{equation}\label{eqn:Coh-V0-2}
{\Z}/{p^m} \cong  \H^0_\et(X^h_x, W_m\sF^0) \inj \H^0_\et(X^h_x, W_m\sF^0_\n) \inj
  \H^0_\et(K^h_x, W_m\sF^0_\n) \xrightarrow \cong
  {\Z}/{p^m},
  \end{equation}
the composite map is an isomorphism. This concludes the proof.
\end{proof}

\begin{cor}\label{cor:Coh-V0-00}
  We have
  \[
\H^2_{x}(X, W_m\sF^0_\n) \cong
\frac{H^1_\et(K^h_x, {\Z}/{p^m})}{\Fil_{n_x} H^1_\et(K^h_x, {\Z}/{p^m})}; \ \ \
\H^i_{x}(X, W_m\sF^0_\n)  = 0 \ \ \mbox{for} \ \ i \neq 2.
\]
\end{cor}
\begin{proof}
Follows immediately from \lemref{lem:Coh-V0} using the localization
  sequence
  \[
    \cdots \to \H^i_x(X, W_m\sF^0_\n) \to  \H^i_\et(X^h_x, W_m\sF^0_\n) \to
    \H^i_\et(K^h_x, W_m\sF^0_\n) \to \H^{i+1}_x(X, W_m\sF^0_\n) \to \cdots .
  \]
  \end{proof}

Using ~\eqref{eqn:Coh-V0-0}, ~\eqref{eqn:Coh-V0-1} and the fact that the
operators $V$ and $C$ commute with each other on $j_* W_m\sO_{X^o}$, we get the
following.
\begin{cor}\label{cor:F-V-R}
  There is an exact sequence of complexes
  \[
    0 \to \sF^0_\n \xrightarrow{V^{m-1}} W_m\sF^0_\n \xrightarrow{R}
    W_{m-1}\sF^0_{{\n}/p} \to 0.
  \]
\end{cor}

\begin{lem}\label{lem:Coh-V1}
    We have
    \[
      \H^i_\et(X, W_m\sV^0_{\n}) = \left\{\begin{array}{ll}
    {\underset{x \in D_\dagger}\bigoplus}
\frac{\Fil_{n_x} H^1_\et(K^h_x, {\Z}/{p^m})}{\Fil_0 H^1_\et({K}^h_x, {\Z}/{p^m})}
& \mbox{if $i = 1$} \\
 0 & \mbox{otherwise.}
\end{array}\right.
\]
\end{lem}
\begin{proof}
Since 
$\H^i_\et(X,  W_m\sV^0_{\n}) \cong {\underset{x \in D_\dagger}\bigoplus}
\H^i_\et(X^h_x, W_m\sV^0_{\n})$, we can replace $X$ by $X^h_x$
in order to prove the lemma. We now consider the long exact sequence
  \begin{equation}\label{eqn:Coh-V1-0}
\cdots \to \H^0_\et(X^h_x, W_m\sF^0) \xrightarrow{\alpha_0} \H^0_\et(X^h_x, W_m\sF^0_\n)
\to \H^0_\et(X^h_x, W_m\sV^0_{\n}) \hspace*{3cm}
\end{equation}
\[
  \hspace*{9cm} \to \H^{1}_\et(X^h_x, W_m\sF^0) \xrightarrow{\alpha_1} \cdots .
\]
%There are inclusions
%\[ H^0_\et(X^h_x, {\Z}/{p^m}) \cong \H^0_\et(X^h_x, W_m\sF^0) \inj
   % \H^0_\et(X^h_x, W_m\sF^0_\n) \inj
    %\H^0_\et(K^h_x, W_m\sF^0) \cong H^0_\et(K^h_x, {\Z}/{p^m}).
  %\]
It follows from ~\eqref{eqn:Coh-V0-2}
that $\alpha_0$ is a bijection between two groups, each isomorphic to
  ${\Z}/{p^m}$.
  %Since $\H^1_\et(X^h_x, W_m\sF^0) = \Fil_{0} H^1_\et(K^h_x, {\Z}/{p^m})$, we get from
  \lemref{lem:Coh-V0} implies that $\alpha_1$ is injective. It follows that
  $\H^0_\et(X^h_x, W_m\sV^0_{\n}) = 0$. We now apply \lemref{lem:Coh-V0} and
  ~\eqref{eqn:Coh-V1-0} to conclude the proof.
\end{proof}

\begin{prop}\label{prop:KSL}
  For every $\n \ge 0$, there is a canonical isomorphism
  \[
 \Fil_{D_{\n_+}} H^1_\et(K, {\Z}/{p^m}) \xrightarrow{\cong} \H^1_\et(X, W_m\sF^0_{\n}).
\]
\end{prop}
\begin{proof}
  Using \corref{cor:Coh-V0-00}, we see that the 
 long exact sequence
  \[
   \cdots \to  \H^1_{D_\dagger}(X,W_m\sF^0_\n)  \to \H^1_\et(X, W_m\sF^0_\n)
    \xrightarrow{j^*} \H^1_\et(X^o, W_m\sF^0_\n) \to \H^2_{D_\dagger}(X,W_m\sF^0_\n).
  \]
is of the form
  \begin{equation}\label{eqn:KSL-0} 
0 \to \H^1_\et(X, W_m\sF^0_\n) \xrightarrow{j^*} H^1_\et(X^o, {\Z}/{p^m}) \to 
{\underset{x \in D_\dagger}\bigoplus} 
\frac{H^1_\et({K}^h_x, {\Z}/{p^m})}{\Fil_{n_x} H^1_\et({K}^h_x, {\Z}/{p^m})}.
\end{equation}
We can now use ~\eqref{eqn:Fil-U-p} to conclude the proof.
\end{proof}

\begin{lem}\label{lem:Coh-V2}
  The diagram
  \begin{equation}\label{eqn:Coh-V2-0}
    \xymatrix@C.8pc{
      H^0_\et(X, W_m \sW^2_{\<{\n_+}\>}) \ar[r] 
      \ar[dd]_--{\psi^*_{\<{\n}\>}} & {\underset{x \in D_\dagger}\bigoplus}
      \varpi^m_2(A^h_x|\<n_x+1\>) \ar[d] \\
      & {\underset{x \in D_\dagger}\bigoplus} \left(\frac{\Fil_{\<n_x+1\> -1}
          H^1_\et(K^h_x, {\Z}/{p^m})}{\Fil_{-1} H^1_\et(K^h_x, {\Z}/{p^m})}\right)^\vee
      \ar@{->>}[d] \\
      \H^1_\et(X, W_m\sV^0_{\<\n\>})^\vee \ar[r] &
      {\underset{x \in D_\dagger}\bigoplus} \left(\frac{\Fil_{\<n_x\>}
       H^1_\et(K^h_x, {\Z}/{p^m})}{\Fil_{-1} H^1_\et(K^h_x, {\Z}/{p^m})}\right)^\vee}
  \end{equation}
  commutes, where the right vertical arrow on the top is the direct sum of
  $\wt{\rho}^{\<n_x +1\>}_{K^h_x, m}$ and on the bottom is the direct sum of
  canonical surjections, taken over $x \in D_\dagger$.
\end{lem}
\begin{proof}
The diagram ~\eqref{eqn:Coh-V2-0} is the composition of
  two squares
  \begin{equation}\label{eqn:Coh-V2-1}
    \xymatrix@C.8pc{
      H^0_\et(X, W_m \sW^2_{\<{\n_+}\>}) \ar[r] \ar[d]_-{\psi^*_{\<\n\>}} &
      {\underset{x \in D_\dagger}\bigoplus}
      H^0_\et(X^h_x, W_m \sW^2_{\<{\n_+}\>}) \ar[r] \ar[d]^-{\oplus \ \psi^*_{\<\n\>,x}} &
      {\underset{x \in D_\dagger}\bigoplus} \varpi^m_2(A^h_x|\<n_x+1\>)
      \ar[d]^-{\theta^*_{\<\n\>}} \\
      \H^1_\et(X, W_m\sV^0_{\<\n\>})^\vee \ar[r] &   
{\underset{x \in D_\dagger}\bigoplus}
      \H^1_x(X^h_x, W_m\sV^0_{\<\n\>})^\vee \ar[r] &
      {\underset{x \in D_\dagger}\bigoplus} \left(\frac{\Fil_{\<n_x\>}
       H^1_\et(K^h_x, {\Z}/{p^m})}{\Fil_{-1} H^1_\et(K^h_x, {\Z}/{p^m})}\right)^\vee,}
\end{equation}
where $\theta^*_{\<\n\>}$ is the composite right vertical arrow in
~\eqref{eqn:Coh-V2-0}.
Since ~\eqref{eqn:G-comp-6} and ~\eqref{eqn:G-comp-9} are compatible, it follows
that the left square in ~\eqref{eqn:Coh-V2-1} is commutative.
It remains to show that the right square commutes.
We can therefore assume $D_\dagger = \{x\}$. We write $\un{n} = n$.

Now, it follows from \lemref{lem:Coker-1} that the map
  ${K_2(A^h_x)}/{p^m} \cong H^0_\et(X^h_x, W_m\Omega^2_{X,\log}) \to
  H^0_\et(X^h_x, W_m\sW^2_{\<n\>})$ is surjective.
  \corref{cor:Coh-V0-00} (with $\n = 0$) and
  \lemref{lem:Coh-V1} imply that
  $\H^2_x(X^h_x, W_m\sF^0)^\vee \to \H^1_x(X^h_x, \sV^0_{\<n\>})^\vee$
  is also surjective. We conclude that ~\eqref{eqn:G-comp-9} breaks into
  a commutative diagram of short exact sequences
  \begin{equation}\label{eqn:Coh-V2-2}
\xymatrix@C.8pc{
 0 \ar[r] & H^0_\et(X^h_x, W_m\Omega^2_{X^h_x|D_{\<{n+1}\>}, \log}) \ar[r]
    \ar[d]_-{\theta^*_{\<{n}\>}} &
    H^0_\et(X^h_x, W_m\Omega^2_{X^h_x, \log}) \ar[r] \ar[d]^-{\theta^*_x} &
    H^0_\et(X^h_x, W_m \sW^2_{\<{n+1}\>}) \ar[r]
    \ar[d]^-{\psi^*_{\<{n}\>,x}} \ar[r] & 0  \\
    0 \ar[r] & \H^2_x(X^h_x, W_m\sF^0_{\<n\>})^\vee \ar[r] &
    \H^2_x(X^h_x, W_m\sF^0)^\vee \ar[r] &
    \H^1_x(X^h_x, W_m\sV^0_{\<n\>})^\vee \ar[r] & 0.}
\end{equation}

By definition of Kato's reciprocity map (cf. \S~\ref{sec:RHom}), it follows
from \corref{cor:Coh-V0-00} that
  the left square in ~\eqref{eqn:Coh-V2-2} is the same as the left square in
  the commutative diagram
  \begin{equation}\label{eqn:Coh-V2-4}
\xymatrix@C.8pc{
  0 \ar[r] & \kappa^m_2(A^h_x|\<n+1\>) \ar[r] \ar[d]_-{\rho^{\<n+1\>}_{K^h_x, m}} &
  \kappa^m_2(A^h_x|0) \ar[d]^-{\rho^{0}_{K^h_x, m}} \ar[r] &
\varpi^m_2(A^h_x|\<n+1\>) \ar[d]^-{\wt{\rho}^{\<n+1\>}_{K^h_x,m}} \ar[r] & 0 \\
0 \ar[r] & G^{(\<n+1\>)}_{K^h_x,m} \ar[r] & G^{(0)}_{K^h_x,m} \ar[r] &
\left(\frac{\Fil_{\<n\>} H^1_\et(K^h_x, {\Z}/{p^m})}
{\Fil_{-1} H^1_\et(K^h_x, {\Z}/{p^m})}\right)^\vee\ar[r] & 0.}
\end{equation}
From this, it immediately follows that
$\psi^*_{\<{n}\>,x} = \theta^*_{\<n\>}$.
\end{proof}

The following is the main cohomology pairing of this section that we shall use in
\S~\ref{sec:DThm} to prove the duality theorem for $X^o$.

\begin{prop}\label{prop:Pair-0-4}
 For $\n \ge 0$, the bottom row of ~\eqref{eqn:G-comp-2} induces a
 cohomology pairing 
  \begin{equation}\label{eqn:Pair-0-4-2}
    \H^i_\et(X, W_m\sF^0_{\<\n\>}) \times
    H^{2-i}_\et(X, W_m\Omega^2_{X|D_{\<\n\>_+}, \log}) \to {\Z}/{p^m}.
  \end{equation}
\end{prop}
\begin{proof}
We look at the diagram 
  \begin{equation}\label{eqn:Pair-0-4-0}
  \xymatrix@C.5pc{
    H^0(W_m\Omega^2_{X, \log}) \ar[r] \ar@{=}[d] &
    H^0(W_m \sW^2_{\<{\n_+}\>}) \ar[r]^-{\partial}
    \ar@{->>}[d] &
    H^1(W_m\Omega^2_{X|D_{\<{\n_+}\>}, \log}) \ar[r] \ar@/^4pc/[dd]
    \ar@{->>}[d] &  H^1(W_m\Omega^2_{X, \log}) \ar@{=}[d] \\
 H^0(W_m\Omega^2_{X, \log}) \ar[r] \ar[d]_-{\theta^*} &
    H^0(W_m \sW^2_{\<{\n}\>_+}) \ar[r]^-{\partial}
    \ar[d]^-{\psi^*_{\<{\n}\>}} &
    H^1(W_m\Omega^2_{X|D_{\<{\n}\>_+}, \log}) \ar[r]
    \ar@{.>}[d]^-{\phi^*_{\<{\n}\>}} &  H^1(W_m\Omega^2_{X, \log}) \ar[d]^-{\theta^*} \\
\H^2(W_m\sF^0)^\vee \ar[r] & \H^1(W_m\sV^0_{\<\n\>})^\vee \ar[r] &
    \H^1(W_m\sF^0_{\<\n\>})^\vee \ar[r] &  \H^1(W_m\sF^0)^\vee}
\end{equation}
of {\'e}tale cohomology on $X$ whose rows are exact and the curved arrow in the
middle is the map $\theta^*_{\<{\n}\>}$.
Note that $\psi^*_{\<\n\>}$ has a factorization as
indicated in the diagram because it is simply Kato's reciprocity map
as shown in \lemref{lem:Coh-V2}.

To prove the proposition for $i =1$ is to show that 
$\theta^*_{\<{\n}\>}$ factors through $\phi^*_{\<{\n}\>}$ such that all squares in the
lower floor commute. But this is a straightforward diagram chase and we omit
the details.
The $i =0$ case is clear because neither $\H^0_\et(X, W_m\sF^0_{\<\n\>})$ nor
$H^{2}_\et(X, W_m\Omega^2_{X|D_{\<\n\>_+}, \log})$ depends on $\n$, as one can easily
deduce from Lemmas~\ref{lem:Coker-1} and ~\ref{lem:Coh-V1}.

For $i=2$ case, we look at the diagram
 \begin{equation}\label{eqn:Pair-0-4-1}
  \xymatrix@C.5pc{
    0 \ar[r] &
    H^0(W_m\Omega^2_{X|D_{\<{\n_+}\>}, \log}) \ar[r] \ar@/_5pc/[dd]_-{\theta^*_{\<\n\>}}
\ar@{^{(}->}[d] &  H^0(W_m\Omega^2_{X, \log}) \ar@{=}[d] \ar[r] &
H^0(W_m \sW^2_{\<{\n_+}\>}) \ar@{->>}[d] \\
& H^0(W_m\Omega^2_{X|D_{\<{\n}\>_+}, \log}) \ar[r]
\ar@{.>}[d]_-{\phi^*_{\<{\n}\>}} &  H^0(W_m\Omega^2_{X, \log}) \ar[d] \ar[r] &
H^0(W_m \sW^2_{\<{\n}\>_+}) \ar[d]^-{\psi^*_{\<{\n}\>}}  \\
0 \ar[r] &  \H^2(W_m\sF^0_{\<\n\>})^\vee \ar[r] &  \H^2(W_m\sF^0)^\vee \ar[r] &
    \H^1(W_m\sV^0_{\<\n\>})^\vee.}
\end{equation}
Since the three rows are exact and the two squares on the top floor as well as
the right square on the bottom floor are commutative, an easy diagram chase
shows that there exists $\phi^*_{\<\n\>}$ which factorizes $\theta^*_{\<\n\>}$ and
all squares in ~\eqref{eqn:Pair-0-4-1} are commutative. 
\end{proof}

\section{Duality theorem for open curves}\label{sec:DThm}
We fix a local field $k$ of characteristic $p > 0$ and a connected, smooth and
projective curve $X$ over $k$. We also fix a positive integer $m$.
We let $D_\dagger = \{x_1, \ldots , x_s\}$ be a reduced effective divisor on $X$
and let $j \colon X^o \inj X$ be the inclusion of the complement of $D$.
In this section, we shall prove a duality theorem for the $p$-adic {\'e}tale
cohomology on $X^o$. We shall maintain the notations of \S~\ref{sec:Exist}.
For $\n, i \ge 0$, we let  $G^i_\n(m) = H^i_\et(X, W_m\Omega^2_{X|D_{\n}, \log})$
and $F^i_\n(m) = \H^i_\et(X, W_m\sF^0_\n)$. Note that
$G^i_0(m) = H^i_\et(X, W_m\Omega^2_{X,\log})$ and
$F^i_0(m) = \H^i_\et(X, W_m\sF^0_\n) = H^i_\et(X, {\Z}/{p^m})$ for $\n \le 0$.

\subsection{Topology of $G^i_\n(m)$ in a special case}\label{sec:Top}
We shall endow each $G^i_\n(m)$
with the structure of a topological abelian group with some restriction on $\n$
when $i =1$.
It is clear that $G^i_\n(m) = 0$ for $i \ge 3$
(cf. \cite[Chap.~VI, Rem.~1.5]{Milne-etale}). As for $i =2$, we have seen in the 
proof of \propref{prop:Pair-0-4} that $G^2_\n(m)$ does not depend on
$\n$, as is the case with $F^0_\n(m)$.
On the other hand, the pairing ~\eqref{eqn:Comp-3} induces an isomorphism
$G^2_0(m) \xrightarrow{\cong} F^0_0(m)^\vee \cong {\Z}/{p^m}$ by
\cite[\S~3, Prop.~4]{Kato-Saito-Ann} (see also \cite[Thm.~4.7]{KRS}).
We therefore endow $G^2_\n(m)$ with the discrete
topology via this isomorphism.

For $i = 0$, we note that the commutative squares on the lower floor of
~\eqref{eqn:Pair-0-4-1} exist for any $\n \ge 0$. It follows then from
\corref{cor:Local-inj-6}, ~\lemref{lem:Coh-V2}, ~\eqref{eqn:Pair-0-4-1} and
\cite[\S~3, Prop.~4]{Kato-Saito-Ann} (see also \cite[Thm.~4.7]{KRS}) that
the map $G^0_\n(m) \to F^2_{\n_-}(m)^\vee$ is an isomorphism.
Since $F^2_0(m)$ has discrete topology (cf. \cite[Prop.~3.6]{KRS}) and
$F^2_{\n_-}(m)$ is its quotient,
the latter is a discrete topological abelian group. We therefore endow
$G^0_\n(m)$ with the profinite topology so that
\begin{equation}\label{eqn:Top-0}
G^0_\n(m) \times  F^2_{\n_-}(m) \to {\Z}/{p^m}
\end{equation}
is a perfect pairing of topological abelian groups.

To define the topology on $G^1_{\<\n\>}(m)$, we shall use a recipe.
Let $G$ and $G'$ be two abelian groups such that $G'$ is a torsion group. Let
$G' \times G \to {\Q}/{\Z}$ be a pairing which has no left kernel. Let
$\beta \colon G' \to G^\vee$ be the induced injective homomorphism.
We let $G$ be endowed with the weakest topology
with respect to which the map $\beta(f) \colon G \to  {\Q}/{\Z}$ is continuous
for every $f \in G'$ (called the weak topology induced by $G'$).
This makes $G$ a topological abelian group for which 
the identity element has a fundamental system of open neighborhoods of the form
$\Ker(\beta(f_1)) \cap \cdots \cap \Ker(\beta(f_r))$ with
$f_1, \ldots , f_r \in G'$.
It is clear that $\beta$ has a factorization $\beta \colon G' \inj G^\star$, where
the dual on the right is taken with respect to the weak topology.

\begin{lem}\label{lem:Top-1}
The map $\beta \colon G' \to G^\star$ is a bijection.
\end{lem}
\begin{proof}
  Let $f \in G^\star$. We can find $f_1, \ldots , f_r \in G'$ such that
  the continuous homomorphism $f \colon G \to {\Q}/{\Z}$ has the property that
  $H:= \stackrel{r}{\underset{i=1}\cap} \Ker(\beta(f_i)) \subset \Ker(f)$.
  We now look at 
  \begin{equation}\label{eqn:Top-1-0}
    \xymatrix@C.8pc{
      G \ar@{->>}[r] \ar[dr]_-{\Delta} \ar@/^2pc/[rrr]^-{f}
      & {G}/H \ar@{->>}[r] \ar@{^{(}->}[d] &
      \frac{G}{\Ker(f)} \ar[r]_-{\ov{f}} & {\Q}/{\Z} \\
      & \stackrel{r}{\underset{i =1}\bigoplus} \frac{G}{\Ker(\beta(f_i))}.
      \ar@/_2pc/[urr]_-{\wt{f}} & &}
  \end{equation}

 Since ${\Q}/{\Z}$ is injective, there exists $\wt{f}$ such that the above diagram
  is commutative. Since $G'$ is a torsion group, we see that each summand of the
  the group on the bottom is a finite cyclic group.
  It follows that we can replace ${\Q}/{\Z}$ by a finite cyclic group ${\Z}/q$
  in ~\eqref{eqn:Top-1-0}. Using the standard description of homomorphisms between
  finite cyclic groups, we can find integers $n_1, \ldots , n_r$ such
  that $f = \wt{f} \circ \Delta = \stackrel{r}{\underset{i =1}\sum} n_i \beta(f_i)
  = \beta(\stackrel{r}{\underset{i =1}\sum} n_i f_i) \in G'$.
\end{proof}

For a topological abelian group $G$ whose topology is $\tau$, we let $G^{\pf}$
denote the profinite completion of $G$ with respect to $\tau$. That is, $G^{\pf}$ is
the inverse limit ${\varprojlim}_U \ {G}/U$ with the inverse limit
topology, where $U$ runs through the $\tau$-open subgroups of finite index in $G$.
It is clear that $G^{\pf}$ is a profinite abelian group and the
profinite completion map $G \to G^{\pf}$ is continuous with dense image.
By the Pontryagin duality, this implies the following.
\begin{lem}\label{lem:PF-compln}
  If $G$ is a torsion abelian group, then the map $(G^{\pf})^\star \to G^\star$ is a
  topological isomorphism for the discrete topology of $G^\star$.
\end{lem}

Let $G$ be a torsion topological abelian group and let  $G^\star$ be endowed with
the discrete topology. Then the evaluation map ${\rm ev}_G \colon G \to
(G^\star)^\vee$ is continuous with respect to the profinite topology of the target
group. This implies that ${\rm ev}_G$ has a canonical factorization
$G \to G^{\pf} \xrightarrow{{\rm ev}^{\pf}_G} (G^\star)^\vee$ with ${\rm ev}^{\pf}_G$
continuous.
\begin{cor}\label{cor:PF-compln-0}
  The map ${\rm ev}^{\pf}_G \colon G^{\pf} \to  (G^\star)^\vee$ is a topological
  isomorphism.
\end{cor}
\begin{proof}
  Follows from \lemref{lem:PF-compln} using the Pontryagin duality.
\end{proof}

To apply the above recipe, we use \propref{prop:Pair-0-4} to
endow $G^1_{\<\n\>_+}(m)$ with the structure of a
topological abelian group via the weak topology induced by $F^1_{\<\n\>}(m)$.
 To check that the hypothesis of \lemref{lem:Top-1} is satisfied, 
we look at the commutative diagram
\begin{equation}\label{eqn:Top-2}
  \xymatrix@C.8pc{
    0 \ar[r] & F^1_0(m) \ar[r] \ar[d]_-{(\theta^*)^\vee} &
    F^1_{\<\n\>}(m) \ar[r] \ar[d]^-{(\phi^*_{\<\n\>})^\vee} &
    \H^1_\et(X, W_m\sV^0_{\<\n\>}) \ar[d]^-{(\psi^*_{\<\n\>})^\vee} \ar[r] & F^2_0(m)
    \ar[d]^-{(\theta^*)^\vee} \\
    0 \ar[r] & G^1_0(m)^\vee \ar[r] &
    G^1_{\<\n\>_+}(m)^\vee \ar[r]^-{\partial^\vee} &
    H^0_\et(X,W_m \sW^2_{\<{\n}\>_+})^\vee\ar[r] & G^0_0(m)^\vee,}
  \end{equation}
  obtained by dualizing the lower floor of ~\eqref{eqn:Pair-0-4-0}.
Lemmas~\ref{lem:Coker-1} and ~\ref{lem:Coh-V1} imply that
  the rows of this diagram are exact. In combination with
  \propref{prop:Dual-iso}, \lemref{lem:Coh-V2} and \cite[Prop.~3.4]{Kato-Saito-Ann}
  (see also \cite[Thm.~4.7]{KRS}), these results also imply that
  $(\phi^*_{\<\n\>})^\vee$ is a monomorphism. This proves that the desired hypothesis
  is satisfied.

We conclude from \lemref{lem:Top-1}
  that with respect to the weak topology of $G^1_{\<\n\>_+}(m)$,
  the map $(\phi^*_{\<\n\>})^\vee \colon F^1_{\<\n\>}(m) \to G^1_{\<\n\>_+}(m)^\star$ is a
  bijection. We thus get a commutative diagram
  \begin{equation}\label{eqn:Top-3}
  \xymatrix@C.8pc{
    0 \ar[r] & F^1_0(m) \ar[r] \ar[d]_-{(\theta^*)^\vee} &
    F^1_{\<\n\>}(m) \ar[r] \ar[d]^-{(\phi^*_{\<\n\>})^\vee} &
    \H^1_\et(X, W_m\sV^0_{\<\n\>}) \ar[d]^-{(\psi^*_{\<\n\>})^\vee} \ar[r] &
    F^2_0(m) \ar[d]^-{(\theta^*)^\vee} \\
    0 \ar[r] & G^1_0(m)^\star \ar[r] &
    G^1_{\<\n\>_+}(m)^\star \ar[r]^-{\partial^\vee} &
    H^0_\et(X,W_m \sW^2_{\<{\n}\>_+})^\star \ar[r] & G^0_0(m)^\star}
  \end{equation}
  whose rows are exact and the vertical arrows are bijections, where
  $H^0_\et(X,W_m \sW^2_{\<{\n}\>_+})$ is endowed with the Kato topology.

\begin{remk}\label{remk:Top-4}
    Via the pairing $F^1_0(m) \times G^1_0(m) \to {\Z}/{p^m}$, we can endow
    $G^1_0(m)$ with the weak topology induced by
    $F^1_0(m)$. Using the Pontryagin duality between the
    profinite topology of $G^1_0(m)$ and the discrete
    topology of $F^1_0(m)$ (cf. \cite[Thm~4.7]{KRS}),
    it is easy to check that
    $\{\Ker(\beta(f))|f \in F^1_0(m) \}$ forms a sub-base for
    the profinite topology of $G^1_0(m)$. It follows that
    the weak topology of $G^1_0(m)$ coincides with the
    profinite topology. Furthermore, $G^1_{\<\n\>_+}(m) \to  G^1_0(m)$
    is continuous.
    %because the latter has the weak topology induced by $F^1_0(m)$ and the composite map $G^1_{\<\n\>_+}(m) \to  G^1_0(m) \xrightarrow{\beta(f)} {\Z}/{p^m}$ is clearly continuous for every $f \in F^1_0(m) \subseteq F^1_{\<\n\>}(m)$. 
    \end{remk}

\begin{lem}\label{lem:Top-5}
      The profinite completion map $G^1_{\<\n\>_+}(m) \to (G^1_{\<\n\>_+}(m))^{\pf}$
        is injective, and the induced subspace topology of $G^1_{\<\n\>_+}(m)$
      coincides with its weak topology. In particular, the latter topology is
      Hausdorff.
    \end{lem}
\begin{proof}
  By \corref{cor:PF-compln-0}, the injectivity claim is equivalent to that
  $\ev_G \colon G^1_{\<\n\>_+}(m) \to ((G^1_{\<\n\>_+}(m))^\star)^\vee$ is
      injective.
      By ~\eqref{eqn:Top-3}, the latter claim is equivalent to that the map
      $\phi^*_{\<\n\>}$ in ~\eqref{eqn:Pair-0-4-0} is injective. 
      But this follows because $\psi^*_{\<\n\>}$ is injective by
      \corref{cor:Local-inj-6} and \lemref{lem:Coh-V2} while the maps $\theta^*$
      are bijective by \cite[\S~3, Prop.~4]{Kato-Saito-Ann} (see also
      \cite[Thm.~4.7]{KRS}). The claim, that the weak topology of $G^1_{\<\n\>_+}(m)$
      coincides with the subspace
      topology induced from the profinite topology of $(G^1_{\<\n\>_+}(m))^\pf$,
      follows from \lemref{lem:Top-1} whose proof shows that every open subgroup of
      $G^1_{\<\n\>_+}(m)$ has finite index.
      \end{proof}

\begin{lem}\label{lem:Top-6}
     For $\n \ge \n'$, the canonical map $G^1_{\<\n\>_+}(m) \to G^1_{\<n'\>_+}(m)$ is
     continuous and the induced map $(G^1_{\<\n\>_+}(m))^\pf \to (G^1_{\<\n'\>_+}(m))^\pf$
     is a topological quotient.
   \end{lem}
   \begin{proof}
We look at the commutative diagram
     \begin{equation}\label{eqn:Top-6-0}
       \xymatrix@C2pc{
         G^1_{\<\n\>_+}(m) \ar@{->>}[r] \ar@{^{(}->}[d] & G^1_{\<n'\>_+}(m)
           \ar@{^{(}->}[d] \\
           (G^1_{\<\n\>_+}(m))^\pf \ar[r] &  (G^1_{\<\n'\>_+}(m))^\pf.}
         \end{equation}
         \lemref{lem:Coh-V2} implies that
         the transition map $\H^1_\et(X, W_m\sV^0_{\<\n'\>}) \to
         \H^1_\et(X, W_m\sV^0_{\<\n\>})$ is injective.
By comparing the exact sequence in the top row of ~\eqref{eqn:Top-2} with the
similar exact sequence for $\<\n'\>$, we see that the transition map
$F^1_{\<\n'\>}(m) \to F^1_{\<\n\>}(m)$ is injective. 
We conclude from ~\eqref{eqn:Top-3} and \corref{cor:PF-compln-0} that
the bottom horizontal arrow in ~\eqref{eqn:Top-6-0} is a surjective continuous
homomorphism between two profinite groups. It must therefore be a topological
quotient map. Finally, the claim, that the horizontal arrow in the top row of
~\eqref{eqn:Top-6-0} is continuous, follows from \lemref{lem:Top-5} which
asserts that each term of this row is a subspace of the corresponding term in the
bottom row.
\end{proof}

\begin{remk}\label{remk:Top-6-1}
  It is unclear that the  transition map $G^1_{\<\n\>_+}(m) \surj G^1_{\<n'\>_+}(m)$
  is a quotient map in general.
\end{remk}

By combining \corref{cor:PF-compln-0}, ~\eqref{eqn:Top-3} and the Pontryagin
duality, we get the following special case of the duality  theorem. We let
$\n \ge 0$.

   \begin{cor}\label{cor:Duality-Spl}
     For $i \ge 0$, there is a perfect pairing of
     topological abelian groups
     \[
       (G^i_{\<\n\>_+}(m))^{\pf} \times F^{2-i}_{\<\n\>}(m) \to {\Z}/{p^m}.
     \]
   \end{cor}

\subsection{Topology of $G^i_\n(m)$ in general case}\label{sec:Top-G}
To endow $G^1_\n(m)$ with the structure of a topological abelian group in general,
we fix $\n \ge 0$ and let $\phi(\n) = (m_1, \cdots , m_r)$ be such that
$m_i$ is the smallest non-negative integer having the property that
$p^{m-1}m_i+1 \ge n_i$ for $1 \le i \le r$. In particular, the canonical map
$G^1_{\<\phi(\n)\>_+}(m) \to G^1_\n(m)$ is surjective by \lemref{lem:Surj-TM}. We endow
$G^1_\n(m)$ with the quotient of the weak topology of $G^1_{\<\phi(\n)\>_+}(m)$ via this
surjection. This makes $G^1_\n(m)$ into a topological abelian group.

\begin{lem}\label{lem:Top-7}
  For $\n \ge \n'$, the transition map $\gamma^\n_{\n'} \colon
  G^1_\n(m) \to G^1_{\n'}(m)$ is continuous and the induced map
$(G^1_{\n}(m))^\pf \to (G^1_{\n'}(m))^\pf$ is a topological quotient.
\end{lem}
\begin{proof}
By definition of $\phi(\n)$, there is a commutative diagram
  \begin{equation}\label{eqn:Top-7-0}
    \xymatrix@C2pc{
      G^1_{\<\phi(\n)\>_+}(m) \ar@{->>}[r] \ar@{->>}[d] & G^1_{\<\phi(\n')\>_+}(m)
      \ar@{->>}[d] \\
      G^1_\n(m) \ar@{->>}[r]  & G^1_{\n'}(m).}
  \end{equation}
vertical arrows are quotient maps. It follows that the bottom horizontal arrow
  is also continuous. For the second part, we note that any quotient of
  $G^1_\n(m)$ by an open subgroup is also a quotient of $G^1_{\<\phi(\n)\>_+}(m)$ by an
  open subgroup. It follows therefore from \cite[Lem.~1.1.5]{Pro-fin} that
  the map $(G^1_{\<\phi(\n)\>_+}(m))^{\pf} \to (G^1_\n(m))^\pf$ is surjective.
We now use the commutative diagram obtained by taking the profinite completions of
  the groups appearing in ~\eqref{eqn:Top-7-0} and use \lemref{lem:Top-6} to
  conclude that the map $(G^1_{\n}(m))^\pf \to (G^1_{\n'}(m))^\pf$ is a topological
  quotient.
\end{proof}

\begin{lem}\label{lem:Top-8}
  For every $\n \ge 0$, the topological abelian group $G^1_\n(m)$ is Hausdorff.
\end{lem}
\begin{proof}
We shall prove the equivalent statement that the kernel of the transition map
 $\alpha_\n \colon G^1_{\<\phi(\n)\>_+}(m) \surj G^1_\n(m)$ is closed. By 
 the commutative diagram
 \begin{equation}\label{eqn:Top-8-0}
   \xymatrix@C.8pc{
     G^0_0(m) \ar[r] \ar@{=}[d] & H^0_\et(X, W_m\sW^2_{\<\phi(\n)\>_+})
     \ar[r]^-{\partial_0}
     \ar[d]^-{\beta_\n} & G^1_{\<\phi(\n)\>_+}(m) \ar[r] \ar[d]^-{\alpha_\n} &
     G^1_0(m) \ar[r] \ar@{=}[d] & 0 \\
     G^0_0(m) \ar[r] & H^0_\et(X, W_m\sW^2_\n) \ar[r] & G^1_\n(m) \ar[r] &
     G^1_0(m) \ar[r] &
     0}
   \end{equation}
   whose rows are exact,
   it suffices to show that $\partial_0(\Ker(\beta_\n))$ is closed in
   $G^1_{\<\phi(\n)\>_+}(m)$.

Now, we look at the commutative diagram of continuous homomorphisms
   \begin{equation}\label{eqn:Top-8-1}
     \xymatrix@C2pc{
       H^0_\et(X, W_m\sW^2_{\<\phi(\n)\>_+}) \ar[r]^-{\psi^*_{\<\phi(\n)\>}}
       \ar[d]_-{{\beta_\n}} &
       \H^1_\et(X, W_m\sV^0_{\<\phi(\n)\>})^\vee \ar[d]^-{\beta'_\n} \\
      H^0_\et(X, W_m\sW^2_\n) \ar[r]^-{\psi^*_{\n_-}} &
      \H^1_\et(X, W_m\sV^0_{\n_-})^\vee.}
  \end{equation}
The horizontal arrows in this diagram are injective by \corref{cor:Local-inj-6} and
  \lemref{lem:Coh-V2}. This implies that
  \begin{equation}\label{eqn:Top-8-4}
    \Ker(\beta_\n) = (\psi^*_{\<\phi(\n)\>})^{-1}(\Ker(\beta'_\n)).
  \end{equation}
  In particular, $\Ker(\beta_\n)$ is closed in $H^0_\et(X, W_m\sW^2_{\<\phi(\n)\>_+})$.

We next consider the commutative diagram
  \begin{equation}\label{eqn:Top-8-2}
     \xymatrix@C1pc{
G^0_0(m) \ar[r]^-{\delta} \ar[d]_-{\cong} & H^0_\et(X, W_m\sW^2_{\<\phi(\n)\>_+}) 
\ar[d]^-{\psi^*_{\<\phi(\n)\>}} \ar[r]^-{\partial_0} & 
G^1_{\<\phi(\n)\>_+}(m) \ar[r] \ar[d]^-{\phi^*_{\<\phi(\n)\>}} & G^1_0(m) \ar[d]^-{\cong}
\ar[r] & 0 \\
(F^0_0(m))^\vee \ar[r]^-{\delta'} & \H^1_\et(X, W_m\sV^0_{\<\phi(\n)\>})^\vee 
\ar[r]^-{\gamma^\vee_{\<\phi(\n)\>}} & (F^1_{\<\phi(\n)\>}(m))^\vee \ar[r] &
(F^1_0(m))^\vee \ar[r] & 0.}
\end{equation}
Since $\gamma^\vee_{\<\phi(\n)\>}$ is a continuous homomorphism between profinite
abelian groups and $\Ker(\beta'_\n)$ is closed in
$\H^1_\et(X, W_m\sV^0_{\<\phi(\n)\>_+})^\vee$, it follows that
$\gamma^\vee_{\<\phi(\n)\>}(\Ker(\beta'_\n))$ is closed in
$(F^1_{\<\phi(\n)\>}(m))^\vee$. In particular,
$(\phi^*_{\<\phi(\n)\>})^{-1}(\gamma^\vee_{\<\phi(\n)\>}(\Ker(\beta'_\n))$ is closed in
$G^1_{\<\phi(\n)\>_+}(m)$. It remains therefore to show that
\begin{equation}\label{eqn:Top-8-3}
  \partial_0(\Ker(\beta_\n)) = 
  (\phi^*_{\<\phi(\n)\>})^{-1}(\gamma^\vee_{\<\phi(\n)\>}(\Ker(\beta'_\n)).
  \end{equation}

Let  $x \in  (\phi^*_{\<\phi(\n)\>})^{-1}(\gamma^\vee_{\<\phi(\n)\>}(\Ker(\beta'_\n))$
so that $\phi^*_{\<\phi(\n)\>}(x) = \gamma^\vee_{\<\phi(\n)\>}(y)$ for some
$y \in \Ker(\beta'_\n)$. By a diagram chase of ~\eqref{eqn:Top-8-2}, we get
$a \in H^0_\et(X, W_m\sW^2_{\<\phi(\n)\>_+})$ and
$b \in G^0_0(m)$ such that $x = \partial_0(a)$ and
$\psi^*_{\<\phi(\n)\>}(a - \delta(b)) = y \in \Ker(\beta'_\n)$.
It follows from ~\eqref{eqn:Top-8-4} that
$a' := a - \delta(b) \in \Ker(\beta_\n)$ and $x = \partial_0(a) =
 \partial_0(a')$. This shows that 
$(\phi^*_{\<\phi(\n)\>})^{-1}(\gamma^\vee_{\<\phi(\n)\>}(\Ker(\beta'_\n)) \subseteq
\partial_0(\Ker(\beta_\n))$. The reverse inclusion follows by ~\eqref{eqn:Top-8-1}.
\end{proof}

\begin{comment}
If $x \in \Ker(\beta_\n)$, then $\phi^*_{\<\phi(\n)\>} \circ \partial_0(x) =
  \gamma^\vee_{\<\phi(\n)\>} \circ \psi^*_{\<\phi(\n)\>}(x) =  \gamma^\vee_{\<\phi(\n)\>}(y)$,
  where $y \in \Ker(\beta'_\n)$ by ~\eqref{eqn:Top-8-1}. This shows that
  $\partial_0(\Ker(\beta_\n)) \subseteq
  (\phi^*_{\<\phi(\n)\>})^{-1}(\gamma^\vee_{\<\phi(\n)\>}(\Ker(\beta'_\n))$.
To prove the reverse inclusion, we let $x \in 
(\phi^*_{\<\phi(\n)\>})^{-1}(\gamma^\vee_{\<\phi(\n)\>}(\Ker(\beta'_\n))$
so that $\phi^*_{\<\phi(\n)\>}(x) = \gamma^\vee_{\<\phi(\n)\>}(y)$ for some
$y \in \Ker(\beta'_\n)$. By a diagram chase of ~\eqref{eqn:Top-8-2}, we get
$x = \partial_0(a)$ for some $a \in H^0_\et(X, W_m\sW^2_{\<\phi(\n)\>_+})$.
Another diagram chase of ~\eqref{eqn:Top-8-2} shows that there is some
$b \in G^0_0(m)$ such that $\psi^*_{\<\phi(\n)\>}(a - \delta(b)) = y \in
\Ker(\beta'_\n)$. It follows from ~\eqref{eqn:Top-8-4} that
$a' := a - \delta(b) \in \Ker(\beta_\n)$. Furthermore, $x = \partial_0(a) =
 \partial_0(a')$. This proves the other inclusion of ~\eqref {eqn:Top-8-3}
and concludes the proof of the lemma.
\end{proof}
\end{comment}

\begin{cor}\label{cor:Top-9}
  For every $\n \ge 0$, the profinite completion map
  $G^1_\n(m) \to (G^1_\n(m))^{\pf}$ is
  injective and the topology of $G^1_\n(m)$ coincides with the subspace topology
  induced from the profinite topology of $(G^1_\n(m))^{\pf}$.
\end{cor}
\begin{proof}
  In view of \lemref{lem:Top-8}, it suffices to show that every open subgroup
  of $G^1_\n(m)$ has finite index. But this is clear because this property
  holds for $G^1_{\<\phi(\n)\>_+}(m)$ whose topological quotient is $G^1_\n(m)$.
\end{proof}

\subsection{Duality for $X^o$}\label{sec:D-open}
We let $G^i_m({X^o}) = {\varprojlim}_{\n} (G^i_\n(m))^{\pf}$ and endow it with the
structure of a profinite abelian group by means of the inverse limit topology.
Note that the canonical map ${\varprojlim}_{\n} G^i_\n(m) \to G^i_m({X^o})$ is a
topological
isomorphism for $i \neq 1$. We let $H^i_\et(X^o, {\Z}/{p^m})$ have the discrete
topology for each $i \ge 0$. We end \S~\ref{sec:DThm} with the following duality
theorem for $X^o$. This extends the duality theorem of Kato-Saito
(cf. \cite[\S~3, Prop.~4]{Kato-Saito-Ann}) to open curves.

\begin{thm}\label{thm:Duality-M}
  For each $i \ge 0$, there is a perfect pairing of topological abelian
  groups
  \begin{equation}\label{eqn:Duality-M-0}
    G^i_m({X^o}) \times H^{2-i}_\et(X^o, {\Z}/{p^m}) \to {\Z}/{p^m}.
  \end{equation}
\end{thm}
\begin{proof}
 First of all, it follows easily by a combination of \lemref{lem:Fil-comp}
  and the argument given at the outset of \cite[\S~4]{JSZ} that there is a
  canonical isomorphism
  ${\varinjlim}_\n F^i_\n(m) \xrightarrow{\cong} H^i_\et(X^o, {\Z}/{p^m})$.
  Now, the $i =0$ case of the theorem follows directly from ~\eqref{eqn:Top-0},
  which says that the duality statement in fact holds for every $\n$. One can also
  check using the ampleness of $D_\n$ that $F^2_\n(1) = 0$ for $\n \gg 0$. Using
  \corref{cor:F-V-R} and induction on $m$, one gets that
  $F^2_\n(m) = 0 = G^0_\n(m)$ for $\n \gg 0$.
  As shown in ~\S~\ref{sec:Top}, the $i =2$ case is independent of $\n$ and the
  duality statement holds by \cite[Prop.~3.4]{Kato-Saito-Ann} (see also
  \cite[Thm.~4.7]{KRS}). 

 To prove $i =1$ case, we first note that there are topological isomorphisms
  $G^1_m({X^o}) \xrightarrow{\cong} {\varprojlim}_\n G^1_{\<\n\>_+}(m)$ and
  ${\varinjlim}_\n F^1_{\<\n\>_+}(m) \xrightarrow{\cong} H^{1}_\et(X^o, {\Z}/{p^m})$.
  We now combine \corref{cor:Duality-Spl}, \cite[Lem.~5.2]{Gupta-Krishna-Duality}
  and \cite[Lem.~2.6]{KRS} to conclude the proof.
\end{proof}

\begin{remk}\label{remk:Duality-M-1}
  If we let $H^1_{c, \et}(X^o, W_m\Omega^2_{X^o, \log}) :=
  H^1_\et(X, j_! W_m\Omega^2_{X^o, \log})$, we can endow
  $H^1_{c, \et}(X^o, W_m\Omega^2_{X^o, \log})$ with
  the weak topology induced by $H^1_\et(X^o, {\Z}/{p^m})$ via the pairing
  $H^1_\et(X^o, {\Z}/{p^m}) \times H^1_{c, \et}(X^o, W_m\Omega^2_{X^o, \log}) \to
  {\Z}/{p^m}$. It can be shown that the
  canonical map $H^1_{c, \et}(X^o, W_m\Omega^2_{X^o, \log}) \to G^1_m(X^o)$ induces an
  isomorphism of topological abelian groups $H^1_{c, \et}(X^o, W_m\Omega^2_{X^o, \log})^\pf
  \xrightarrow{\cong} G^1_m(X^o)$. In particular,
  ~\eqref{eqn:Duality-M-0} for $i =1$ can be written as
  \[
    H^1_{c, \et}(X^o, W_m\Omega^2_{X^o, \log})^\pf  \times
    H^{1}_\et(X^o, {\Z}/{p^m}) \to {\Z}/{p^m}.
  \]
  \end{remk}
% This does not hold for $i =0$ as $H^0_{c, \et}(X^o, W_m\Omega^2_{X^o, \log}) = 0$.

\section{Class field theory with modulus}\label{sec:CFT-mod}
We fix a local field $k$ of characteristic $p > 0$ and a connected, smooth and
projective curve $X$ over $k$. 
We let $D_\dagger = \{x_1, \ldots , x_s\}$ be a reduced effective divisor on $X$
and let $j \colon X^o \inj X$ be the inclusion of the complement of $D$.
Let $K$ denote the function field of $X$.
In this section, we shall prove a general form of \corref{cor:Duality-Spl}
and derive the class field of $X$ with modulus.
We shall continue to maintain the notations of
\S~\ref{sec:Exist} and \S~\ref{sec:DThm}.

\subsection{Duality theorem with modulus}\label{sec:DTM*}
We fix an integer $m \ge 1$. We let $\psi(n) = p^{m-1}n +1$ for $n \ge 0$.
We let $\rho^m_{X^o}\colon  \frac{C(X^o)}{p^m} \to
\frac{\pi^{\ab}_1(X^o)}{p^m}$ and $\rho^m_{X|D}\colon C_m(X,D) \to
\frac{\pi^{\ab}_1(X, D)}{p^m}$ be the continuous homomorphisms
induced by the reciprocity maps of ~\eqref{eqn:Rec-curve-Main}.

We let $\rho^\et_{X^o} \colon G^1_m({X^o}) \to \frac{\pi^{\ab}_1(X^o)}{p^m}$
and $\lambda^\et_{X^o} \colon \frac{\pi^{\ab}_1(X^o)}{p^m} \to  G^1_m({X^o})$
be the continuous homomorphisms induced by the perfect pairing of
~\eqref{eqn:Duality-M-0} via the identification
$H^{1}_\et(X^o, {\Z}/{p^m})^\vee \cong \frac{\pi^{\ab}_1(X^o)}{p^m}$.
We consider the sequence of homomorphisms
\begin{equation}\label{eqn:Rec-dual-0}
  \frac{C(X^o)}{p^m} \to {\varprojlim}_\n \ C_m(X, D_{\<\n\>_+})
  \inj  {\varprojlim}_\n \ C^\et_m(X, D_{\<\n\>_+}) \inj 
  G^1_m({X^o}) \xrightarrow{\rho^\et_{X^o}} \frac{\pi^{\ab}_1(X^o)}{p^m}.
  \end{equation}
  The two arrows in the middle are injective by
  \propref{prop:Nis-et-ICG} and \corref{cor:Top-9}.
  We let $\wt{\rho}^m_{X^o}$ denote the composite homomorphism.

\begin{lem}\label{lem:Rec-dual-1}
  We have $\wt{\rho}^m_{X^o} = \rho^m_{X^o}$.
  \end{lem}
  \begin{proof}
    We let $x \in D_\dagger$ and look at the diagrams
    \begin{equation}\label{eqn:Rec-dual-1-0}
      \xymatrix@C.8pc{
        \frac{K_2(K^h_x)}{p^m} \ar[rrr]^-{\rho_{K^h_x}} \ar[dd] \ar[dr] & & &
        \H^1_\et(K^h_x, W_m\sF^0_{\<\n\>})^\vee  \ar[d] & \\
        & {\varprojlim}_\n H^1_{x, \ \nis}(X^h_x, \sK^m_{X|D_{\<\n\>_+}})
        \ar[d] \ar[r]^-{\cong} &
        {\varprojlim}_\n H^1_x(X^h_x, \sK^m_{X|D_{\<\n\>_+}}) \ar[r] \ar[d] &
        {\varprojlim}_\n \H^1_\et(X^h_x, W_m\sF^0_{\<\n\>})^\vee \ar[d] \\
 \frac{C(X^o)}{p^m} \ar[r] & {\varprojlim}_\n C_m(X, D_{\<\n\>_+})
 \ar[r] &  {\varprojlim}_\n C^\et_m(X, D_{\<\n\>_+}) \ar[r] &
 {\varprojlim}_\n (F^1_{\<\n\>}(m))^\vee}
\end{equation}
and
\begin{equation}\label{eqn:Rec-dual-1-1}
  \xymatrix@C.8pc{
 \H^1_\et(K^h_x, W_m\sF^0_{\<\n\>})^\vee  \ar[r] \ar[d]_-{\cong} &
 {\varprojlim}_\n \H^1_\et(X^h_x, W_m\sF^0_{\<\n\>})^\vee \ar[r] &
 {\varprojlim}_\n (F^1_{\<\n\>}(m))^\vee \ar[d]^-{\cong} \\
  H^1_\et(K^h_x, {\Z}/{p^m})^\vee \ar[rr] & & H^1_\et(X^o, {\Z}/{p^m})^\vee.}
 \end{equation}

By ~\eqref{eqn:Nis-et-ICG-3} and \lemref{lem:Coh-V0}, the map
$H^1_x(X^h_x, \sK^m_{X|D_{\<\n\>_+}}) \to \H^1_\et(X^h_x, W_m\sF^0_{\<\n\>})^\vee$
in the middle row of ~\eqref{eqn:Rec-dual-1-0}
coincides (for a fixed $\<\n\>$) with the
composition of Kato's reciprocity homomorphism
with the projection $\frac{K_2(K^h_x)}{K_2(A^h_x|\psi(n_x)) + p^m K_2(K^h_x)} \to
(\Fil_{n_x} H^1_\et(K^h_x, {\Z}/{p^m}))^\vee$. In particular, the top square in
~\eqref{eqn:Rec-dual-1-0} commutes. The bottom square clearly commutes because
it is the local-global comparison square for the cohomology pairing of
\propref{prop:Pair-0-4}. The diagram~\eqref{eqn:Rec-dual-1-1} is easily seen to be
commutative.

The composition of the horizontal arrow on the top of
~\eqref{eqn:Rec-dual-1-0} and the left vertical arrow in ~\eqref{eqn:Rec-dual-1-1}
is $\rho_{K^h_x}$. On the other hand, the composition of the horizontal rows on the
bottom of ~\eqref{eqn:Rec-dual-1-0} and the right vertical arrow in
~\eqref{eqn:Rec-dual-1-1} is $\wt{\rho}^m_{X^o}$. 
It follows that $\rho^m_{X^o}$ and $\wt{\rho}^m_{X^o}$ coincide when
restricted to $\frac{K_2(K^h_x)}{p^m}$. A similar and easier argument shows that
$\rho^m_{X^o}$ and $\wt{\rho}^m_{X^o}$ coincide when
restricted to $\frac{K_1(k(x))}{p^m}$ if $x\in X^o_{(0)}$. As
$F_1(X^o) \bigoplus F_2(X^o) \surj C(X^o)$, the lemma follows.
\end{proof}

\begin{lem}\label{lem:Rec-dual-3}
    For any $\n \ge 1$, we have a Cartesian square
    \begin{equation}\label{eqn:Rec-dual-3-0}
      \xymatrix@C2pc{
        \Fil_{D_\n} H^1_\et(K, {\Z}/{p^m}) \ar[r]^-{(\rho^m_{X|D_\n})^\vee}
        \ar@{^{(}->}[d]_-{p^\vee_\n} &
C_m(X,D_\n)^\vee \ar@{^{(}->}[d]^-{s^\vee_\n} \\
H^1_\et(X^o, {\Z}/{p^m}) \ar[r]^-{(\rho^m_{X^o})^\vee} &
(\frac{C(X^o)}{p^m})^\vee.}
\end{equation}     
\end{lem}
\begin{proof}
 Using the commutative diagram
\begin{equation}\label{eqn:Rec-dual-3-1}
    \xymatrix@C2pc{
      H^1_\et(X^o, {\Z}/{p^m}) \ar[r]^-{(\rho^m_{X^o})^\vee} \ar[d]
      & (\frac{C(X^o)}{p^m})^\vee
      \ar[d] \\
      H^1_\et({K}^h_x, {\Z}/{p^m}) \ar[r]^-{\rho^m_{\wh{K}_x}} &
      (\frac{K_2({K}^h_x)}{p^m})^\vee}
  \end{equation}
for every $x \in D_\dagger$, the lemma is easily deduced from 
\cite[Rem.~6.6]{Kato89} and \corref{cor:p-tor-*}.
\end{proof}

 Let $\nu_\n \colon C_m(X, D_\n) \to C^\et_m(X, D_\n)$ denote the {\'e}tale
  realization map. 

 \begin{lem}\label{lem:Nis-et-Cartesian}
   For any $\n \ge 1$, we have a Cartesian square
\begin{equation}\label{eqn:Nis-et-Cartesian-0}
      \xymatrix@C2pc{
        ((G^1_\n(m))^{\pf})^\star \ar[d]_-{\nu^\vee_\n} \ar[r]^-{q^\vee_\n}  &
        G^1_m(X^o)^\star \ar[d]^-{(\rho^m_{X^o})^\vee} \\
       C_m(X, D_\n)^\vee \ar[r]^-{s^\vee_\n} &   (\frac{C(X^o)}{p^m})^\vee.}
\end{equation}
\end{lem}
  \begin{proof}
    We let $\n' = \<\phi(\n)\>_+$ and observe that
    ~\eqref{eqn:Nis-et-Cartesian-0} has a factorization
   \begin{equation}\label{eqn:Nis-et-Cartesian-1}
      \xymatrix@C2pc{
        ((G^1_\n(m))^{\pf})^\star \ar[d]_-{\nu^\vee_\n} \ar[r]  &
((G^1_{\n'}(m))^{\pf})^\star \ar[d]^-{\nu^\vee_{\n'}} \ar[r]^-{q^\vee_{\n'}}  &
G^1_m(X^o)^\star \ar[d]^-{(\rho^m_{X^o})^\vee} \\
C_m(X, D_\n)^\vee \ar[r] &  C_m(X, D_{\n'})^\vee \ar[r]^-{s^\vee_{\n'}}
& (\frac{C(X^o)}{p^m})^\vee,}
\end{equation} 
where the left horizontal arrows are obtained by dualizing the transition
maps.

We claim that the left square in ~\eqref{eqn:Nis-et-Cartesian-1}
is Cartesian. To show this, we can replace $(G^1_\n(m))^{\pf}$
(resp. $(G^1_{\n'}(m))^{\pf}$) by $C^\et_m(X, D_\n)$ (resp. $C^\et_m(X, D_{\n'})$)
by \lemref{lem:PF-compln}. By dualizing the exact sequence of
\propref{prop:Nis-et-ICG}, it already follows now that the left square in
~\eqref{eqn:Nis-et-Cartesian-1} is Cartesian if we use
$(-)^\vee$ instead of $(-)^\star$ in its top row. The claim now follows because
$C^\et_m(X, D_{\n'}) \surj  C^\et_m(X, D_\n)$ is a topological quotient map.

To show that the right square in  ~\eqref{eqn:Nis-et-Cartesian-1}
is Cartesian, we can replace $((G^1_{\n'}(m))^{\pf})^\star$ and $G^1_m(X^o)^\star$
by $\Fil_{D_{\n'}} H^1_\et(K, {\Z}/{p^m})$ and
$H^1_\et(X^0, {\Z}/{p^m})$, respectively, by \propref{prop:KSL},
\corref{cor:Duality-Spl} and \thmref{thm:Duality-M}.
We can now apply \lemref{lem:Rec-dual-3} to finish the proof.
\end{proof}

\begin{lem}\label{lem:Rec-dual-4}
 The map $\rho^{\et}_{X^o}\colon G^1_m(X^o) \to
  \frac{\pi^{\ab}_1(X^o)}{p^m}$ induces a continuous homomorphism
  $\rho^\et_{X|D_\n} \colon C^\et_m(X, D_\n) \to \frac{\pi^{\ab}_1(X,D_\n)}{p^m}$
  for every $\n \ge 0$ such that 
  \begin{equation}\label{eqn:Rec-dual-4-0}
    \xymatrix@C2pc{
      G^1_m(X^o) \ar[r]^-{\rho^{\et}_{X^o}} \ar@{->>}[d]_-{q_\n} &  
      \frac{\pi^{\ab}_1(X^o)}{p^m} \ar@{->>}[d]^-{p_\n} \\
      (G^1_\n(m))^{\pf} \ar[r]^-{\rho^\et_{X|D_\n}} &
      \frac{\pi^{\ab}_1(X,D_\n)}{p^m}}
  \end{equation}
  is a commutative diagram of continuous homomorphisms between topological abelian
  groups. Moreover, $\rho^\et_{X|D_\n}$ is a topological isomorphism.
\end{lem}
\begin{proof}
 We shall assume $\n \ge 1$ as the lemma is clear otherwise.
For the first part, we only need to show the existence of $\rho^\et_{X|D_\n}$
  such that ~\eqref{eqn:Rec-dual-4-0} commutes as the continuity follows from
  \lemref{lem:Top-7}.
  % which says that $q_\n$ is a topological quotient.
  We now look at the diagram
\begin{equation}\label{eqn:Rec-dual-2-2}
    \xymatrix@C3pc{
\Fil_{D_\n} H^1_\et(K, {\Z}/{p^m}) \ar@{^{(}->}[d]_-{p^\vee_\n}  \ar@{.>}[r]
\ar@/^2pc/[rr]^-{(\rho^m_{X|D_\n})^\vee} &
(G^1_\n(m))^{\pf})^\star \ar@{^{(}->}[d]_-{q^\vee_\n}  \ar[r]^-{\nu^\vee_\n} & 
C_m(X,D_\n)^\vee \ar@{^{(}->}[d]^-{s^\vee_\n} \\
H^1_\et(X^o, {\Z}/{p^m}) \ar@/_2pc/[rr]_-{(\rho^m_{X^o})^\vee}
\ar[r]^-{(\rho^{\et}_{X^o})^\vee}_-{\cong} & (G^1_m(X^o))^\star \ar[r]^-{\nu^\vee} &
(\frac{C(X^o)}{p^m})^\vee.}
\end{equation}

The outer square is commutative by the compatibility between
$\rho_{X^o}$ and $\rho_{X|D_\n}$,
and $\nu^\vee \circ (\rho^{\et}_{X^o})^\vee = (\rho^m_{X^o})^\vee$ by
\lemref{lem:Rec-dual-1}.
By \lemref{lem:Nis-et-Cartesian}, a diagram chase yields
a continuous homomorphism
$(\rho^\et_{X|D_\n})^\star \colon \Fil_{D_\n} H^1_\et(K, {\Z}/{p^m}) \to
(G^1_\n(m))^{\pf})^\star$ such that the left square in ~\eqref{eqn:Rec-dual-2-2}
commutes.
Letting $\rho^\et_{X|D_\n}$ to be the dual this map, we conclude the proof of the
first part of the lemma.
For the second part, we only need to show that
$(\rho^\et_{X|D_\n})^\star$ is surjective. But this follows 
by Lemmas~\ref{lem:Rec-dual-3} and ~\ref{lem:Nis-et-Cartesian}.
\end{proof}

The following stronger version of \corref{cor:Duality-Spl} 
proves \thmref{thm:Main-4}. For $D = \emptyset$, this was shown by Kato-Saito
\cite[\S~3, Prop.~4]{Kato-Saito-Ann} by a different method.

\begin{thm}\label{thm:Duality-Gen}
  For any effective divisor $D \subset X$ and $i \ge 0$,
  there is a perfect pairing of topological abelian groups
  \[
    H^i_\et(X, W_m\Omega^2_{X|D, \log})^\pf \times H^{2-i}_\et(X, W_m\Omega^0_{X|D,\log})
    \to {\Z}/{p^m}.
  \]
such that for $D' \ge D$, one has a commutative diagram
  \begin{equation}\label{eqn:Duality-Gen-0}
    \xymatrix@C2pc{
      H^i_\et(X, W_m\Omega^2_{X|D, \log})^{\pf} \times  
      H^{2-i}_\et(X, W_m\Omega^0_{X|D,\log}) \ar@<10ex>[d]  \ar[r] &
      {\Z}/{p^m} \ar@{=}[d] \\
      H^i_\et(X, W_m\Omega^2_{X|D', \log})^{\pf}  \times  
      H^{2-i}_\et(X, W_m\Omega^0_{X|D',\log}) \ar@<10ex>[u] \ar[r] &
      {\Z}/{p^m}.}
  \end{equation}
\end{thm}
\begin{proof}
 For $i \neq 1$, this is already shown in the proof of \thmref{thm:Duality-M}.
  For $i =1$, the existence of the pairing and its perfectness follow from
  \lemref{lem:Rec-dual-4}. To prove the commutativity of ~\eqref{eqn:Duality-Gen-0},
  we can assume $D' = D_{\n'}$ and $D = D_\n$ for some $\n' \ge \n$ and
    look at the diagram
  \begin{equation}\label{eqn:Duality-Gen-1}
    \xymatrix@C2pc{
      G^1_m(X^o) \ar[d]_-{\rho^{\et}_{X^o}} \ar@{->>}[r]^-{q_{\n'}}
      \ar@/^2pc/[rr]^-{q_\n} &
      (G^1_{\n'}(m))^{\pf} \ar[d]^-{\rho^\et_{X|D_{\n'}}}
      \ar@{->>}[r]^-{\gamma^{\n'}_{\n}} &
       (G^1_{\n}(m))^{\pf} \ar[d]^-{\rho^\et_{X|D_{\n}}} \\
       \frac{\pi^{\ab}_1(X^o)}{p^m} \ar@{->>}[r]^-{p_{\n'}}
       \ar@/_2pc/[rr]_-{p_\n} &
     \frac{\pi^{\ab}_1(X,D_{\n'})}{p^m} \ar@{->>}[r]^-{\lambda^{\n'}_{\n}} &
     \frac{\pi^{\ab}_1(X,D_{\n})}{p^m}.}
\end{equation}
Since the left and the outer squares commute by \lemref{lem:Rec-dual-4} and
$q_{\n'}$ is surjective by \lemref{lem:Top-7},
it follows using \propref{prop:KSL} that the right square
commutes.
\end{proof}

\begin{remk}\label{remk:Duality-Gen-2}
  Using ~\eqref{eqn:Top-8-2} and the proof of \propref{prop:Coker-p},
  one can show that the canonical map $H^1_\et(X, W_m\Omega^2_{X|D, \log}) \to
  H^{1}_\et(X, W_m\Omega^0_{X|D,\log})^\vee$ is not surjective if $D \neq \emptyset$. 
This implies that no matter what topology we endow 
$H^1_\et(X, W_m\Omega^2_{X|D, \log})$ with, \thmref{thm:Duality-Gen} will not hold 
without taking the profinite completion.
\end{remk}

\vskip .3cm

\subsection{Proofs of parts (1) $\sim$ (3) of
  \thmref{thm:Main-2}}\label{sec:RM*}
We need two more lemmas in order to prove 
parts (1) and (2) of \thmref{thm:Main-2}.

\begin{lem}\label{lem:Rec-dual-5}
  For every $m, \n \ge 1$, the composite map
  \[
    \wt{\rho}^m_{X|D_\n} \colon   C_m(X,D_\n) \xrightarrow{\nu_\n} G^1_\n(m) \inj
    (G^1_\n(m))^\pf \xrightarrow{\rho^\et_{X|D_\n}} \frac{\pi^{\ab}_1(X,D_{\n})}{p^m}
  \]
  coincides with $\rho^m_{X|D_\n}$.
\end{lem}
\begin{proof}
We look at the diagram
  \begin{equation}\label{eqn:Rec-dual-5-0}
    \xymatrix@C2pc{
 \frac{C(X^o)}{p^m} \ar[r] \ar@{->>}[dr]_-{s_\n} & {\varprojlim}_\n \ C_m(X, D_{\n})
\ar[r] \ar[d]  & {\varprojlim}_\n \ G^1_\n(m) \ar[r] \ar[d] &
G^1_m({X^o}) \ar[r]^-{\rho^\et_{X^o}} \ar[d]^-{q_\n} &
\frac{\pi^{\ab}_1(X^o)}{p^m} \ar[d]^-{p_\n} \\
& C_m(X, D_\n) \ar[r]^-{\nu_\n} &  G^1_\n(m) \ar[r] & (G^1_\n(m))^\pf
\ar[r]^-{\rho^\et_{X|D_\n}} & \frac{\pi^{\ab}_1(X,D_{\n})}{p^m}.}
\end{equation}
Note that the composite horizontal arrow on the top is $\wt{\rho}^m_{X^o}$.
Since $s_\n$ is surjective, it suffices to show that
$\rho^m_{X|D_\n} \circ s_\n = p_\n \circ \wt{\rho}^m_{X^o}$.
But this follows directly from \lemref{lem:Rec-dual-1},
\lemref{lem:Rec-dual-4} and the compatibility between $\rho_{X^o}$ and $\rho_{X|D_\n}$.
\end{proof}

Let $T^{\ab}_1(X,D)$ denote the kernel of the canonical
surjection $\pi^{\ab}_1(X,D) \surj \pi^{\ab}_1(X)$. Let
$\rho^0_{X|D} \colon C(X,D)_0 \to \pi^{\ab}_1(X,D)_0$ be the geometric part of
the reciprocity map $\rho_{X|D}$.

\begin{lem}\label{lem:Rec-ker-0}
  $T^{\ab}_1(X,D)$ is a torsion group of finite exponent. In particular,
  ${\rm Image}(\rho^0_{X|D})$ is a torsion group of finite exponent.
\end{lem}
\begin{proof}
The second part of the lemma follows directly from its first part by comparing
  $\rho^0_{X|D}$ with $\rho^0_X$ since the image of the latter map is finite
 by \cite{Saito-JNT} and \cite{Yoshida03} (see also \cite[Thm.~1.4]{GKR-arxiv}).
  It remains to prove the first part of the lemma.
We let $D = \stackrel{s}{\underset{i =1}\sum} n_i[x_i]$.

By ~\eqref{eqn:Fil-K-p-1} and \cite[Lem.~8.7]{Gupta-Krishna-Duality},
we have an inclusion
  $\frac{\Fil_D H^1(K)}{H^1(X)} \inj \stackrel{s}{\underset{i =1}\bigoplus}
  \frac{\Fil_{n_i-1} H^1({K}^h_{x_i})}{\Fil_{-1} H^1({K}^h_{x_i})}$.
Equivalently, $\stackrel{s}{\underset{i =1}\bigoplus}
\left(\frac{\Fil_{n_i-1} H^1({K}^h_{x_i})}{\Fil_{-1} H^1({K}^h_{x_i})}\right)^\vee
\surj T^{\ab}_1(X,D)$.
Hence, it suffices to show that for every $x \in D_\dagger$ and integer $n \ge 0$,
$\frac{\Fil_{n} H^1({K}^h_{x})}{\Fil_{-1} H^1({K}^h_{x})}$ has
  finite exponent. By \lemref{lem:Kato-equiv} and the definition of Kato's
ramification filtration, we can replace $K^h_x$ by $\wh{K}_x$.

Now, $\frac{\Fil_{n} H^1(\wh{K}_{x})}{\Fil_{0} H^1(\wh{K}_{x})}$ is annihilated by
  some power of $p$ by \cite[Cor.~3.3]{Kato89} and 
  $\frac{\Fil_{0} H^1(\wh{K}_{x})}{\Fil_{-1} H^1(\wh{K}_{x})}
  \cong \frac{H^1(\wh{K}_{x})\{p'\}}{H^1(\wh{A}_{x})\{p'\}}$
  by ~\eqref{eqn:Fil-K-p-1}.
On the other hand, the localization sequence and purity theorem for
  the {\'e}tale cohomology imply that there is an inclusion
  $\frac{H^1(\wh{K}_{x})\{p'\}}{H^1(\wh{A}_{x})\{p'\}} \xrightarrow{\cong}
  H^0(k(x))\{p'\} \cong (\mu_\infty(k(x)))^\vee$. 
But the last term is finite (of order prime to $p$).
  It follows that $\frac{\Fil_{0} H^1(\wh{K}_{x})}{\Fil_{-1} H^1(\wh{K}_{x})}$
  is finite, and this clearly implies the lemma.
\end{proof}

The following result extends the class field theory
of Kato-Saito \cite[\S~3]{Kato-Saito-Ann} and Hiranouchi
\cite{Hiranouchi-2} to the general case of ramified class field theory.

\begin{thm}\label{thm:Rec-dual-6}
  For every effective divisor $D \subset X$ and every integer $n \ge 1$, the
  reciprocity homomorphism
  \[
    {\rho}_{X|D} \colon   {C(X,D)}/n \to  {\pi^{\ab}_1(X,D)}/n
\]
is injective.
\end{thm}
\begin{proof}
  If $p \nmid n$, then the right hand side can be replaced by
  ${\pi^{\ab}_1(X^o)}/n$ by Definition~\ref{defn:FGMP} and
  \cite[Lem.~2.8]{GKR-arxiv}.
On the other hand, the left hand side can be replaced by ${C(X^o)}/n$.
The corollary now follows from \cite[Thm.~1.1]{Hiranouchi-2} (see also
\cite[Thm.~1.2]{GKR-arxiv}). We can therefore assume $n = p^m$ for some $m \ge 1$
and $D = D_\n$ for some $\n \ge 1$.
By \lemref{lem:Rec-dual-5}, it suffices in this case to show that
$\wt{\rho}^m_{X|D_\n}$ is injective. But this follows by combining
\propref{prop:Nis-et-ICG}, \corref{cor:Top-9} and \lemref{lem:Rec-dual-4}.
\end{proof}

\begin{thm}\label{thm:Rec-ker-1}
   For every effective divisor $D \subset X$, the kernel of the 
  reciprocity map
  \[
    {\rho}_{X|D} \colon   {C(X,D)} \to  {\pi^{\ab}_1(X,D)}
\]
is the maximal divisible subgroup of $C(X,D)$.
\end{thm}
\begin{proof}
  Given \thmref{thm:Rec-dual-6} and \lemref{lem:Rec-ker-0}, we can 
  repeat the proof of \cite[Thm.~5.1]{Saito-JNT}.  We omit the
  details.
\end{proof}

The following result proves part (3) of \thmref{thm:Main-2}.

\begin{prop}\label{prop:Rec-Image}
The image of the map $\rho_{X|D} \colon C(X,D)_0 \to \pi^{\ab}_1(X,D)_0$ 
is not necessarily finite.
\end{prop}
\begin{proof}
It suffices to show that the image of the map
  $\rho_{X|D} \colon {C(X, D)_0}/p \to {\pi^{\ab}_1(X,D)_0}/p$ may not be
  finite. Using \thmref{thm:Rec-dual-6} and the inclusion
  ${C(X,D)_0}/p \inj {C(X,D)}/p$, it suffices to show that
  ${C(X,D)_0}/p$ may not be finite. It further suffices to show that
  $\Ker(\theta_{X|D} \colon {C(X,D)}/p \to {C(X)}/p)$ may not be finite.

  We now take $X = \P^1_k, \ D_\dagger = \{\infty\}$ and $D = 2D_\dagger$. By
  ~\eqref{eqn:Kato-0} and \corref{cor:ICG-BF}, we have have an exact sequence
  \begin{equation}\label{eqn:Rec-Image-0}
    H^0_\nis(X, \Omega^2_{X, \log}) \to \varpi^1_2(A^h_\infty|2) \to 
    \Ker(\theta_{X|D}) \to 0.
    \end{equation}
By the duality, we have $H^0_\nis(X, \Omega^2_{X, \log}) \cong
  H^0_\et(X, \Omega^2_{X, \log}) \cong H^2_\et(X, {\Z}/p)^\star =0$.
  It remains therefore to show that $\varpi^1_2(A^h_\infty|2)$  is infinite.
  But this follows from \cite[\S~2, Prop.~1]{Kato80-1} and \lemref{lem:Kato-1}
  which imply that $\varpi^1_2(A^h_\infty|2)$  is an infinite-dimensional
  $\F_p$-vector space.
\end{proof}

\section{Class field theory of ${X^o}$}\label{sec:Div-Ker}
We maintain the notations of \S~\ref{sec:CFT-mod}. We shall prove \thmref{thm:Main-5}
and prove parts of \thmref{thm:Main-3} in this section.

\subsection{Proofs of parts (1) $\sim$ (3) of
  \thmref{thm:Main-3}}\label{sec:Pf-open}
In this subsection, we shall prove the first three parts of \thmref{thm:Main-3}.
We begin with the injectivity result.
\begin{thm}\label{thm:Rec-ker-8}
  For every integer $n \ge 1$, the map
  ${\rho_{X^o}} \colon {C(X^o)}/n \to {\pi^{\ab}_1(X^o)}/n$ is injective.
 \end{thm}
\begin{proof}
If $p \nmid n$, then the projection maps ${C(X^o)}/n \to 
  {C(X,D_\dagger)}/n$ are isomorphisms by \cite[Cor.~5.10]{GKR-arxiv}.
    On the other hand, it is clear that the map ${\pi^{\ab}_1(X^o)}/n \to
    {\pi^{\ab}_1(X, D_\dagger)}/n$ is an isomorphism. The injectivity of
    ${\rho_{X^o}}/n$ now follows from \thmref{thm:Rec-dual-6}.

 We now let $n = p^m$ and $\sF = j_!(W_m\Omega^2_{X^o, \log})$. We have an exact sequence
    \begin{equation}\label{eqn:Rec-ker-8-0}
      {K_2(K)}/{p^m} \to {\underset{x \in X_{(0)}}\bigoplus}
      H^1_{x}(X, \sF) \to H^1_\et(X, \sF).
    \end{equation}
    If $x \in X^o$, then $H^1_\et(X^h_x, W_m\Omega^2_{X^h_x, \log}) =0$
    (cf. the proof of \propref{prop:Nis-et-ICG}). This implies
    $H^1_{x}(X, \sF) \cong H^1_{x, \nis}(X, \sF) \cong
    H^1_{x, \nis}(X^o, W_m\Omega^2_{X^o, \log}) \cong {k(x)^\times}/{p^m}$.
    If $x \in D_\dagger$, then we have an exact sequence
    \begin{equation}\label{eqn:Rec-ker-8-1}
      0 \to H^0_\et(X^h_x, \sF) \to H^0_\et(K^h_x, \sF) \to  
      H^1_{x, \et}(X, \sF) \to H^1_\et(X^h_x, \sF).
    \end{equation}
In the exact sequence (where $\iota_x \colon \Spec(k(x)) \inj X^h_x$ is the
    inclusion) of {\'e}tale  cohomology
    \[
      H^0(X^h_x, W_m\Omega^2_{X^h_x, \log}) \to
      H^0(k(x), \iota^*_x(W_m\Omega^2_{X^h_x, \log})) \to
      H^1(X^h_x, \sF) \to  H^1(X^h_x, W_m\Omega^2_{X^h_x, \log}),
    \]
    the first arrow from the left is bijective by the proper base change theorem
    and the last group was observed above to be zero.
    It follows that $H^1_\et(X^h_x, \sF) =0$.
Since $H^0_\et(X^h_x, \sF)$ is clearly zero and $H^0_\et(K^h_x, \sF) \cong
    H^0_\et(K^h_x,  W_m\Omega^2_{K^h_x, \log})$, we conclude that
    the canonical map ${K_2(K^h_x)}/{p^m} \to H^1_{x}(X, \sF)$ is an isomorphism.
 It follows that the left arrow in \eqref{eqn:Rec-ker-8-0}
      is the map ${K_2(K)}/{p^m} \xrightarrow{(\partial_{X^o}, \vartheta_{X^o})}
      {F_1(X^o)}/{p^m} \bigoplus {F_2(X^o)}/{p^m}$.

      Since 
${K_2(K)}/{p^m} \to {\underset{x \in X_{(0)}}\bigoplus}
      H^1_{x, \nis}(X, \sF) \to H^1_\nis(X, \sF) \to 0$
is exact and the Nisnevich analogue of ~\eqref{eqn:Rec-ker-8-1} implies that
${K_2(K^h_x)}/{p^m} \xrightarrow{\cong} H^1_{x, \nis}(X, \sF)$ for $x \in D_\dagger$,
we conclude that there is a canonical isomorphism
\begin{equation}\label{eqn:Rec-ker-8-4}
  H^1_\nis(X, \sF) \xrightarrow{\cong} {C(X^o)}/{p^m}.
  \end{equation}
In particular, we have an inclusion ${C(X^o)}/{p^m} \inj H^1_\et(X, \sF)$.

We now let $\iota \colon D_\dagger \inj X$ be the inclusion and consider the
    commutative diagram of pairings in $D_\et(X)$:
 \begin{equation}\label{eqn:Rec-ker-8-2}
   \xymatrix@C2pc{
     \iota_* \iota^*(W_m\Omega^2_{X, \log}) \times
     \iota_* \iota^{!}(W_m\Omega^0_{X, \log}) \ar@<7ex>[d] \ar[r]^-{\cup}
     & W_m\Omega^2_{X, \log} \ar@{=}[d] \\
     W_m\Omega^2_{X, \log} \times   W_m\Omega^0_{X, \log}  \ar@<7ex>[u] \ar@<7ex>[d]
     \ar[r]^-{\cup} &
    W_m\Omega^2_{X, \log} \ar@{=}[d] \\
j_!(W_m\Omega^2_{X^o, \log}) \times {\bf R}j_*(W_m\Omega^0_{X^o, \log}) \ar@<7ex>[u]
\ar[r]^-{\cup} & W_m\Omega^2_{X, \log}.}
\end{equation}

Considering the induced pairings of {\'e}tale cohomology, we get a commutative
diagram of exact sequences
\begin{equation}\label{eqn:Rec-ker-8-3}
   \xymatrix@C.7pc{
     H^0(X, W_m\Omega^2_{X,\log}) \ar[r] \ar[d]_-{\cong} &
     {\underset{x\in D_\dagger}\bigoplus}
     \frac{K_2(A^h_x)}{p^m} \ar[r]^-{\partial_1}
     \ar[d] & H^1(X, \sF) \ar[r] \ar[d]^-{\rho'_{X^o}} &
     H^1(X, W_m\Omega^2_{X,\log}) \ar[d]^-{\cong} \\
     H^2(X, {\Z}/{p^m})^\vee \ar[r] & {\underset{x\in D_\dagger}\bigoplus}
     H^2_x(X,  {\Z}/{p^m})^\vee \ar[r]^-{\partial_2} &
     H^1(X^o, {\Z}/{p^m})^\vee  \ar[r] &
     H^1(X, {\Z}/{p^m})^\vee.}
   \end{equation}

 It is clear from the pairing in the top row of ~\eqref{eqn:Rec-ker-8-2} that
   the second vertical arrow from the left in the above diagram is the direct
   sum of $\rho_{A^h_x}$. Furthermore, the
   commutativity of ~\eqref{eqn:Rec-ker-8-2} and the definition of
   the right arrow $\partial'_{X|D}$ in \lemref{lem:ICG-PC} imply that
   $\partial_1$ coincides with the restriction of the canonical map
   ${F_2(X^o)}/{p^m} \to {C(X^o)}/{p^m}$ in the definition of $C(X^o)$.
   Similarly, $\partial_2$ coincides with the sum of
   the canonical maps ${G^{(0)}_{K^h_x}/{p^m}} \to {G_K}/{p^m} \to
   {\pi^{\ab}_1(X^o)}/{p^m}$. It follows that the composition
   ${C(X^o)}/{p^m} \inj H^1(X, \sF) \xrightarrow{\rho'_{X^o}}
   {\pi^{\ab}_1(X^o)}/{p^m}$ is
   the reciprocity map $\rho^m_{X^o}$. We thus have to show that $\rho'_{X^o}$ is
   injective.

  Now, we note that the vertical arrows on the left and the right ends of
   ~\eqref{eqn:Rec-ker-8-3} are isomorphisms by \cite[Thm.~4.7]{KRS}.
   A diagrams chase shows that it suffices to show that for every $x \in D_\dagger$,
   the reciprocity map
   $\rho_{A^h_x} \colon \frac{K_2(A^h_x)}{p^m} \to {G^{(0)}_{K^h_x}}/{p^m}$ is injective.
   But this follows by \corref{cor:Local-inj-2}.
The proof concludes.
\end{proof}

\begin{cor}\label{cor:Rec-ker-9}
The kernel of $\rho_{X^o} \colon C(X^o) \to \pi^{\ab}_1(X^o)$
  is the maximal divisible subgroup of $C(X^o)$.
\end{cor}
\begin{proof}
The argument for the prime-to-$p$ divisibility of $\Ker(\rho_{X^o})$ is identical
  to that of \cite[Thm.~4.6]{Hiranouchi-2} while the $p$-divisibility
  follows from \lemref{lem:No-p-tor} and \thmref{thm:Rec-ker-8}.
  \end{proof}

The following result proves part (3) of \thmref{thm:Main-3}.

\begin{prop}\label{prop:Rec-Image-open}
The image of $\rho_{X^o} \colon C(X^o)_0 \to \pi^{\ab}_1(X^o)_0$ 
is not necessarily finite.
\end{prop}
\begin{proof}
  Using \thmref{thm:Rec-ker-8} in place of \thmref{thm:Rec-dual-6}
  and repeating the proof of \propref{prop:Rec-Image}, it suffices to show
  that $\Ker(\theta_{X^o} \colon {C(X^o)}/p \to {C(X)}/p)$ may not be finite.
Taking $X = \P^1_k$ and $D_\dagger = \{\infty\}$, we see from 
  ~\eqref{eqn:Rec-ker-8-4} that there is an exact sequence
  \begin{equation}\label{eqn:Rec-Image-open-0}
    0 \to H^0_\nis(X, \Omega^2_{X, \log}) \to \Omega^2_{A^h_\infty, \log} \to
    \Ker(\theta_{X^o}) \to 0.
  \end{equation}
  As in the proof of \propref{prop:Rec-Image}, one reduces to showing that
  ${K_2(A^h_\infty)}/p \cong \Omega^2_{A^h_\infty, \log}$ is infinite. But this follows
  from \cite[\S~2, Prop.~1]{Kato80-1}, \lemref{lem:Kato-1} and
  \lemref{lem:Kato-equiv}, which together imply
that ${K_2(A^h_\infty)}/p$ has an  infinite decreasing filtration 
whose graded quotients are infinite-dimensional $\F_p$-vector spaces.
\end{proof}

\subsection{Reciprocity for $\wh{C}(X^o)$}
\label{sec:Rec-open-0}
We observed in \S~\ref{sec:RMap} that $\rho_{X^o}$ factors through the continuous
homomorphism $\wh{\rho}_{X^o} \colon \wh{C}(X^o) \to \pi^{\ab}_1(X^o)$,
where $\wh{\rho}_{X^o} = {\varprojlim}_\n \rho_{X|D_\n}$.
We shall now prove \thmref{thm:Main-5}.
First, we prove some lemmas.

\begin{lem}\label{lem:Rec-lim-0}
  The pro-abelian group $\{_{p^m} C(X, D_\n)\}_{\n}$ is zero for every integer
  $m \ge 1$.
\end{lem}
\begin{proof}
  By an induction argument, it suffices to prove the lemma for $m =1$.
  By ~\eqref{eqn:Nis-et-ICG-0}, we have an exact sequence
  \begin{equation}\label{eqn:Rec-lim-0-0}
    G^0_\n(1) \to {K_2(K)}/{p} \to
    {F_1(X^o)}/p \bigoplus {F_2(X|D_\n)}/p \to {C(X, D_\n)}/p \to 0.
  \end{equation}
  By \lemref{lem:Top-5} and \corref{cor:Duality-Spl}, we have
  $G^0_\n(1) \inj G^0_\n(1)^\pf \cong F^2_\n(1)$ for $\n \gg 0$.
  On the other hand, we observed in the proof of \thmref{thm:Duality-M} that
  $F^2_\n(1) = 0$ for $\n \gg 0$. We conclude that there is an exact sequence
  of pro-abelian groups
   \begin{equation}\label{eqn:Rec-lim-0-1}
0 \to {K_2(K)}/{p} \to
    {F_1(X^o)}/p \bigoplus \{{F_2(X|D_\n)}/p\}_\n \to \{{C(X, D_\n)}/p\}_n \to 0.
  \end{equation}
  Together with \lemref{lem:Fil-M-0}, this implies that 
  in the sequence of canonical maps of pro-abelian groups
  $\bigoplus_{x \in D_\dagger} \{_{p} (K_2({k(x)[t]}/{(t^{n_x})}))\}_{\n}
  \to \bigoplus_{x \in D_\dagger}  \{_{p} ({K_2(K^h_x))}/{\Fil_{n_x} K_2(K^h_x)})\}_{\n} \to
  \{_{p} C(X, D_\n)\}_{\n}$, the first arrow is an isomorphism and
  the second arrow is levelwise surjective.
  It suffices therefore to show that $\{_{p}K_2(A_n, (t))\}_n = 0$ if
  $A_n = {k'[t]}/{(t^n)}$ for a local field $k'$ of characteristic $p$.
  But this follows from \cite[Lem.~5.1, 5.7]{Gupta-Krishna-AKT}.
  \end{proof}

  \begin{lem}\label{lem:Rec-lim-1}
    For $m \ge 1$, there is an exact sequence
    \begin{equation}\label{eqn:Rec-lim-1-1}
   0 \to {\wh{C}(X^o)} \xrightarrow{p^m} {\wh{C}(X^o)} \to
   {\varprojlim}_\n {C_m(X,D_\n)} \to 0.
  \end{equation}
   \end{lem}
   \begin{proof}
     This follows by taking the limits in the exact sequence of
  pro-abelian groups
  \begin{equation}\label{eqn:Rec-lim-1-0}
    0 \to \{_{p^m} C(X, D_\n)\} \to \{C(X, D_\n)\} \xrightarrow{p^m}  \{C(X, D_\n)\} \to
    \{C_m(X,D_\n)\} \to 0
  \end{equation}
  and using \lemref{lem:Rec-lim-0}
\end{proof}

\begin{lem}\label{lem:Rec-lim-2}
  For a prime-to-$p$ integer $q$, the canonical map
  ${\wh{C}(X^o)}/q \to \varprojlim_\n {C(X,D_\n)}/q$ is bijective.
\end{lem}
\begin{proof}
  Since the map ${C(X^o)}/q \to {C(X,D_\n)}/q$ is bijective for every $\n \ge 1$,
  it suffices to show that the map ${C(X^o)}/q \to {\wh{C}(X^o)}/q$ is surjective.
  For this, it further suffices to show that the canonical map
  ${\varprojlim_\n F_2(X|D_\n)} \to {F_2(X|D_\dagger)}$ is injective modulo $q$.
  But this follows by fixing $x \in D_\dagger$,
  taking limits in the exact sequence of pro-abelian groups
\begin{equation}\label{eqn:Rec-lim-2-0}
  0 \to \left\{\frac{\Fil_1 K_2(K^h_x)}{\Fil_n K_2(K^h_x)}\right\} \to
    \left\{\frac{K_2(K^h_x)}{\Fil_n K_2(K^h_x)}\right\} \to 
      \left\{\frac{K_2(K^h_x)}{\Fil_1 K_2(K^h_x)}\right\} \to 0,
      \end{equation}
     and noting that the pro-abelian group on the left is uniquely $q$-divisible.
\end{proof}

\begin{lem}\label{lem:Rec-ker-2}
The transition map $\Ker(\rho_{X|D_{n'}}) \to \Ker(\rho_{X|D_\n})$ is surjective for
all $\n' \ge \n \ge 1$.
\end{lem}
\begin{proof}
  We write $\rho_\n = \rho_{X|D_\n}, \ A_\n = \Ker(\rho_{X|D_\n})$ and look at the
  commutative diagram
\begin{equation}\label{eqn:Rec-ker-2-0}
  \xymatrix@C2pc{
    0 \ar[r] & A_{\n'} \ar[r] \ar[d] & C(X, D_{\n'})_0 \ar[r]^-{\rho^0_{\n'}}
    \ar@{->>}[d]^-{\delta^{\n'}_\n} &
    \pi^{\ab}_1(X,D_{\n'})_0 \ar@{->>}[d]^-{\lambda^{\n'}_\n} \\
  0 \ar[r] & A_{\n} \ar[r] & C(X, D_{\n})_0 \ar[r]^-{\rho^0_{\n}} &
  \pi^{\ab}_1(X,D_{\n})_0.}
\end{equation}
We let $x \in \Ker(\rho_{\n})$ and let $M$ be an integer which annihilates
the image of $\rho^0_{\n'}$ by \lemref{lem:Rec-ker-0}.
By \thmref{thm:Rec-ker-1}, we can write $x = Mx'$ for some
$x' \in  \Ker(\rho_{\n})$. On the other hand, we can write
$x' = \delta^{\n'}_\n(y')$ for some $y' \in C(X, D_{\n'})_0$. Since
$y := My' \in \Ker(\rho_{\n'})$ and $ \delta^{\n'}_\n(y) = x$, we are done.
\end{proof}

We can now prove \thmref{thm:Main-5}.

\begin{thm}\label{thm:Rec-lim-3}
For every integer $n \ge 1$, the reciprocity homomorphism
  \[
    \wh{\rho}_{X^o} \colon   {\wh{C}(X^o)}/n \to  {\pi^{\ab}_1(X^o)}/n
\]
is injective. In particular, $\Ker(\wh{\rho}_{X^o})$ is the maximal divisible
subgroup of $\wh{C}(X^o)$.
\end{thm}
\begin{proof}
  The first part of the theorem follows by combining
  Lemmas~\ref{lem:Rec-lim-1}, ~\ref{lem:Rec-lim-2} and \thmref{thm:Rec-dual-6}.
  The $p$-divisibility of
  $\Ker(\wh{\rho}_{X^o})$ follows by combining the first part of the theorem with
Lemmas~\ref{lem:No-p-tor} and ~\ref{lem:Rec-lim-1}. If $p \nmid n$, 
  we let $B_\n = \Ker(A_\n \surj A_{\un{1}})$ and look at the commutative diagram
\begin{equation}\label{eqn:Rec-ker-7-2}
\xymatrix@C2pc{
0 \to {\varprojlim} B_\n \ar[r] \ar[d]_-{n} & {\varprojlim} A_\n \ar[r] \ar[d]^-{n}
& A_{\un{1}} \ar[r] \ar[d]^-{n} & 0 \\
0 \to {\varprojlim} B_\n \ar[r] & {\varprojlim} A_\n \ar[r] 
& A_{\un{1}} \ar[r]  & 0.}
\end{equation}
\lemref{lem:Rec-ker-2} implies that the rows of this diagram are exact.
Since each $B_\n$ is a $p$-primary group, the left vertical arrow is an
isomorphism. Since the right vertical arrow is surjective by \thmref{thm:Rec-ker-1},
we conclude that the middle vertical arrow is surjective. 
\end{proof}

\section{Dual of the reciprocity maps}\label{sec:Dual-R}
We shall maintain the notations of \S~\ref{sec:CFT-mod}.
We shall make an additional assumption that $X$ is geometrically connected over $k$.
We let $r$ denote the `rank of $X$ over $k$'
(cf. \cite[Defn.~2.5, Thm.~6.2]{Saito-JNT}).
Our goal is to prove the last parts of Theorems~\ref{thm:Main-2} and
~\ref{thm:Main-3}. The first result is the following.

\begin{thm}\label{thm:Dual-R-1}
  We have the following.
  \begin{enumerate}
  \item
    The canonical maps $\wh{C}(X^o)^\star_\fr \to C^c(X^o)^\star_\fr \to C(X^o)^\star_\fr$
    are isomorphisms.
    \item
The reciprocity map for $X^o$ induces an exact sequence
  \begin{equation}\label{eqn:Dual-R-1-*}
    0 \to ({\Q}/{\Z})^r \to \pi^{\ab}_1(X^o)^\star \xrightarrow{\rho^\vee_{X^o}}
    {C}(X^o)^\star_\fr \to 0.
   \end{equation} 
    \end{enumerate}
    \end{thm}
  \begin{proof}
    The injectivity of maps in (1) follows from the fact that the maps
    $C(X^o) \inj C^c(X^o) \to \wh{C}(X^o)$ have dense images. The 
surjectivity follows from (2).

To prove (2), we look at the commutative diagram
\begin{equation}\label{eqn:Dual-R-1-0}
  \xymatrix@C1pc{
    0 \ar[r] & H^1_\et(X, {\Z}/{p^m}) \ar[r] \ar[d]_-{(\rho^m_X)^\vee} &
    H^1_\et(X^o, {\Z}/{p^m}) \ar[r] \ar[d]^-{(\rho^m_{X^o})^\vee} &
    {\underset{x \in D_\dagger}\bigoplus} H^2_x(A^h_x, {\Z}/{p^m}) \ar[r]
    \ar[d]^-{(\rho_{D_\dagger})^\vee} & H^2_\et(X, {\Z}/{p^m})
    \ar[d]^-{(\theta^*)^\vee} \\
    0 \ar[r] & ({C(X)}/{p^m})^\star \ar[r] & ({C(X^o)}/{p^m})^\star \ar[r] &
    {\underset{x \in D_\dagger}\bigoplus} ({K_2(A^h_x)}/{p^m})^\star \ar[r] &
      G^0_0(m)^\vee.}
  \end{equation}
  The top row in this diagram is exact and the bottom row is a complex which is
  exact except possibly at the third term because of the exact sequence of
  continuous maps
  \[
 {\underset{x \in D_\dagger}\bigoplus} {K_2(A^h_x)}/{p^m} \to {C(X^o)}/{p^m} \to
    {C(X)}/{p^m} \to 0.
  \]

 The vertical arrow on the right end of ~\eqref{eqn:Dual-R-1-0} is injective by
  \cite[\S~3, Prop.~4]{Kato-Saito-Ann} (see also \cite[Thm.~4.7]{KRS}).
  The arrow $(\rho_{D_\dagger})^\vee$ is bijective by \cite[\S~8, Thm.~2]{Kato80-1} and
  \cite[\S~3, Rem.~4]{Kato80-2}. 
A diagram chase implies that
$(\rho^m_X)^\vee$ and $(\rho^m_{X^o})^\vee$ share common kernels and
cokernels. We conclude from \cite[Cor.~5.1]{Hiranouchi-2} that there is an
exact sequence
\begin{equation}\label{eqn:Dual-R-1-1}
0 \to ({\Z}/{p^m})^r \to  ~_{p^m}(\pi^{\ab}_1(X^o)^\star) 
    \xrightarrow{(\rho_{X^o})^\vee}  ~_{p^m}(C(X^o))^\star) \to 0.
  \end{equation}

We now let $n$ be an integer prime to $p$ and let $\Lambda = {\Z}/n$.
  One checks from the definition of $C(X^o)$
  that there is an exact sequence
  \begin{equation}\label{eqn:Dual-R-0-3}
    {\underset{x \in D_\dagger}\bigoplus} {K_2(A^h_x)}/n \to {C(X^o)}/n \to
    {C(X)}/n \to 0.
  \end{equation}
  in which all arrows are continuous. By
  \cite[Lem.~8.4, 8.13, Prop.~8.15]{GKR-arxiv}, the first arrow in this sequence
  factors through ${\underset{x \in D_\dagger}\bigoplus} {K_2(k(x))}/n$ and gives
  rise to an exact sequence
  \begin{equation}\label{eqn:Dual-R-0-4}
    H^2(X, \Lambda(2)) \to {\underset{x \in D_\dagger}\bigoplus}
    H^2(k(x), \Lambda(2)) \to {C(X^o)}/n \to {C(X)}/n \to 0,
  \end{equation}
where $H^i(Y, \Lambda(j))$ denotes Voevodsky's motivic cohomology of $Y$.
The continuity of all maps in ~\eqref{eqn:Dual-R-0-3} is equivalent to that
the first two arrows in ~\eqref{eqn:Dual-R-0-4} are continuous.

We thus get a commutative diagram of exact sequences
\begin{equation}\label{eqn:Dual-R-0-5}
  \xymatrix@C.6pc{
    0 \ar[r] & H^1_\et(X, \Lambda) \ar[r] \ar[d]_-{\psi_X} &
    H^1_\et(X^o, \Lambda) \ar[r] \ar[d]_-{\psi_{X^o}} &
    {\underset{x \in D_\dagger}\bigoplus} H^0_\et(k(x), \Lambda(-1)) \ar[r]
    \ar[d]^-{\psi_{D_\dagger}} &
    H^2_\et(X, \Lambda) \ar[r] \ar[d]^-{\psi_X} \ar[r] & 0 \\
    0 \ar[r] & ({C(X)}/n)^\star \ar[r] & ({C(X^o)}/n)^\star \ar[r] &
    {\underset{x \in D_\dagger}\bigoplus} H^2(k(x), \Lambda(2))^\star \ar[r] &
    H^2(X, \Lambda(2))^\vee, &}
\end{equation}
in which all vertical arrows are the compositions of the Saito-Tate duality
maps (cf. \cite[Thm.~9.9]{GKR-arxiv}) with the dual of the {\'e}tale realization
maps between the Nisnevich and {\'e}tale motivic cohomology.

By \cite[Prop.~9.12]{GKR-arxiv}, the first two vertical arrows in
~\eqref{eqn:Dual-R-0-5} coincide with
the dual of the reciprocity maps $\rho^\vee_X$ and $\rho^\vee_{X^o}$, respectively.
The realization maps $H^2(Y, \Lambda(2)) \to H^2_\et(Y, \Lambda(2))$
are isomorphisms for $Y \in \{\Spec(k(x)), X\}$ by the Beilinson-Lichtenbaum
conjecture. In particular, the last two vertical arrows in ~\eqref{eqn:Dual-R-0-5}
are isomorphisms. A diagram chase implies that
$\rho^\vee_X$ and $\rho^\vee_{X^o}$ share common kernels and
cokernels. We conclude from \cite[Cor.~5.1]{Hiranouchi-2} that there is an
exact sequence
\begin{equation}\label{eqn:Dual-R-0-6}
0 \to  ({\Z}/n)^r \to  ~_{n}(\pi^{\ab}_1(X^o)^\star) 
    \xrightarrow{(\rho_{X^o})^\vee}  ~_{n}(C(X^o))^\star) \to 0.
  \end{equation}

 Combining ~\eqref{eqn:Dual-R-1-1} and ~\eqref{eqn:Dual-R-0-6} and taking the
  limit, we get ~\eqref{eqn:Dual-R-1-*}. This proves (2).
\end{proof}

\vskip .3cm

\begin{thm}\label{thm:Dual-R-0}
  For every effective divisor $D \subset X$, we have an
  exact sequence
  \begin{equation}\label{eqn:Dual-R-0-0}
    0 \to ({\Q}/{\Z})^r \to \pi^{\ab}_1(X,D)^\star \xrightarrow{\rho^\vee_{X|D}}
    C(X,D)^\star_\fr \to 0.
    \end{equation}
  \end{thm}
  \begin{proof}
We can assume $D = D_\n$ for some $\n \ge \un{1}$. 
We fix an integer $m \ge 1$.
It follows from \corref{cor:Top-9} and \lemref{lem:Rec-dual-5} that the realization
map $C_m(X, D_\n) \to C_m^\et(X,D_\n)$ is continuous.
By \propref{prop:KSL} and ~\eqref{eqn:Top-3} (whose rows are exact for each
$\n$), we therefore get a commutative diagram
\begin{equation}\label{eqn:Dual-R-0-1}
  \xymatrix@C.7pc{
    0 \ar[r] & H^1_\et(X, {\Z}/{p^m}) \ar[r] \ar[d]^-{(\rho^m_X)^\vee} &
    \Fil_{D_{\n}} H^1_\et(K, {\Z}/{p^m}) \ar[r] \ar[d]^-{(\rho^m_\n)^\vee} &
    \H^1_\et(X, W_m\sV^0_{\n_-}) \ar[r] \ar[d]^-{(\psi^*_\n)^\vee} &
    H^2_\et(X, {\Z}/{p^m}) \ar[d]^-{(\theta^*)^\vee} \\
      0 \ar[r] & C_m(X)^\star \ar[r] & C_m(X,D_\n)^\star \ar[r] &
      H^0_\nis(X, W_m\sW^2_\n)^\star \ar[r] & G^0_0(m)^\vee}
  \end{equation}
  in which the top row is exact and the bottom row is a complex which is exact
  except possibly at $ H^0_\nis(X, W_m\sW^2_\n)^\star$.

The vertical arrow on the right end is injective by
  \cite[\S~3, Prop.~4]{Kato-Saito-Ann} (see also \cite[Thm.~4.7]{KRS}). The map
  $(\psi^*_\n)^\vee$ is an isomorphism by \propref{prop:Dual-iso} and
  \lemref{lem:Coh-V2}. By \cite[Cor.~5.1]{Hiranouchi-2}, \corref{cor:p-tor-*} and a
  diagram chase, we thus get an exact sequence
  \begin{equation}\label{eqn:Dual-R-0-2}
    0 \to ({\Z}/{p^m})^r \to  ~_{p^m}(\pi^{\ab}_1(X,D_\n)^\star) 
    \xrightarrow{(\rho_{X|D_\n})^\vee}  ~_{p^m}(C(X,D_\n)^\star) \to 0.
  \end{equation}

 By ~\eqref{eqn:Dual-R-0-6} and \cite[Lem.~2.8, Cor.~5.10]{GKR-arxiv}, we
 have a similar exact sequence when we replace $p^m$ by any integer prime to $p$.
 Taking the limit over all integers, we get
  ~\eqref{eqn:Dual-R-0-0}.
\end{proof}

The following is immediate from Theorems~\ref{thm:Dual-R-1} and
  ~\ref{thm:Dual-R-0}.
  
  \begin{cor}\label{cor:Top-coker}
    For the reciprocity maps $\rho_{X^o}$ and $\rho_{X|D}$, we have
    \[
      \coker_\tp(\rho_{X^o}) \xrightarrow{\cong} \coker_\tp(\rho_{X|D})
      \xrightarrow{\cong} (\wh{\Z})^r.
    \]
  \end{cor}

  \vskip .3cm

\subsection{Reciprocity for $K$}\label{sec:Rec-K}
We shall now study the reciprocity for the function field $K = k(X)$.
  We let $C(K)$ be the limit of the pro-abelian group $\{C(X,D)\}$, where
  $D$ runs through all effective divisors of $X$. We endow $C(K)$ with the
  projective limit topology. It is clear that $C(K)$ is topologically
  isomorphic to ${\varprojlim}_U \wh{C}(U)$, where the latter has the projective
  limit topology as $U$ runs through open
  subsets of $X$.  Since $G_K \to {\varprojlim}_D \pi^{\ab}_1(X,D)$ is a
  topological isomorphism, we get a continuous reciprocity homomorphism
  $\rho_K \colon C(K) \to G_K$. 
  We prove the following result.

\begin{thm}\label{thm:Rec-K-0}
    The reciprocity map $\rho_K$ satisfies the following.
    \begin{enumerate}
    \item
      $\Ker(\rho_K)$ is the maximal $p$-divisible subgroup of $C(K)$.
    \item
      $\rho^m_K \colon {C(K)}/{p^m} \to {G_K}/{p^m}$ is injective for every
      $m \ge 1$.
   \item
      If $X$ is geometrically connected, there is an exact sequence
      \[
        0 \to ({\Q}/{\Z})^r \to (G_K)^\star \to C(K)^\star_\fr \to 0.
      \]
    \end{enumerate}
  \end{thm}
  \begin{proof}
    To prove (1), we use \lemref{lem:Rec-lim-1} to get an exact sequences of
    pro-abelian
    groups
     \begin{equation}\label{eqn:Rec-K-0-2}
       0 \to \{\wh{C}(U)\} \xrightarrow{p^m}   \{\wh{C}(U)\} \to
       \{{\wh{C}(U)}/{p^m}\} \to 0.
     \end{equation}
     Taking the limit, we get an exact sequence
  \begin{equation}\label{eqn:Rec-K-0-3}   
    0 \to C(K)  \xrightarrow{p^m} C(K) \to {\varprojlim}_U \frac{\wh{C}(U)}{p^m}.
  \end{equation}
  % Why is the third map surjective. This requires study of derived lim for general pro-abelian groups.
  It follows that the map ${C(K)}/{p^m} \to {\varprojlim}_U {\wh{C}(U)}/{p^m}$
  is injective. We now apply \thmref{thm:Rec-lim-3} to conclude the proof of
 (2). The item  (1) follows from (2) and \lemref{lem:No-p-tor}.
 The item (3) follows by taking limit (over $U$) of the exact sequences
 ~\eqref{eqn:Dual-R-1-*} and applying \cite[Lem.~2.6]{KRS}.
\end{proof}

\vskip .4cm

\noindent\emph{Acknowledgements.}
The second author would like to thank IISc Bangalore for invitation and hospitality
to work on this project.

\end{document}